\documentclass[12pt]{amsart}

\usepackage{amscd,latexsym,amsthm,amsfonts,amssymb,amsmath,amsxtra}
\usepackage[mathscr]{eucal}
\usepackage{pictexwd,dcpic}
\usepackage[normalem]{ulem}
\usepackage{hyperref}
\renewcommand\theequation{\thesection.\arabic{equation}}

\newcommand{\BC}{{\mathbb {C}}}

\newcommand{\BF}{{\mathbb {F}}}
\newcommand{\BG}{{\mathbb {G}}}

\newcommand{\BN}{{\mathbb {N}}}

\newcommand{\BP}{{\mathbb {P}}}

\newcommand{\BR}{{\mathbb {R}}}

\newcommand{\BZ}{{\mathbb {Z}}}

\newcommand{\CC}{{\mathcal {C}}}

\newcommand{\CF}{{\mathcal {F}}}
\newcommand{\CG}{{\mathcal {G}}}

\newcommand{\CK}{{\mathcal {K}}}
\newcommand{\CL}{{\mathcal {L}}}

\newcommand{\CO}{{\mathcal {O}}}
\newcommand{\CP}{{\mathcal {P}}}
\newcommand{\CQ}{{\mathcal {Q}}}
\newcommand{\CR}{{\mathcal {R}}}

\newcommand{\CT}{{\mathcal {T}}}
\newcommand{\CU}{{\mathcal {U}}}
\newcommand{\CV}{{\mathcal {V}}}

\newcommand{\CX}{{\mathcal {X}}}

\newcommand{\CZ}{{\mathcal {Z}}}

\newcommand{\FB}{{\mathfrak {B}}}

\newcommand{\Fa}{{\mathfrak {a}}}

\newcommand{\Fg}{{\mathfrak {g}}}
\newcommand{\Fh}{{\mathfrak {h}}}

\newcommand{\Fk}{{\mathfrak {k}}}
\newcommand{\Fl}{{\mathfrak {l}}}
\newcommand{\Fm}{{\mathfrak {m}}}
\newcommand{\Fn}{{\mathfrak {n}}}
\newcommand{\Fo}{{\mathfrak {o}}}
\newcommand{\Fp}{{\mathfrak {p}}}
\newcommand{\Fq}{{\mathfrak {q}}}

\newcommand{\Fu}{{\mathfrak {u}}}

\newcommand{\GL}{{\mathrm{GL}}}

\newcommand{\Hom}{{\mathrm{Hom}}}

\newcommand{\tr}{{\mathrm{tr}}}

\newcommand{\vol}{{\mathrm{vol}}}

\newcommand{\ol}{\overline}

\newcommand{\zg}{Z_G(F)\backslash G(F)}
\newcommand{\zh}{Z_H(F)\backslash H(F)}
\newcommand{\back}{\backslash}

\newcommand{\hc}{\Xi^{H\backslash G}}
\newcommand{\nor}{\sigma_{H\backslash G}}

\newtheorem{thm}{Theorem}[section]
\newtheorem{cor}[thm]{Corollary}
\newtheorem{lem}[thm]{Lemma}
\newtheorem{claim}[thm]{Claim}
\newtheorem{prop}[thm]{Proposition}
\newtheorem {conj}[thm]{Conjecture}

\newtheorem {ques/conj}[thm]{Question/Conjecture}

\newtheorem{defn}[thm]{Definition}
\newtheorem{rmk}[thm]{Remark}

\makeatletter

\newcommand{\Rmnum}[1]{\expandafter\@slowromancap\romannumeral #1@}
\makeatother

\begin{document}
\renewcommand{\theequation}{\arabic{equation}}
\numberwithin{equation}{section}

\title[On the Ginzburg-Rallis models]{Multiplicity One Theorem for the Ginzburg-Rallis Model: the tempered case.}

\author{Chen Wan}
\address{School of Mathematics\\
University of Minnesota\\
Minneapolis, MN 55455, USA}
\email{wanxx123@umn.edu}

\date{\today}

\begin{abstract}
Following the method developed by Waldspurger and Beuzart-Plessis in their proof of the local Gan-Gross-Prasad conjecture, we are able to prove the multiplicity one theorem for the Ginzburg-Rallis model over the Vogan packets in the tempered case. In some cases, we can also relate the multiplicity to the epsilon factor. This is a sequel of our work \cite{Wan15} in which we consider the supercuspidal case.
\end{abstract}

\maketitle

\section{Introduction and Main Result}
\subsection{Main results}
This paper is a continuation of \cite{Wan15}. For an overview of the Ginzburg-Rallis model, see Section 1 of \cite{Wan15}. We recall from there
the definition of the Ginzburg-Rallis models and the conjecture.

Let $F$ be a local field (p-adic field or real field), $D$ be the unique quaternion algebra over $F$. Take $P=P_{2,2,2}=MU$ be the standard parabolic subgroup of $G=\GL_6(F)$ whose Levi part $M$ is isomorphic to $\GL_2\times \GL_2\times \GL_2$, and whose unipotent radical $U$ consists of elements of the form
\begin{equation}\label{unipotent}
u=u(X,Y,Z):=\begin{pmatrix} I_2 & X & Z \\ 0 & I_2 & Y \\ 0 & 0 & I_2 \end{pmatrix}.
\end{equation}
We define a character $\xi$ on $U$ by
\begin{equation}\label{character}
\xi(u(X,Y,Z)):=\psi(\tr(X)+\tr(Y))
\end{equation}
where $\psi$ is a non-trivial additive character on $F$. It's clear that the stabilizer of $\xi$ is the diagonal embedding of $\GL_2$ into $M$, which is denoted by $H_0$. For a given character $\chi$ of $F^{\times}$, one induces a one dimensional representation $\omega$ of $H_0(F)$ given by $\omega(h):=\chi(\det(h))$. Now the character $\xi$ can be extended to the semi-direct product
\begin{equation}
H:=H_0\ltimes U
\end{equation}
by making it trivial on $H_0$. Similarly we can extend the character $\omega$ to $H$. It follows that the one dimensional representation $\omega\otimes \xi$ of $H(F)$ is well defined. The pair $(G,H)$ is the Ginzburg-Rallis model introduced by D. Ginzburg and S. Rallis in their paper \cite{GR00}. Let $\pi$ be an irreducible admissible representation of $G(F)$ with central character $\chi^2$, we are interested in the Hom space
$Hom_{H(F)}(\pi,\omega\otimes \xi)$, the dimension of which is denoted by $m(\pi)$ and is called be the multiplicity.

On the other hand, define $G_D=\GL_3(D)$, similarly we can define $U_D,\;H_{0,D}$ and $H_D$. We also define the character $\omega_D\otimes \xi_D$ on $H_D(F)$ in the same way except that the trace in the definition of $\xi$ is replaced by the reduced trace of the quaternion algebra $D$ and the determinant in the definition of $\omega$ is replaced by the reduced norm of the quaternion algebra $D$. Then for an irreducible admissible representation $\pi_D$ of $G_D(F)$ with central character $\chi^2$, we can also talk about the Hom space
$Hom_{H_D(F)}(\pi_D,\omega_D\otimes \xi_D)$, whose dimension is denoted by $m(\pi_D)$.

The purpose of this paper is to study the multiplicity $m(\pi)$ and $m(\pi_D)$. First, it was proved by C.-F. Nien in \cite{N06} over a $p$-adic local field, and by D. Jiang, B. Sun and C. Zhu in \cite{JSZ11} for an archimedean local field that
both multiplicities are less or equal to 1: $m(\pi),\;m(\pi_D)\leq 1.$
In other word, the pairs $(G,H)$ and $(G_D,H_D)$ are Gelfand pairs. In this paper,
we are interested in the relation between $m(\pi)$ and $m(\pi_D)$ under the local Jacquet-Langlands correspondence established in \cite{DKV84}. The local conjecture has been expected since the work of \cite{GR00}, and was first discussed in details by Jiang in his paper \cite{J08}.

\begin{conj}[Jiang,\cite{J08}]\label{jiang}
For any irreducible admissible representation $\pi$ of $GL_6(F)$, let $\pi_D$ be the local Jacquet-Langlands correspondence of $\pi$ to $GL_3(D)$ if it exists, and zero otherwise. We still assume that the central character of $\pi$ is $\chi^2$. Then the following identity
\begin{equation}\label{equation 1}
m(\pi)+m(\pi_D)=1
\end{equation}
holds for all irreducible generic representation $\pi$ of $GL_6(F)$.
\end{conj}

Note that the assertion in Conjecture \ref{jiang} can be formulated in terms of Vogan packets and pure inner forms of $PGL_6$. We refer to
\cite{Wan15} for discussion.

In the previous paper \cite{Wan15}, we prove Conjecture \ref{jiang} for the case that $F$ is a $p$-adic local field and $\pi$ is an irreducible supercuspidal representation of $GL_6(F)$. In this paper, we prove Conjecture \ref{jiang} for the case that $\pi$ is an irreducible tempered representation of $GL_6(F)$ and $F$ can be either p-adic or $\BR$.

\begin{thm}\label{main}
Let $F$ be a p-adic field or $F=\BR$, for any irreducible tempered representation $\pi$ of $\GL_6(F)$, Conjecture \ref{jiang} holds.
\end{thm}

Our proof of Theorem \ref{main} uses Waldspurger's method in his proof of the local Gan-Gross-Prosad conjecture (orthogonal case) in \cite{W10} and
\cite{W12}; and also some techniques introduced by Beuzart-Plessis in his proof of the local Gan-Gross-Prosad conjecture (unitary case) in \cite{B12} and \cite{B15}. In the p-adic case, the key ingredient of the proof is a local relative trace formula for the Ginzburg-Rallis model, which will be
called the trace formula in this paper for simplicity, unless otherwise specified.
To be specific, let $f\in C_{c}^{\infty}(\zg,\chi^{-2})$ be a strongly cuspidal function (see Section 3.3 for the definition of strongly cuspidal functions), and define the function $I(f,\cdot)$ on $H(F)\back G(F)$ to be
$$
I(f,x)=\int_{\zh} f(x^{-1}hx)\xi(h)\omega(h) dh.
$$
Then define
\begin{equation}\label{1.2}
I(f)=\int_{H(F)\backslash G(F)} I(f,g)  dg.
\end{equation}
We will prove in Section 7.1 that the integral defining $I(f)$ is absolutely convergent. The distribution in the trace formula is just $I(f)$.

Now we define the spectral and geometric sides of the trace formula. To each strongly cuspidal function $f\in C_{c}^{\infty}(\zg,\chi^{-2})$, one can associate a distribution $\theta_f$ on $G(F)$ via the weighted orbital integral (see Section 3.4). It was proved in \cite{W10} that the distribution $\theta_f$ is a quasi-character in the sense that for every semisimple element $x\in G_{ss}(F)$, $\theta_f$ is a linear combination of the nilpotent orbital integrals of $\Fg_x$ near $x$. For each nilpotent orbit $\CO$ of $\Fg_x$, let $c_{\theta_f,\CO}(x)$ be the coefficient, it is called the germ of the distribution $\theta_f$. Let $\CT$ be a subset of subtorus of $H_0$ as defined in Section 7.2. For any $t\in T_{reg}(F)$ and $T\in \CT$, define $c_f(t)$ to be $c_{\theta_f,\CO_t}(t)$ where $\CO_t$ is the unique regular nilpotent orbit in $\Fg_t$. For detailed description of $\CO_t$, see Section 7.2. Then we define the geometric side of our trace formula to be
$$I_{geom}(f)=\sum_{T\in \CT} | W(H,T)|^{-1} \nu(T) \int_{Z_G(F)\backslash T(F)} c_f(t) D^H(t) \Delta(t) \chi(\det(t)) dt$$
where $D^H(t)$ is the Weyl determinant and $\Delta(t)$ is some normalized function as defined in Section 7.2. For the spectral side, define
$$I_{spec}(f)=\int_{Temp(G,\chi^2)} \theta_{f}(\pi) m(\bar{\pi}) d\pi$$
where $Temp(G,\chi^2)$ is the set of irreducible tempered representations of $G(F)=\GL_6(F)$ with central character $\chi^2$, $d\pi$ is the measure on $Temp(G,\chi^2)$ as defined in Section 2.8, and $\theta_f(\pi)$ is the weighted character as defined in Section 3.4. Then the trace formula we proved in this paper is just
\begin{equation}\label{1.3}
I_{spec}(f)=I(f)=I_{geom}(f).
\end{equation}
It is worth to mention that the geometric side of this trace formula has already been proved in the previous paper \cite{Wan15}. So in this paper, we will focus on the spectral side. The proof of the trace formula will be given in Section 7. Similarly, we can also have the quaternion version of the trace formula.

After proving the trace formula, we are going to prove a multiplicity formula for the Ginzburg-Rallis model:
\begin{equation}\label{equation 2}
m(\pi)=m_{geom}(\pi), \;  m(\pi_D)=m_{geom}(\pi_D).
\end{equation}
Here $m_{geom}(\pi)$ (resp. $m_{geom}(\pi_D)$) is defined in the same way as $I_{geom}(f)$ except replacing the distribution $\theta_f$ by the distribution character $\theta_{\pi}$ (resp. $\theta_{\pi_D}$) associated to the representation $\pi$ (resp. $\pi_D$). For the complete definition of the multiplicity formula, see Section 8. Once this formula is proved, we can use the relation between the distribution characters $\theta_{\pi}$ and $\theta_{\pi_D}$ under the local Jacquet-Langlands correspondence to cancel out all terms in the expression of $m_{geom}(\pi)+m_{geom}(\pi_D)$ except the term $c_{\theta_{\pi},\CO_{reg}}$, which is the germ at the identity element. Then the work of Rodier (\cite{Rod81}) shows that $c_{\theta_{\pi},\CO_{reg}}=0$ if $\pi$ is non-generic, and $c_{\theta_{\pi},\CO_{reg}}=1$ if $\pi$ is generic. Because all tempered representations of $GL_n(F)$ are generic, we get the following identity
\begin{equation}\label{gm=1}
m_{geom}(\pi)+m_{geom}(\pi_D)=1.
\end{equation}
And this proves Theorem \ref{main}. The proof of the multiplicity formula uses the local relative trace formula for the Ginzburg-Rallis model we proved in Section 7, together with the Plancherel formula and Arthur's local trace formula. For details, see Section 8.

In the archimedean case, although we can use the same method as in the p-adic case (like Beuzart-Plessis did in \cite{B15} for the GGP case), it is actually much easier. All we need to do is to show that the multiplicity is invariant under the parabolic induction, this will be done in Section 5 for both p-adic and archimedean case. Then since only $\GL_1(\BR)$ and $\GL_2(\BR)$ have discrete series, we can reduce the problem to the trilinear $\GL_2$ model which has been considered by D. Prasad in his thesis \cite{P90}. For details, see Section 6.

Another result of this paper is that in some cases, we can related the multiplicity to the central value of exterior cube epsilon factor $\epsilon(1/2,\pi,\wedge^3)$. The conjecture has also been expected since the work of \cite{GR00}, and was first mentioned in \cite{J08}.

\begin{conj}\label{epsilon}
With the same assumptions as in Conjecture \ref{jiang}, assume that the central character of $\pi$ is trivial, then we have
\begin{eqnarray*}
m(\pi)=1 &\iff& \epsilon(1/2,\pi,\wedge^3)=1,\\
m(\pi)=0 &\iff& \epsilon(1/2,\pi,\wedge^3)=-1.
\end{eqnarray*}
\end{conj}

In this paper, we always fix a Haar measure $dx$ on $F$ and an additive charcter $\psi$ such that the Haar measure is selfdual for Fourier transform with respect to $\psi$. We use such $dx$ and $\psi$ in the definition of the $\epsilon$ factor. For simplicity, we will write the epsilon factor as $\epsilon(s,\pi,\rho)$ instead of $\epsilon(s,\pi,\rho,dx,\psi)$.

\begin{rmk}
In Conjecture \ref{epsilon}, we do need the assumption that the central character of $\pi$ is trivial. Otherwise, the exterior cube of the Langlands parameter of $\pi$ will no longer be selfdual, and hence the value of the epsilon factor at $1/2$ may not be $\pm 1$.
\end{rmk}

Our result for Conjecture \ref{epsilon} is the following Theorem.
\begin{thm}\label{main 1}
Let $\pi$ be an irreducible tempered representation of $\GL_6(F)$ with trivial central character.
\begin{enumerate}
\item If $F=\BR$, Conjecture \ref{epsilon} holds.
\item If $F$ is p-adic, and if $\pi$ is not discrete series or the parabolic induction of some discrete series of $\GL_4(F)\times \GL_2(F)$, Conjecture \ref{epsilon} holds.
\end{enumerate}
\end{thm}
This will be proved in Section 6 for the archimedean case, and in Section 8 for the p-adic case. Our method is to study the behavior of the multiplicity and the epsilon factor under the parabolic induction, and this is why in the p-adic case we have more restrictions.

Finally, our method can also be applied to all reduced models of the Ginzburg-Rallis model coming from the parabolic induction. For some models such results are well know (like the trilinear $\GL_2$ model); but for many other models, as far as we know, such results never appear in literature. For details, see Appendix B.

\subsection{Organization of the paper and remarks on the proofs}
In Section 2, we introduce basic notation and convention of this paper, including the definition and some basic facts on weight orbital integrals, weighted character, intertwining operator and the Plancherel formula. In Section 3, we will study quasi-characters and strongly cuspidal functions. For Sections 2 and 3, we follow \cite{B15} closely and provide details for the current case as needed.

In Section 4, we study the analytic and geometric properties of the Ginzburg-Rallis model. In particular, we show that it is a wavefront spherical variety and has polynomial growth as a homogeneous space. This gives us the weak Cartan decomposition. Then by applying those results, we prove some estimations for various integrals which will be used in later sections. The proof of some estimations are similar to the GGP case in \cite{B15}, we only include the proof here for completion.

In Section 5, we study an explicit element in the Hom space coming from the (normalized) integration of the matrix coefficient. The goal is prove that the Hom space is nonzero if and only if the explicit element is nonzero. It is standard to prove such a statement
by using the Plancherel formula together with the fact that the nonvanishing property of the explicit element is invariant under parabolic induction and unramified twist. However, there are two main difficulties in the proof of such a result for the Ginzburg-Rallis models. First, unlike the Gan-Gross-Prasad case, we do have nontrivial center for the Ginzburg-Rallis model. As a result, for many parabolic subgroups of $\GL_6(F)$ (the one which don't have an analogy in the quaternion case, i.e. the one not of type $(6),\;(4,2)$ or $(2,2,2)$, we will call theses models "type II models"), it is not clear why the nonvanishing property of the explicit intertwining operator is invariant under the unramified twist. Instead, we show that for such parabolic subgroups, the explicit intertwining operator will always be nonzero. This is Theorem \ref{thm 2}, which will
be proved in the last section of this paper.
Another difficulty is that unlike the Gan-Gross-Prasad case, when we do parabolic induction, we don't always have the strongly tempered model (in the GGP case, one can always go up to the codimension one case which is strongly tempered, then run the parabolic induction process). As a result, in order to prove the nonvanishing property of the explicit intertwining operator is invariant under parabolic induction, it is not enough to just change the order of the integral. This is because if the model is not strongly tempered, the explicit intertwining operator is defined via the normalized integral, not the original integral. We will find a way to deal with this issue in Section 5, but we have to treat the p-adic case and the archimedean case separately. For details, see Section 5.3 and 5.4.

In Section 6, we prove our main Theorems for the archimedean case by reducing it to the trilinear $\GL_2$ model, and then applying the result of Prasad \cite{P90}.

Starting from Section 7, we only consider the case that $F$ is a p-adic field. In Section 7, we state our trace formula. The geometric expansion of the trace formula has already been proved in the previous paper \cite{Wan15}. By applying the results in Section 3, 4 and 5, we will prove the spectral side of the trace formula.

In Section 8, we prove the multiplicity formula by applying our trace formula. Using that, we prove our main Theorems for the p-adic case.

In Section 9, we prove the argument we left in Section 5 (i.e. Theorem \ref{thm 2}), which is to show that for all type II models, the explicit interwining operator is always nonzero. By some similar arguments as in Section 5, we only need to show that the Hom space for all type II models is nonzero. Our method is to prove the local relative trace formula for these models which will gives us a multiplicity formula. \textbf{The most important feature for those models is that all semisimple elements in the small group are split, i.e. there is no elliptic torus. As a result, the geometric side of the trace formula for those models only contains the germ at 1 (as in Appendix B of \cite{Wan15}).} But by the work of Rodier \cite{Rod81}, the germ is always 1 for generic representations. Hence we know that the Hom space is always nonzero, and this proves Theorem \ref{thm 2}.

There are two appendices in this paper. In Appendix A, we prove the weak Cartan decomposition for the Ginzburg-Rallis model when $F$ is p-adic. In Appendix B, we will state some similar results for the reduced models of the Ginzburg-Rallis models coming from parabolic induction. Since the proof of those results are similar to the Ginzburg-Rallis model case we considered in this paper, we will skip it here and we refer the readers to my thesis \cite{Wan17} for details of the proof.

\subsection{Acknowledgement}
I would like to thank my advisor Dihua Jiang for suggesting me thinking about this problem, providing practical and thought-provoking viewpoints that lead to solutions of the problem, and carefully reviewing the first draft of this paper. I would like to thank Y. Sakellaridis for helpful discussions of the spherical varieties.

\section{Preliminarities}
\subsection{Notation and conventions}
Let $F$ be a p-adic field or $\BR$. If $F$ is p-adic, we fix the algebraic closure $\ol{F}$. Let val$_F$ and $\mid \cdot \mid_F$ be the valuation and absolute value on $F$, $\Fo_F$ be the ring of integers of $F$, and $\BF_q$ be the residue field. We fix an uniformizer $\varpi_F$.

For every connected reductive algebraic group $G$ defined over $F$, let $A_G$ be the maximal split center of $G$ and $Z_G$ be the center of $G$.
We denote by $X(G)$ the group of $F$-rational characters of $G$. Define $\Fa_G=$Hom$(X(G),\BR)$, and let $\Fa_{G}^{\ast}=X(G)\otimes_{\BZ} \BR$ be the dual of $\Fa_G$. We define a homomorphism $H_G:G(F)\rightarrow \Fa_G$ by $H_G(g)(\chi)=\log(|\chi(g)|_F)$ for every $g\in G(F)$ and $\chi\in X(G)$. Let $\Fa_{G,F}$ (resp. $\tilde{\Fa}_{G,F}$) be the image of $G(F)$ (resp. $A_G(F)$) under $H_G$. In the real case, $\Fa_G=\Fa_{G,F}=\tilde{\Fa}_{G,F}$; in the p-adic case, $\Fa_{G,F}$ and $\tilde{\Fa}_{G,F}$ are lattices in $\Fa_G$. Let $\Fa_{G,F}^{\vee}=Hom(\Fa_{G,F},2\pi \BZ)$ and $\tilde{\Fa}_{G,F}^{\vee}= Hom(\tilde{\Fa}_{G,F}, 2\pi \BZ)$, note that both $\Fa_{G,F}^{\vee}$ and $\tilde{\Fa}_{G,F}^{\vee}$ are zero in the real case; and they are lattices in $\Fa_{G}^{\ast}$ in the p-adic case. Set $\Fa_{G,F}^{\ast}=\Fa_{G}^{\ast}/\Fa_{G,F}^{\vee}$, and we can identify $i\Fa_{G,F}^{\ast}$ with the group of unitary unramified characters of $G(F)$ by letting $\lambda(g)=e^{<\lambda,H_G(g)>},\;\lambda\in i\Fa_{G,F}^{\ast},\; g\in G(F)$. For a Levi subgroup $M$ of $G$, let $\Fa_{M,0}^{\ast}$ be the subset of elements in $\Fa_{M,F}^{\ast}$ whose restriction to $\tilde{\Fa}_{G,F}$ is zero, then we can identify $i\Fa_{M,0}^{\ast}$ with the group of unitary unramified characters of $M(F)$ which is trivial on $Z_G(F)$.

Denote by $\Fg$ the Lie algebra of $G$, with the adjoint action of $G$. Since the Ginzburg-Rallis model has non-trivial center, all our integration need to modulo the center. To simplify the notation, for any Lie algebra $\Fg$, denote by $\Fg_0$ the elements in $\Fg$ whose trace is zero, $\Fg_0$ can be viewed as the Lie algebra of the reductive group $G\cap SL$.

For a Levi subgroup $M$ of $G$, let $\CP(M)$ be the set of parabolic subgroups of $G$ whose Levi part is $M$, $\CL(M)$ be the set of Levi subgroups of $G$ containing $M$, and $\CF(M)$ be the set of parabolic subgroups of $G$ containing $M$. We have a natural decomposition $\Fa_M=\Fa_{M}^{G}\oplus \Fa_G$, denote by $proj_{M}^{G}$ and $proj_G$ the projections of $\Fa_M$ to each factors. The subspace $\Fa_{M}^{G}$ has a set of coroots $\check{\Sigma}_M$, and for each $P\in \CP(M)$, we can associate a positive chamber $\Fa_{P}^{+}\subset \Fa_M$ and a subset of simple coroots $\check{\Delta}_P\subset \check{\Sigma}_M$. For such $P=MU$, we can also define a function $H_P:G(F)\rightarrow \Fa_M$ by $H_P(g)=H_M(m_g)$ where $g=m_g u_g k_g$ is the Iwasawa decomposition of $g$. According to Harish-Chandra, we can define the height function $\Vert \cdot\Vert$ on $G(F)$, take values in $\BR_{\geq 1}$ and a log-norm $\sigma$ on $G(F)$ by $\sigma(g)=\sup(1,\log(\Vert g\Vert))$. We also define $\sigma_0(g)=\inf_{z\in Z_G(F)} \{\sigma(zg)\}$. Similarly, we can define the log-norm function on $\Fg(F)$ as follows: fix a basis $\{X_i\}$ of $\Fg(F)$ over $F$, for $X\in \Fg(F)$, let $\sigma(X)=\sup(1,\sup\{ \log(|a_i|_F) \})$, where $a_i$ is the $X_i$-coordinate of $X$.

For $x\in G$ (resp. $X\in \Fg$), let $Z_G(x)$(resp. $Z_G(X)$) be the centralizer of $x$ (resp. $X$) in $G$, and let $G_x$(resp. $G_X$) be the neutral component of $Z_G(x)$ (resp. $Z_G(X)$). Accordingly, let $\Fg_x$ (resp. $\Fg_X$) be the Lie algebra of $G_x$ (resp. $G_X$). For a function $f$ on $G(F)$ (resp. $\Fg(F)$), and $g\in G(F)$, let ${}^g f$ be the $g$-conjugation of $f$, i.e. ${}^g f(x)=f(g^{-1}xg)$ for $x\in G(F)$ (resp. ${}^g f(X)=f(g^{-1}Xg)$ for $X\in \Fg(F)$).

Denote by $G_{ss}(F)$ the set of semisimple elements in $G(F)$, and by $G_{reg}(F)$ the set of regular elements in $G(F)$. The Lie algebra versions
are denoted by $\Fg_{ss}(F)$ and $\Fg_{reg}(F)$, respectively. Now for $X\in G_{ss}(F)$, the operator $ad(x)-1$ is defined and invertible on $\Fg(F)/\Fg_x(F)$. We define
$$
D^G(x)=\mid \det((ad(x)-1)_{\mid \Fg(F)/\Fg_x(F)})\mid_F.
$$
Similarly for $X\in \Fg_{ss}(F)$, define
$$
D^G(X)=\mid \det((ad(X))_{\mid \Fg(F)/\Fg_X(F)})\mid_F.
$$
For any subset $\Gamma\subset G(F)$, define $\Gamma^G:=\{g^{-1}\gamma g\mid g\in G(F),\gamma\in \Gamma\}$. We say an invariant subset $\Omega$ of $G(F)$ is compact modulo conjugation if there exist a compact subset $\Gamma$ such that $\Omega\subset \Gamma^G$

For two complex valued functions $f$ and $g$ on a set $X$ with $g$ taking values in the positive real numbers, we write that
$
f(x)\ll g(x),
$
and say that $f$ is {\sl essentially bounded} by $g$, if there exists a constant $c>0$ such that for all $x\in X$, we have
$
| f(x)| \leq cg(x).
$
We say $f$ and $g$ are {\sl equivalent}, which is denoted by
$f(x)\sim g(x)$,
if $f$ is essentially bounded by $g$ and $g$ is essentially bounded by $f$.

\subsection{Measures}
Through this paper, we fix a non-trivial additive character $\psi: F\rightarrow \BC^{\times}$. If $G$ is a connected reductive group, we may fix a non-degenerate symmetric bilinear form $<\cdot,\cdot>$ on $\Fg(F)$ that is invariant under $G(F)$-conjugation. For any smooth compactly supported complex valued function $f\in C_{c}^{\infty}(\Fg(F))$, we can define its Fourier transform $f\rightarrow \hat{f}\in C_{c}^{\infty}(\Fg(F))$ to be
\begin{equation}\label{FT}
\hat{f}(X)=\int_{\Fg(F)} f(Y) \psi(<X,Y>) dY
\end{equation}
where $dY$ is the selfdual Haar measure on $\Fg(F)$ such that $\hat{\hat{f}}(X)=f(-X)$.

Let $Nil(\Fg)$ be the set of nilpotent orbits of $\Fg$. For $\CO\in Nil(\Fg)$ and $X\in \CO$, the bilinear form $(Y,Z)\rightarrow <X,[Y,Z]>$ on $\Fg(F)$ can be descent to a symplectic form on $\Fg(F)/\Fg_X(F)$. The nilpotent $\CO$  has naturally a structure of $F$-analytic symplectic variety, which yields a selfdual measure on $\CO$. This measure is invariant under the $G(F)$-conjugation.

A $G$-domain on $G(F)$ (resp. $\Fg(F)$) is an open subset of $G(F)$ (resp. $\Fg(F)$) invariant under the $G(F)$-conjugation. Let $P=MU$ be a parabolic subgroup of $G$, and $K$ be a open compact subgroup of $G$ which is hyperspecial with respect to $M$ in the p-adic case. Then we can fix Haar measures on $G(F), M(F), U(F), K$ such that for $f\in C_{c}^{\infty} (G(F))$, we have
\begin{equation}\label{measure}
\int_{G(F)} f(g) dg=\int_K \int_{U(F)} \int_{M(F)} f(muk) dm du dk
\end{equation}
Moreover, if $F$ is p-adic, we require $mes(K)=1$.

If $T$ is a subtorus of $G$, we can provide a measure on $T$ via the Iwasawa decomposition as above, denoted by $dt$. On the other hand, we can define another measure $d_c t$ on $T(F)$ as follows: If $T$ is split, we require the volume of the maximal compact subgroup of $T(F)$ is 1 under $d_c t$. In general, $d_c t$ is compatible with the measure $d_c t'$ defined on $A_T(F)$ and with the measure on $T(F)/A_T(F)$ of total volume 1. Then we have a constant number $\nu(T)$ such that $d_c t=\nu(T)dt$.
Finally, if $M$ is a Levi subgroup of $G$, we can define the Haar measure on $\Fa_{M}^{G}$ such that the quotient
$\Fa_{M}^{G}/proj_{M}^{G} (H_M(A_M(F)))$
is of volume 1.

\subsection{Induced representation and intertwining operator}
Given a parabolic subgroup $P=MU$ of $G$ and an admissible representation $(\tau,V_{\tau})$ of $M(F)$, let $(I_{P}^{G}(\tau),I_{P}^{G}(V_{\tau}))$ be the normalized parabolic induced representation: $I_{P}^{G}(V_{\tau})$ consist of smooth functions $e:G(F)\rightarrow V_{\tau}$ such that
$$e(mug)=\delta_P(m)^{1/2} \tau(m)e(g), \;m\in M(F),\; u\in U(F),\; g\in G(F).$$
And the $G(F)$ action is just the right translation.

For $\lambda\in \Fa_{M}^{\ast}\otimes_{\BR}\BC$, let $\tau_{\lambda}$ be the unramified twist of $\tau$, i.e. $\tau_{\lambda}(m)=\exp(\lambda(H_M(m))) \tau(m)$ and let $I_{P}^{G}(\tau_{\lambda})$ be the induced representation. By the Iwasawa decomposition, every function $e\in I_{P}^{G}(V_{\tau})$ is determined by its restriction on $K$, and that space is invariant under the unramified twist. i.e. for any $\lambda$, we can realize the representation $I_{P}^{G}(\tau)$ on the space $I_{K\cap P}^{K}(\tau_K)$ which is consisting of functions $e_K:K\rightarrow V_{\tau}$ such that
$$e(mug)=\delta_P(m)^{1/2} \tau(m)e(g), \;m\in M(F)\cap K,\; u\in U(F)\cap,\; g\in K.$$

If $\tau$ is unitary, so is $I_{P}^{G}(\tau)$, the inner product on $I_{P}^{G}(V_{\tau_K})$ can be realized as
$$(e,e')=\int_K (e'(k),e(k)) dk.$$
This is an invariant inner product under the representation $I_{P}^{G}(\tau_{\lambda})$ for all $\lambda\in i\Fa_{M}^{\ast}$.

Now we define the intertwining operator. For a Levi subgroup $M$ of $G$, $P,P'\in \CP(M)$, and $\lambda\in \Fa_{M}^{\ast}\otimes_{\BR} \BC$, define the intertwining operator $J_{P'|P}(\tau_{\lambda}):I_{P}^{G}(V_{\tau})\rightarrow I_{P'}^{G}(V_{\tau})$ to be
$$J_{P'|P}(\tau_{\lambda})(e)(g)=\int_{(U(F)\cap U'(F))\backslash U'(F)} e(ug)du.$$
In general, the integral above is not absolutely convergent. But it is absolutely convergent for $Re(\lambda)$ sufficiently large, and it is $G(F)$-invariant.
By restricting to $K$, we can view $J_{P'|P}(\tau_{\lambda})$ as a homomorphism from $I_{P}^{G}(V_{\tau_K})$ to $I_{P'}^{G}(V_{\tau_K})$. In general, $J_{P'|P}(\tau_{\lambda})$ can be meromorphically continued to a function on $\Fa_{M}^{\ast}\otimes_{\BR}\BC /i\Fa_{M,F}^{\vee}$.

If $\tau$ is irreducible, by Schur's lemma, the operator $J_{P|\bar{P}}(\tau_{\lambda})J_{\bar{P}|P}(\tau_{\lambda})$ is a scalar, let $j(\tau_{\lambda})$ be the scalar, this is independent of the choice of $P$. We can normalize the intertwining operator by a complex valued function $r_{P'|P}(\tau_{\lambda})$ such that the normalized intertwining operator
$$R_{P'|P}(\tau_{\lambda})=r_{P'|P}(\tau_{\lambda})^{-1}J_{P'|P}(\tau_{\lambda})$$
satisfies the condition of Theorem 2.1 of \cite{Ar89}. The key conditions are
\begin{enumerate}
\item For $P,P',P''\in \CP(M)$, $R_{P''|P'}(\tau_{\lambda})R_{P'|P}(\tau_{\lambda})=R_{P''|P}(\tau_{\lambda})$.
\item Suppose $\tau$ is tempered, for $\lambda\in i\Fa_{M,F}^{\ast}$, $R_{P'|P}(\tau_{\lambda})$ is holomorphic and unitary.
\item The normalized intertwining operator are compatible with the unramified twist and the parabolic induction.
\end{enumerate}

\subsection{$(G,M)$-families}
Let $M$ be a Levi subgroup of $G$, we recall the notion of $(G,M)$-family introduced by Arthur. A $(G,M)$-family is a family $(c_P)_{P\in \CP(M)}$ of smooth functions on $i\Fa_{M}^{\ast}$ taking values in a locally convex topological vector space $V$ such that for all adjacent parabolic subgroups $P,P'\in \CP(M)$, the functions $c_p$ and $c_{P'}$ coincide on the hyperplane supporting the wall that separates the positive chambers for $P$ and $P'$. For such a $(G,M)$-family, one can associate an element $c_M\in V$ (\cite[Page 37]{Ar81}). If $L\in \CL(M)$, for a given $(G,M)$-family, we can deduce a $(G,L)$-family. Denote by $c_L$ the element in $V$ associated to such $(G,L)$-family. If $Q=L_Q U_Q\in \CF(L)$, we can deduce a $(L_Q,L)$-family from the given $(G,M)$-family, the element in $V$ associated to which is denoted by $c_{L}^{Q}$.

If $(Y_P)_{P\in \CP(M)}$ is a family of elements in $\Fa_M$, we say it is a $(G,M)$-orthogonal set (resp. and positive) if the following condition holds: if $P,P'$ are two adjacent elements of $\CP(M)$, there exists a unique coroot $\check{\alpha}$ such that $\check{\alpha}\in \check{\Delta}_P$ and $-\check{\alpha}\in \check{\Delta}_{P'}$, we require that $Y_P-Y_{P'}\in \BR \check{\alpha}$ (resp. $Y_P-Y_{P'}\in \BR_{\geq 0} \check{\alpha}$). For $P\in \CP(M)$, define a function $c_P$ on $i\Fa_{M}^{\ast}$ by $c_P(\lambda)=e^{-\lambda(Y_P)}$. Suppose the family $(Y_P)_{P\in \CP(M)}$ is a $(G,M)$-orthogonal set. Then the family $(c_P)_{P\in \CP(M)}$ is a $(G,M)$-family. If the family $(Y_P)_{P\in \CP(M)}$ is positive, then the number $c_M$ associated to this $(G,M)$-family is just the volume of the convex hull in $\Fa_{M}^{G}$ generated by the set $\{Y_P\mid P\in \CP(M)\}$. If $L\in \CL(M)$, the $(G,L)$-family deduced from this $(G,M)$-family is the $(G,L)$-family associated to the $(G,L)$-orthogonal set $(Y_Q)_{Q\in \CP(L)}$ where $Y_Q=proj_L(Y_P)$ for some $P\in \CP(M)$ such that $P\subset Q$ (it is easy to see this is independent of the choice of $P$). Similarly, if $Q\in \CP(L)$, then the $(L,M)$-family deduced from this $(G,M)$-family is the $(L,M)$-family associated to the $(L,M)$-orthogonal set $(Y_{P'})_{P'\in \CP^L(M)}$ where $Y_{P'}=Y_P$ with $P$ being the unique element of $\CP(M)$ such that $P\subset Q$ and $P\cap L=P'$.

\subsection{Weighted orbital integrals}
If $M$ is a Levi subgroup of $G$ and $K$ is a hyperspecial open compact subgroup with respect to $M$. For $g\in G(F)$, the family $(H_P(g))_{P\in \CP(M)}$ is $(G,M)$-orthogonal and positive. Let $(v_P(g))_{P\in \CP(M)}$ be the $(G,M)$-family associated to it and $v_M(g)$ be the number associated to this $(G,M)$-family. Then $v_M(g)$ is just the volume of the convex hull in $\Fa_{M}^{G}$ generated by the set $\{H_P(g),\; P\in \CP(M)\}$. The function $g\rightarrow v_M(g)$ is obviously left $M(F)$-invariant and right $K$-invariant.

If $f\in C_{c}^{\infty}(G(F))$ and $x\in M(F)\cap G_{reg}(F)$, define the weight orbital integral to be
\begin{equation}\label{WOI}
J_M(x,f)=D^G(x)^{1/2} \int_{G_x(F)\backslash G(F)} f(g^{-1}xg)v_M(g)dg.
\end{equation}
Note the definition does depend on the choice of the hyperspecial open compact subgroup $K$. But we will see later that if $f$ is strongly cuspidal, this definition is independent of the choice of $K$.

\begin{lem}
With the notation as above, the following holds.
\begin{enumerate}
\item If $f\in C_{c}^{\infty}(G(F))$, the function $x\rightarrow J_M(x,f)$ defined on $M(F)\cap G_{reg}(F)$ is locally constant,
invariant under $M(F)$-conjugation and has a compact support modulo conjugation.
\item There exists an integer $k\geq 0$, such that for every $f\in C_{c}^{\infty}(G(F))$, there exists $c>0$ such that
$$
| J_M(x,f)| \leq c(1+| \log D^G(x) |)^k
$$
for every $x\in M(F)\cap G_{reg}(F)$.
\end{enumerate}
\end{lem}

\begin{proof}
See Lemma 2.3 of \cite{W10}.
\end{proof}

\subsection{Weighted characters}
Let $M$ be a Levi subgroup, and $\tau$ be a tempered representation of $M(F)$. For $P,P'\in \CP(M)$, we have defined the normalized intertwining operator $R_{P'|P}(\tau_{\lambda})$ for $\lambda\in i\Fa_{M}^{\ast}$. Fix $P$, for every $P'\in \CP(M)$, define the function $\CR_{P'}(\tau)$ on $i\Fa_{M}^{\ast}$ by
$$\CR_{P'}(\tau,\lambda)=R_{P'|P}(\tau)^{-1}R_{P'|P}(\tau_{\lambda}).$$
This function takes value in the space of endomorphisms of $I_{P\cap K}^{K}(\tau_K)$ (not necessarily commutes with the $G$-action), recall that this space is invariant under the unramified twist. By \cite{Ar81}, this is a $(G,M)$-family. Then for $L\in \CL(M)$ and $Q\in \CF(L)$, we can associate an operator $\CR_{L}^{Q}(\tau)$ to this $(G,M)$ family. We define the weighted character of $\tau$ to be the distribution $f\rightarrow J_{L}^{Q}(\tau,f)$ given by
$J_{L}^{Q}(\tau,f)=\tr(\CR_{L}^{Q}(\tau) I_{P}^{G}(\tau)(f))$
for every $f\in C_{c}^{\infty}(G(F))$. This is independent of the choice of $P$ but depends on $K$ and the way we normalized the intertwining operators. If $L=Q=G$, the distribution $J_{G}^{G}(\tau,f)$ is just $\theta_{\pi}$ for $\pi=I_{P}^{G}(\tau)$ where $\theta_{\pi}(f)=tr(\pi(f))$.

\subsection{Harish-Chandra-Schwartz space}
Let $P_{min}$ be a minimal parabolic subgroup of $G$ and let $K$ be a hyperspecial maximal compact subgroup of $G(F)$, then we have the Iwasawa decomposition $G(F)=P_{min}(F) K$. Consider the normalized induced representation
$$I_{P_{min}}^{G}(1):=\{ e\in C^{\infty}(G(F))\mid e(pg)=\delta_{P_{min}}(p)^{1/2} e(g) \; for \; all \; p\in P_{min}(F), g\in G(F)\}$$
and we equip the representation with the inner product
$$(e,e')=\int_{K} e(k)\bar{e'}(k) dk.$$
Let $e_K\in I_{P_{min}}^{G}(1)$ be the unique function such that $e_K(k)=1$ for all $k\in K$.
\begin{defn}
The Harish-Chandra function $\Xi^G$ is defined to be
$$\Xi^G(g)=(I_{P_{min}}^{G}(1)(g)e_K, e_K)$$
\end{defn}

\begin{rmk}
The function $\Xi^G$ depends on the various choices we made, but this doesn't matter since different choices give us equivalent functions and the function $\Xi^G$ will only be used in estimations.
\end{rmk}

The next proposition summarize some basic properties of the function $\Xi^G$, the proof can be found in \cite{W03}. Also see Proposition 1.5.1 of \cite{B15}.

\begin{prop}\label{h-c function}
\begin{enumerate}
\item Let
$$M_{min}^{+}=\{m\in M_{min}(F)\mid \mid \alpha(m)\mid \leq 1 \; for\; all \; \alpha\in \Psi(A_{M_{min}},P_{min})\}.$$
Here $\Psi(A_{M_{min}},P_{min})$ is the set of positive root associated to $P_{min}$. Then there exists $d>0$ such that
$$\delta_{P_{min}}(m)^{1/2} \ll \Xi^G(m) \ll \delta_{P_{min}}(m)^{1/2} \sigma_0(m)^d$$
for all $m\in M_{min}^{+}$.
\item There exist $d>0$ such that
$\Xi^G(g)\ll \delta_{P_{min}}(m_{P_{min}}(g))^{1/2}\sigma_0(g)^d $
for all $g\in G(F)$, here $m_{P_{min}}(g)$ is the $M_{min}$-part of $g$ under the Iwasawa decomposition $G=M_{min}U_{min}K$.
\item Let $P=MU$ be a parabolic subgroup containing $P_{min}$, then we have
$$\Xi^G(g)=\int_K \delta_{P}(m_P(kg))^{1/2}\Xi^M(m_P(kg)) dk $$
for all $g\in G(F)$, here $m_P(g)$ is the $M$-part of $g$ under the Iwasawa decomposition $G=MUK$.
\item Let $P=MU$ be a parabolic subgroup of $G$. Then for all $d>0$, there exist $d'>0$ such that
$$\delta_P(m)^{1/2}\int_{U(F)} \Xi^G(mu)\sigma_0(mu)^{-d'}du\ll \Xi^M(m)\sigma_0(m)^{-d}$$
for all $m\in M(F)$.
\item There exist $d>0$ such that
$\int_{G(F)} \Xi^G(g)^2 \sigma(g)^{-d}dg$
is convergent.
\item We have the equality
$$\int_K \Xi^G(g_1kg_2)dk=\Xi^G(g_1)\Xi^G(g_2)$$
for all $g_1,g_2\in G(F)$.
\end{enumerate}
\end{prop}

For $f\in C^{\infty}(G(F))$ and $d\in \BR$, let
$$p_d(f)=\sup_{g\in G(F)} \{ |f(g)|\Xi^G(g)^{-1}\sigma(g)^d\}$$
If $F$ is p-adic, we define the Harish-Chandra-Schwartz space to be
$$\CC(G(F))=\{f\in C^{\infty}(G(F))| p_d(f)<\infty,\forall d>0\}.$$
If $F=\BR$, for $u,v\in \CU(\Fg)$ and $d\in \BR$, let
$$p_{u,v,d}(f)=p_d(R(u)L(v)f)$$
where "R" stands for the right translation, "L" stands for the left translation and $\CU(\Fg)$ is the universal enveloping algebra. We define the Harish-Chandra-Schwartz space to be
$$\CC(G(F))=\{f\in C^{\infty}(G(F))| p_{u,v,d}(f)<\infty,\forall d>0,\; u,v\in \CU(\Fg)\}.$$

We also need the weak Harish-Chandra-Schwartz space $\CC^w(G(F))$. For $d>0$, let
$$\CC^{w}_{d}(G(F))=\{f\in C^{\infty}(G(F))| p_{-d}(f)<\infty\}$$
if $F$ is p-adic. And let
$$\CC^{w}_{d}(G(F))=\{f\in C^{\infty}(G(F))| p_{u,v,-d}(f)<\infty,\forall  u,v\in \CU(\Fg)\}$$
if $F=\BR$. Then the weak Harish-Chandra-Schwartz space is defined to be
$$\CC^w(G(F))=\cup_{d>0} \CC^{w}_{d}(G(F)).$$

Also we can define the Harish-Chandra-Schwartz space (resp. weak Harish-Chandra-Schwartz space) with given unitary central character $\chi$: let $\CC(G(F),\chi)$ (resp. $\CC^w(G(F),\chi)$) be the Mellin transform of the space $\CC(G(F))$ (resp. $\CC^w(G(F))$) with respect to $\chi$.

\subsection{The Harish-Chandra-Plancherel formula}
Since the Ginzburg-Rallis model has nontrivial center, we only introduce the Plancherel formula with given central character. We fix an unitary character $\chi$ of $Z_G(F)$. For every $M\in \CL(M_{min})$, fix an element $P\in \CP(M)$. Let $\Pi_2(M,\chi)$ be the set of discrete series of $M(F)$ whose central character agree with $\chi$ on $Z_G(F)$. Then $i\Fa_{M,0}^{\ast}$ acts on $\Pi_2(M,\chi)$ by the unramified twist, let $\{\Pi_2(M,\chi)\}$ be the set of orbits under this action. For every orbit $\CO$, and for a fixed $\tau\in \CO$, let $i\Fa_{\CO}^{\vee}$ be the set of $\lambda\in i\Fa_{M,0}^{\ast}$ such that the representation $\tau$ and $\tau_{\lambda}$ are equivalent, which is a finite set. For $\lambda\in i\Fa_{M,0}^{\ast}$, define the Plancherel measure to be
$$\mu(\tau_{\lambda})=j(\tau_{\lambda})^{-1}d(\tau)$$
where $d(\tau)$ is the formal degree of $\tau$, which is invariant under the unramified twist, and $j(\tau_{\lambda})$ is defined in Section 2.3. Then for $f\in \CC(G(F),\chi^{-1})$, the Harish-Chandra-Plancherel formula is
\begin{eqnarray*}
f(g)&=&\Sigma_{M\in \CL(M_{min})} |W^M||W^G|^{-1} \Sigma_{\CO\in \{\Pi_2(M,\chi)\}} |i\Fa_{\CO}^{\vee}|^{-1}\\
&&\int_{i\Fa_{M,0}^{\ast}} \mu(\tau_{\lambda}) trace(I_{P}^{G}(\tau_{\lambda})(g^{-1}) I_{P}^{G}(\tau_{\lambda})(f))d\lambda.
\end{eqnarray*}
The proof of the above formula can be found in \cite{W03} for the p-adic case, and in \cite{Ar75} for the real case.

To simplify our notation, let $Temp(G,\chi)$ be the union of $I_{P}^{G}(\tau)$ for $P=MN$, $M\in \CL(M_{min})$, $\tau\in \CO$ and $\CO\in \{\Pi_2(M,\chi)\}$. We define a Borel measure $d\pi$ on $Temp(G,\chi)$ such that
$$\int_{Temp(G,\chi)}\varphi(\pi)d\pi=\Sigma_{M\in \CL(M_{min})} |W^M||W^G|^{-1} \Sigma_{\CO\in \{\Pi_2(M,\chi)\}} |i\Fa_{\CO}^{\vee}|^{-1} \int_{i\Fa_{M,0}^{\ast}} \varphi(I_{P}^{G}(\tau_{\lambda})) d\lambda$$
for every compact supported function $\varphi$ on $Temp(G,\chi)$. Here by saying a function $\varphi$ is compactly supported on $Temp(G,\chi)$ we mean that it is supported on finitely many orbit $\CO$ and for every such orbit $\CO$, it is compactly supported. Note that the second condition is automatic if $F$ is p-adic. Then the Harish-Chandra-Plancherel formula above becomes
$$f(g)=\int_{Temp(G,\chi)} Trace(\pi(g^{-1}) \pi(f)) \mu(\pi) d\pi.$$

We also need the matrical Paley-Wiener Theorem. Let $C^{\infty}(Temp(G,\chi))$ be the space of functions $\pi\in Temp(G,\chi)\rightarrow T_{\pi}\in End(\pi)^{\infty}$ such that it is smooth on every orbits $\CO$ as functions from $\CO$ to $End(\pi)^{\infty}\simeq End(\pi_K)^{\infty}$. Now we define $\CC(Temp(G,\chi))$ to be a subspace of $C^{\infty}(Temp(G,\chi))$ consisting of those $T:\pi\rightarrow T_{\pi}$ such that
\begin{enumerate}
\item If $F$ is p-adic, $T$ is nonzero on finitely many orbits $\CO$.
\item If $F=\BR$, for all parabolic subgroup $P=MU$ and for all differential operator with constant coefficient $D$ on $i\Fa_{M}^{\ast}$, the function
$DT: \sigma\in \Pi_2(M,\chi)\rightarrow D(\lambda\rightarrow T_{I_{P}^{G}(\sigma_{\lambda})})$
satisfies $p_{D,u,v,k}(T)=\sup_{\sigma\in \Pi_2(M,\chi)} ||DT(\sigma)||_{u,v} N(\sigma)^k<\infty$ for all $u,v\in \CU(\Fk)$ and $k\in \BN$. Here $||DT(\sigma)||_{u,v}$ is the norm of the operator $\sigma(u)DT(\sigma)\sigma(v)$ and $N(\sigma)$ is the norm on the set of all tempered representations (See Section 2.2 of \cite{B15}).
\end{enumerate}
Then the matrical Paley-Wiener Theorem states that we have an isomorphism between $\CC(G,\chi^{-1})$ and $\CC(Temp(G,\chi))$ given by
$$f\in \CC(G,\chi^{-1})\rightarrow (\pi\in Temp(G,\chi)\rightarrow \pi(f)\in End(\pi)^{\infty})$$
and
$$T\in \CC(Temp(G,\chi))\rightarrow f_T(g)=\int_{Temp(G,\chi)} Trace(\pi(g^{-1}) T_{\pi}) \mu(\pi) d\pi.$$

\section{Strongly Cuspidal Functions and Quasi Characters}
\textbf{Throughout this section, assume that $F$ is p-adic}.
\subsection{Quasi-characters of $G(F)$}
If $\theta$ is a smooth function defined on $G_{reg}(F)$, invariant under $G(F)-$conjugation. We say it is a quasi-character on $G(F)$ if for every $x\in G_{ss}(F)$, there is a good neighborhood $\omega_x$ of $0$ in $\Fg_x(F)$, and for every $\CO\in Nil(\Fg_x)$, there exists $c_{\theta,\CO}(x)\in \BC$ such that
\begin{equation}\label{germ 2}
\theta(x\exp(X))=\Sigma_{\CO\in Nil(\Fg_x)} c_{\theta,\CO}(x) \hat{j}(\CO,X)
\end{equation}
for every $X\in \omega_{x,reg}$. Here $\hat{j}(\CO,X)$ is the function on $\Fg_{reg}(F)$ represent the Fourier transform of the nilpotent orbital integral. For the definition of good neighborhood, see Section 3 of \cite{Wan15}. It is easy to see that $c_{\theta,\CO}(x)$ are uniquely determined by $\theta$.
If $\theta$ is a quasi-character on $G(F)$ and $\Omega\subset G(F)$ is an open $G$-domain, then $\theta 1_{\Omega}$ is still a quasi-character.

\subsection{Quasi-characters under parabolic induction}
Let $M$ be a Levi subgroup of $G$. Given an invariant distribution $D^M$ on $M(F)$, we define the induced distribution $D=I_{M}^{G}(D^M)$ on $G(F)$ as follows.

Fix a parabolic subgroup $P=MU\in \CP(M)$ and a hyperspecial maximal compact subgroup $K$. Assume that the Haar measure on $G(F)$, $M(F)$, $U(F)$ and $K$ are compatible, i.e. $\int_G=\int_M \int_U \int_K$. For $f\in C_{c}^{\infty}(G(F))$, define $f_P\in C_{c}^{\infty}(M(F))$ to be
$$f_P(m)=\delta_P(m)^{1/2}\int_K \int_{U(F)}f(k^{-1}muk)du dk.$$
Then we define $D(f)=D^M(f_P)$.

If $D^M$ is represented by a function $\theta^M$ on $M_{reg}(F)$, locally integrable on $M(F)$ and invariant under conjugation, i.e. $D^M(f)=\int_{M(F)} f(m)\theta^M(m)dm$ for all $f\in C_{c}^{\infty}(M(F))$. Then $D$ is also represented by a function $\theta$ on $G_{reg}(F)$ defined by
$$\theta(x)=\Sigma_{x'\in \CX^M(x)} D^G(x)^{-1/2}D^M(x')^{1/2} \theta^M(x'),\; x\in G_{reg}(F).$$
Here $\CX^M(x)$ is the set of the $M(F)$-conjugation classes in the $G(F)$-conjugation class of $x$. In particular, if $\tau$ is an irreducible admissible representation of $M(F)$ and $\pi=I_{P}^{G}(\tau)$, then $\theta_{\pi}=I_{M}^{G}(\theta_{\tau})$.

Now we talk about the parabolic induction of quasi-characters. If $\CO^M\in Nil(\Fm)$ and $\CO\in Nil(\Fg)$, we say $\CO$ is contained in the induced orbit of $\CO^M$ if the intersection $\CO\cap (\CO^M+\Fu(F))$ is a nonempty open subset in $\CO^M+\Fu(F)$. The following result is Lemma 2.3 of \cite{W12}.

\begin{lem}\label{quasi character parabolic}
If $\theta^M$ is a quasi-character of $M(F)$ and $\theta=I_{M}^{G}(\theta^M)$. Then
\begin{enumerate}
\item $\theta$ is a quasi-character of $G(F)$.
\item If $x\in G_{ss}(F)$ and $\CO\in Nil(\Fg_x)$ is a regular orbit, then we have
\begin{eqnarray*}
c_{\theta,\CO}(x)&=&\Sigma_{x'\in \CX^M(x)} \Sigma_{g\in \Gamma_{x'}/G_x(F)} \Sigma_{\CO'} D^G(x)^{-1/2}D^M(x')^{1/2}\\
&&[Z_M(x')(F):M_{x'}(F)]^{-1} c_{\theta^M,\CO'}(x').
\end{eqnarray*}
Here $\CO'$ runs over elements in $Nil(\Fm_{x'})$ such that $g\CO$ is contained in the induced orbit of $\CO'$. And for $x'\in \CX^M(x)$, $\Gamma_{x'}$ is the set of $g\in G(F)$ such that $gxg^{-1}=x'$.
\end{enumerate}
\end{lem}

\subsection{Definition of strongly cuspidal functions}
We say a function $f\in C_{c}^{\infty}(Z_G(F)\backslash G(F),\chi)$ is strongly cuspidal if for every proper parabolic subgroup $P=MU$ of $G$, and for every $x\in M(F)$, we have
\begin{equation}\label{cuspidal 1}
\int_{U(F)} f(xu)du=0.
\end{equation}
Some basic examples for strongly cuspidal functions are the matrix coefficients of supercuspidal representations, or the pseudo coefficients of discrete series (see Section 3.5).

\subsection{Some properties of strongly cuspidal functions}
We first study the weight orbital integral associated to the strongly cuspidal functions. The following lemma is proved in Section 5.2 of \cite{W10}.
\begin{lem}
Let $M$ be a Levi subgroup of $G$ and $K$ be a hyperspecial maximal compact subgroup.
If $f\in C_{c}^{\infty}(Z_G(F)\backslash G(F),\chi)$ is strongly cuspidal and $x\in M(F)\cap G_{reg}(F)$, then the following hold.
\begin{enumerate}
\item The weight orbital integral $J_M(x,f)$ does not depend on the choice of $K$.
\item For every $y\in G(F)$, we have $J_M(x,{}^y f)=J_M(x,f)$.
\item If $A_{G_x}\neq A_M$, then $J_M(x,f)=0$.
\end{enumerate}
\end{lem}

The next result is about the weighted characters associated to the strongly cuspidal functions.
\begin{lem}
If $f\in C_{c}^{\infty}(Z_G(F)\backslash G(F),\chi^{-1})$ is strongly cuspidal, $M$ is a Levi subgroup of $G$ and $\tau$ is a tempered representation of $M(F)$ whose central character equals to $\chi$ on $Z_G(F)$, then the following hold.
\begin{enumerate}
\item For any $L\in \CL(M)$ and $Q\in \CF(L)$, $J_{L}^{Q}(\tau,f)=0$ if $L\neq M$ or $Q\neq G$.
\item If $\tau$ is induced from a proper parabolic subgroup of $M$, then $J_{M}^{G}(\tau,f)=0$.
\item For $x\in G(F)$, we have $J_{xMx^{-1}}^{G}(x\tau x^{-1},f)=J_{M}^{G}(\tau,f)$.
\item The weight character $J_{M}^{G}(\tau,f)$ does not depend on the choice of $K$, and the way we normalize the intertwining operators.
\end{enumerate}
\end{lem}

\begin{proof}
See Section 2.2 of \cite{W12}, or Section 5.4 of \cite{B15}.
\end{proof}

Now same as in Section 6.1 of \cite{Wan15}, given a strongly cuspidal function $f\in C_{c}^{\infty}(Z_G(F)\backslash G(F),\chi)$, we can define a function $\theta_f$ on $G_{reg}(F)$ with central character $\chi$ to be
\begin{equation}
\theta_f(x)=(-1)^{a_{M(x)}-a_G} \nu(G_x)^{-1} D^G(x)^{-1/2} J_{M(x)}(x,f),\; x\in G_{reg}(F).
\end{equation}
Here $M(x)$ is the centralizer of $A_{G_x}$ in $G$, $a_G$ is the dimension of $A_G$, and the same for $a_{M(x)}$. The following Proposition is just Proposition 6.3 of \cite{Wan15}.

\begin{prop}\label{distribution 2}
The following hold.
\begin{enumerate}
\item The function $\theta_f$ is invariant under $G(F)$-conjugation. It is compactly supported modulo conjugation and modulo the center, locally integrable on $G(F)$ and locally constant on $G_{reg}(F)$.
\item $\theta_f$ is a quasi-character.
\end{enumerate}
\end{prop}

\textbf{For the rest of this section, we assume that $G$ is $GL_n(D)$ for some division algebra $D/F$. In particular, all irreducible tempered representation $\pi$ of $G(F)$ is of the form $\pi=I_{M}^{G}(\tau)$ for some $\tau\in \Pi^2(M)$.} For such $\pi$, let $\chi$ be the central character of $\pi$. For $f\in C_{c}^{\infty}(Z_G(F)\backslash G(F),\chi^{-1})$ strongly cuspidal, define
\begin{equation}\label{theta_f}
\theta_f(\pi)=(-1)^{a_G-a_M}J_{M}^{G}(\tau,f).
\end{equation}

\begin{prop}\label{expansion of character}
For every $f\in C_{c}^{\infty}(Z_G(F)\backslash G(F),\chi^{-1})$ strongly cuspidal, we have
$$\theta_f=\int_{Temp(G,\chi)} \theta_f(\pi) \bar{\theta}_{\pi} d\pi.$$
\end{prop}

\begin{proof}
This is just Proposition 5.6.1 of \cite{B15}. The only thing worth to mention is that the function $D(\pi)$ in the loc. cit. is identically $1$ in our case since we assume $G=GL_n(D)$.
\end{proof}

To end this section, we need a local trace formula for strongly cuspidal functions. It will be used in Section 7 for the proof of the spectral side of our trace formula. For $f\in \CC(G(F),\chi^{-1}),\;f'\in \CC(G(F),\chi)$ and $g_1,g_2\in G(F)$, set
$$K_{f,f'}^{A}(g_1,g_2)=\int_{Z_G(F)\back G(F)} f(g_{1}^{-1}gg_2) f'(g) dg.$$
By Proposition \ref{h-c function}, the integral above is absolutely convergent.

\begin{thm}\label{local trace formula}
\begin{enumerate}
\item For all $d\geq 0$, there exist $d'\geq 0$, a continuous semi-norm $\nu_{d,d'}$ on $\CC(G(F),\chi^{-1})$ and a continuous semi-norm $\nu_{d,d'}'$ on $\CC(G(F),\chi)$ such that
$$|K_{f,f}^{A}(g_1,g_2)|\leq \nu_{d,d'}(f)\nu_{d,d'}'(f') \Xi^G(g_1)\sigma_0(g_1)^{-d} \Xi^G(g_2) \sigma_0(g_2)^{d'}$$
and
$$|K_{f,f}^{A}(g_1,g_2)|\leq \nu_{d,d'}(f)\nu_{d,d'}'(f') \Xi^G(g_1)\sigma_0(g_1)^{d'} \Xi^G(g_2) \sigma_0(g_2)^{-d}.$$
\item Assume that $f$ is strongly cuspidal for the rest part of the Theorem, then for all $d\geq 0$, there exists a continuous semi-norm $\nu_{d}$ on $\CC(G(F),\chi^{-1})$ and a continuous semi-norm $\nu_{d}'$ on $\CC(G(F),\chi)$ such that
$|K_{f,f'}^{A}(g,g)|\leq \nu_{d}(f)\nu_{d}'(f') \Xi^G(g)^2 \sigma_0(g)^{-d}.$
\item There exists $c>0$ such that for all $d\geq 0$, there exists $d'\geq 0$ such that
$|K_{f,f'}^{A}(g,hg)|\ll \Xi^G(g)^2 \sigma_0(g)^{-d} e^{c\sigma_0(h)} \sigma_0(h)^{d'}.$
\item Set
$J^A(f,f')=\int_{Z_G(F)\back G(F)} K_{f,f'}^{A}(g,g)dg.$
This is absolutely convergent by part (2). Then we have
$$J^{A}(f,f')=\int_{Temp(G,\chi)} \theta_f(\pi)\theta_{\bar{\pi}}(f')d\pi.$$
\end{enumerate}
\end{thm}

\begin{proof}
This is just Theorem 5.5.1 of \cite{B15}.
\end{proof}

\subsection{Pseudo coefficients}
Recall that we assume $G=\GL_n(D)$ for some division algebra $D/F$. Let $\pi$ be a discrete series of $G(F)$ with central character $\chi$. For $f\in C_{c}^{\infty}(\zg,\chi^{-1})$, we say $f$ is a pseudo coefficient of $\pi$ if the following conditions holds.
\begin{itemize}
\item $tr(\pi(f))=1$.
\item For all $\sigma\in Temp(G,\chi)$ with $\sigma \neq \pi$, we have $trace(\sigma(f))=0$.
\end{itemize}

\begin{lem}\label{pseudo coefficient}
For all discrete series $\pi$ of $G(F)$ with central character $\chi$, the pseudo coefficient of $\pi$ exists. Moreover, all pseudo coefficients are strongly cuspidal.
\end{lem}

\begin{proof}
The existence of the pseudo coefficient is proved in \cite{BDK}. Let $f$ be a pseudo coefficient, we want to show that $f$ is strongly cuspidal. By the definition of $f$, we know that for all proper parabolic subgroup $P=MU$ of $G$, and for all tempered representations $\tau$ of $L(F)$, we have $trace(\pi'(f))=0$ where $\pi'=I_{P}^{G}(\tau)$. Then by Section 5.3 of \cite{B15}, we know that $f$ is strongly cuspidal. This proves the lemma.
\end{proof}

\section{The Ginzburg-Rallis model}
In this section, we study the analytic and geometric properties of the Ginzburg-Rallis model. Geometrically, we show that it is a wavefront spherical variety. This gives us the weak Cartan decomposition. Analytically, we show it has polynomial growth as a homogeneous space. Then by applying all such properties, we prove some estimations for several integrals which will be used in Section 5 and Section 7. This is a technical section, readers may assume the results in the section at the beginning and come back for the proof later.
\subsection{Definition of the Ginzburg-Rallis model}
Let $(G,H)$ be the pair $(G,H)$ or $(G_D,H_D)$ as in Section 1, and let $G_0=M$, then $(G_0,H_0)$ is just the trilinear model of $GL_2(F)$ or $GL_1(D)$. We define a homomorphism $\lambda: U(F)\rightarrow F$ to be
$$\lambda(u(X,Y,Z))=\tr(X)+\tr(Y).$$
Therefore the character $\xi$ we defined in Section 1 can be written as $\xi(u)=\psi(\lambda(u))$ for $u\in U(F)$. Similarly, we can define $\lambda$ on the Lie algebra of $U$.

\begin{lem}\label{norm}
\begin{enumerate}
\item The map $G\rightarrow H\backslash G$ has the norm descent property.
\item The orbit of $\lambda$ under the $M(F)$-conjugation is a Zariski open subset in $(\Fu/[\Fu,\Fu])^{\ast}$.
\end{enumerate}
\end{lem}

\begin{proof}
(1) Since the map is obviously $G$-equivariant, by Proposition 18.2 of \cite{K05}, we only need to show that it admits a section over a nonempty Zariski-open subset. Let $\bar{P}=M\bar{U}$ be the opposite parabolic subgroup of $P=MU$ with respect to $M$, let $P'$ be the subgroup of $\bar{P}$ consists of elements in $\bar{P}$ whose $M$-part is of the form $(1,h_1,h_2)$ where $h_1,h_2\in GL_2(F)$ or $GL_1(D)$. Then by Bruhat decompostion, the map $\phi:P'\rightarrow H\backslash G$ is injective and the image is a Zariski open subset of $H\backslash G$. Then the composition of $\phi^{-1}$ and the inclusion $P'\hookrightarrow G$ is a section on $Im(\phi)$. This proves (1).

(2) Assume $G=GL_6(F)$, we can easily identify $(\Fu/[\Fu,\Fu])^{\ast}$ with $M_2(F)\times M_2(F)$ where $M_2(F)$ are the two by two matrix over $F$. Then it is easy to see the orbit of $\lambda$ under the $M(F)$-conjugation is $GL_2(F)\times GL_2(F)$, which is a Zariski open subset. This proves (2) for the split case. The proof for the quaternion case is similar.
\end{proof}

\subsection{The spherical pair $(G,H)$}\label{spherical section}
We say a parabolic subgroup $\bar{Q}$ of $G$ is good if $H\bar{Q}$ is a Zariski open subset of $G$. This is equivalent to say that $H(F)\bar{Q}(F)$ is open in $G(F)$ under the analytic topology.

\begin{prop}\label{spherical}
\begin{enumerate}
\item There exist minimal parabolic subgroups of $G$ that are good and they are all conjugated to each other by some elements in $H(F)$. If $\bar{P}_{min}=M_{min}\bar{U}_{min}$ is a good minimal parabolic subgroup, we have $H\cap \bar{U}_{min}=\{1\}$ and the complement of $H(F)\bar{P}_{min}(F)$ in $G(F)$ has zero measure.
\item A parabolic subgroup $\bar{Q}$ of $G$ is good if and only if it contains a good minimal parabolic subgroup.
\item Let $\bar{P}_{min}=M_{min}\bar{U}_{min}$ be a good minimal parabolic subgroup and $A_{min}=A_{M_{min}}$ be the split center of $M_{min}$, set
$$A_{min}^{+}=\{a\in A_{min}(F)\mid \mid \alpha(a)\mid \geq 1 \; for\; any \; \alpha\in \Psi(A_{min},\bar{P}_{min})$$
where $\Psi(A_{min},\bar{P}_{min})$ is the set of positive roots associated to $\bar{P}_{min}$. Then we have
\begin{enumerate}
\item $\sigma_0(h)+\sigma_0(a)\ll \sigma_0(ha)$ for all $a\in A_{min}^{+}$, $h\in H(F)$.
\item $\sigma(h)\ll \sigma(a^{-1}ha)$ and $\sigma_0(h)\ll \sigma_0(a^{-1}ha)$ for all $a\in A_{min}^{+}$, $h\in H(F)$.
\end{enumerate}
\item $(1),(2)$ and $(3)$ also holds for the pair $(G_0,H_0)$.
\end{enumerate}
\end{prop}

\begin{proof}
(1) We first show the existence of the minimal parabolic subgroup. In the quaternion case, we can just choose the lower triangle matrix, this is a good minimal parabolic subgroup by the Bruhat decomposition. (Note that in this case the minimal parabolic subgroup is not a Borel subgroup since $G$ is not split). In the split case, we first show that it is enough to find the minimal parabolic subgroup for the pair $(G_0,H_0)$. Let $B_0$ be a good minimal parabolic subgroup for the pair $(G_0,H_0)$, since we are in the split case, $B_0$ is a Borel subgroup of $G_0$. Let $B=\bar{U}B_0$, it is a Borel subgroup of $G$. By Bruhat decomposition, $\bar{U}P$ is open in $G$. Together with the fact that $B_0$ is a good Borel subgroup of $(G_0, H_0)$, we know $BH$ is open in $G$, which makes $B$ a good minimal parabolic subgroup.

For the pair $(G_0,H_0)$, let $B_0=(B^{+},B^{-},B')$ where $B^{+}$ is upper triangle Borel subgroup of $GL_2$, $B^{-}$ is lower triangle Borel subgroup of $GL_2$ and $B'=\left( \begin{array}{cc} 1 & -1 \\ 0 & 1 \end{array} \right) B^{-}\left( \begin{array}{cc} 1 & 1 \\ 0 & 1 \end{array} \right)$. It is easy to see that $B^{+}\cap B^{-}\cap B'=\{\left( \begin{array}{cc} a & 0 \\ 0 & a \end{array} \right) \}$, hence $B_0\cap H=\{\left( \begin{array}{cc} a & 0 \\ 0 & a \end{array} \right) \times \left( \begin{array}{cc} a & 0 \\ 0 & a \end{array} \right) \times \left( \begin{array}{cc} a & 0 \\ 0 & a \end{array} \right)\}$. Then by comparing the dimensions, we know $B_0$ is a good parabolic subgroup.

Now we need to show that two good minimal parabolic subgroups are conjugated to each other by some elements in $H(F)$. Let $\bar{P}_{min}$ be the good minimal parabolic subgroup defined above, let $\bar{P}_{min}'$ be another good minimal parabolic subgroup. We can always find $g\in G(F)$ such that $g\bar{P}_{min}g^{-1}=\bar{P}_{min}'$. Let $\CU=H\bar{P}_{min}$ and $\CZ=G-\CU$. If $g\in \CZ$, then
$$H\bar{P}_{min}'=hg\bar{P}_{min}g^{-1}\subset \CZ g^{-1}$$
which is impossible since $H\bar{P}_{min}'$ is Zariski open and $\CZ$ is Zariski closed. Hence $g\in \CU\cap G(F)=\CU(F)$. If $g\in H(F)\bar{P}_{min}(F)$, then we are done. So it is enough to show that
$$\CU(F)=H(F)\bar{P}_{min}(F).$$
We have the following two exact sequence:
$$0\rightarrow H^0(F,\bar{P}_{min})\rightarrow H^0(F,H\bar{P}_{min})\rightarrow H^0(F, H/H\cap \bar{P}_{min}),$$
$$0\rightarrow H^0(F,H\cap \bar{P}_{min})\rightarrow H^0(F,H)\rightarrow H^0(F, H/H\cap \bar{P}_{min})\rightarrow H^1(F, H\cap \bar{P}_{min})\rightarrow H^1(F,H).$$
Therefore it is enough to show that the map
\begin{equation}\label{galois}
H^1(F,H\cap \bar{P}_{min})\rightarrow H^1(F,H)
\end{equation}
is injective.

If $G$ is split, by our construction, $H\cap \bar{P}_{min}=\GL_1$. Since $H^1(F,\GL_n)=\{1\}$ for any $n\in \BN$, the map \eqref{galois} is injective. If $G$ is not split, by our construction, $H\cap \bar{P}_{min}=H_0$, $H/H\cap \bar{P}_{min}=U$. Then the map \eqref{galois} lies inside the exact sequence
$$0\rightarrow H^0(F,H_0)\rightarrow H^0(F,H)\rightarrow H^0(F,U)\rightarrow H^1(F,H_0)\rightarrow H^1(F,H)$$
It is easy to see that the map $H^0(F,H)\rightarrow H^0(F,U)$ is surjective, therefore \eqref{galois} is injective. This finishes the proof.

For the rest part of (1), since we have already proved that two good minimal parabolic subgroups can be conjugated to each other by some elements in $H(F)$, it is enough to prove the rest part for a specific good minimal parabolic $\bar{P}_{min}$ we defined above, which is obvious from the construction of $\bar{P}_{min}$. This proves (1). The proof for the pair $(G_0,H_0)$ is similar.

(2) Let $\bar{Q}$ be a good parabolic subgroup and $P_{min}\subset \bar{Q}$ be a minimal parabolic subgroup. Set
$$\CG=\{g\in G\mid g^{-1}P_{min}g \; is \; good\}.$$
This is a Zariski open subset of $G$ since it is the inverse image of the Zariski open subset $\{\CV\in Gr_n(\Fg)\mid \CV+\Fh=\Fg \}$ of the Grassmannian variety $Gr_n(\Fg)$ under the morphism $g\in G\rightarrow g^{-1}\Fp_{min} g\in Gr_n(\Fg)$, here $n=dim(P_{min})$. By (1), there exists good minimal parabolic subgroup, hence $\CG$ is non-empty. Since $\bar{Q}$ is good, $\bar{Q}H$ is a Zariski open subset, hence $\bar{Q}H\cap \CG\neq \emptyset$. So we can find $\bar{q}_0\in \bar{Q}$ such that $\bar{q}_{0}^{-1} P_{min}\bar{q}_0$ is a good parabolic subgroup. Let
$$\CQ=\{\bar{q}\in \bar{Q}\mid \bar{q}^{-1} P_{min}\bar{q}\; is \; good\}$$
Then we know $\CQ$ is a non-empty Zariski open subset. Since $\bar{Q}(F)$ is dense in $\bar{Q}$, $\CQ(F)$ is non-empty. Let $\bar{q}\in \CQ(F)$, then the minimal parabolic subgroup $\bar{q}^{-1} P_{min}\bar{q}$ is good and is defined over $F$. This proves (2). The proof for the pair $(G_0,H_0)$ is similar.

(3) By the first part of the proposition, two good minimal parabolic subgroups are conjugated to each other by some elements in $H(F)$, then it is easy to see that (a) and (b) do not depend on the choice of minimal parabolic subgroups, hence we may use the minimal parabolic subgroup $\bar{P}_{min}$ defined in (1). Then we show that (a) and (b) do not depend on the choice of $M_{min}$. Let $M_{min}, M_{min}'$ be two choices of Levi subgroup, then there exists $\bar{u}\in \bar{U}_{min}(F)$ such that $M_{min}'=\bar{u}M_{min}\bar{u}^{-1}$ and ${A'}_{min}^{+}=\bar{u} A_{min}^{+} \bar{u}^{-1}$. Since for $a\in A_{min}^{+}$, $a^{-1}\bar{u}a$ is a contraction, the sets $\{a^{-1}\bar{u}a\bar{u}^{-1} \mid a\in A_{min}^{+}\}$ and $\{a^{-1}\bar{u}^{-1}a\bar{u} \mid a\in A_{min}^{+}\}$ are bounded. Then we have
\begin{eqnarray*}
\sigma_0(h\bar{u}a\bar{u}^{-1})&\sim & \sigma_0(ha),\\
\sigma(\bar{u}a\bar{u}^{-1}h\bar{u}a\bar{u}^{-1})&\sim & \sigma(a^{-1}ha),\\
\sigma_0(\bar{u}a\bar{u}^{-1}h\bar{u}a\bar{u}^{-1})&\sim & \sigma_0(a^{-1}ha)
\end{eqnarray*}
for all $a\in A_{min}^{+}$ and $h\in H(F)$. Therefore (a) and (b) do not depend on the choice of $M_{min}$. We may choose
$$M_{min}=\{diag(\left( \begin{array}{cc} a_1 & 0 \\ 0 & a_2 \end{array} \right),\left( \begin{array}{cc} a_3 & 0 \\ 0 & a_4 \end{array} \right),\left( \begin{array}{cc} a_5 & a_5-a_6 \\ 0 & a_6 \end{array} \right)) \mid a_i\in F^{\times}\}$$
in the split case, and choose
$$M_{min}=\{diag(b_1,b_2,b_3)\mid b_j\in D^{\times}\}$$
in the non-split case.

For part (a), let $h=uh_0$ for $u\in U(F)$ and $h_0\in H_0(F)$, then we know $\sigma_0(h)\ll \sigma_0(h_0)+\sigma_0(u)$ and $\sigma_0(ha)=\sigma_0(uh_0 a)\gg \sigma_0(u)+\sigma_0(h_0 a)$, so we may assume that $h=h_0\in H_0(F)$. If we are in the non-split case, $Z_{H_0}\backslash H_0(F)$ is compact, the argument is trivial. In the split case, since the norm is $K$-invariant, by the Iwasawa decomposition, we may assume that $h_0$ is upper triangle. Then by using the same argument as above, we can get rid of the unipotent part, so we may assume that $h_0=diag(h_1,h_2)$ with $h_1,h_2\in F^{\times}$. By our choice of $M_{min}$,
\begin{equation}\label{torus 1}
a=diag(\left( \begin{array}{cc} a_1 & 0 \\ 0 & a_2 \end{array} \right),\left( \begin{array}{cc} a_3 & 0 \\ 0 & a_4 \end{array} \right),\left( \begin{array}{cc} a_5 & a_5-a_6 \\ 0 & a_6 \end{array} \right))=diag(A_1,A_2,A_3)
\end{equation}
with $\mid a_2\mid\leq \mid a_1\mid\leq\mid a_3\mid\leq\mid a_4\mid\leq\mid a_5\mid\leq\mid a_6\mid$. Since we only consider $\sigma_0$, may assume that $\Pi a_i=1$ and $h_1h_2=1$. (In general, after modulo the center, we cannot make determinant equal to 1, there should be some square class left. But we are talking about majorization, the square class will not effect our estimation.) In order to make the argument holds for the pair $(G_0,H_0)$, here we only assume that $\mid a_2\mid\leq \mid a_1\mid,\mid a_3\mid\leq\mid a_4\mid,\mid a_5\mid\leq\mid a_6\mid$. It is enough to show that
\begin{equation}\label{6.2}
\sigma(h_0)+\sigma(a)\ll \sigma(h_0a).
\end{equation}
In this case, $\sigma(h_0)\sim \log(max\{\mid h_1\mid,\mid h_2\mid \})$ and $\sigma(a)\sim \log(max\{\mid a_6\mid,\mid a_4\mid,\mid a_1\mid \})\sim \log(max\{ \mid a_{2}^{-1} \mid \mid a_{3}^{-1} \mid\mid a_{5}^{-1} \mid\})$.

\begin{itemize}
\item If $h_2\geq 1$, we have $\sigma(h_0)\sim \log(\mid h_2\mid ),\parallel h_0 A_3\parallel \geq \mid a_6 h_2\mid$, and $\parallel h_0 A_2\parallel \geq \mid a_4 h_2\mid$. So if $max\{\mid a_6\mid,\mid a_4\mid,\mid a_1\mid \}=\mid a_6\mid \; or\; \mid a_4\mid$, \eqref{6.2} holds. By the same argument, if $max\{ \mid a_{2}^{-1} \mid \mid a_{3}^{-1} \mid\mid a_{5}^{-1} \mid\}=\mid a_{3}^{-1} \mid \; or \; \mid a_{5}^{-1} \mid$, \eqref{6.2} also holds. Now the only case left is $max\{\mid a_6\mid,\mid a_4\mid,\mid a_1\mid \}=\mid a_1\mid$ and $max\{ \mid a_{2}^{-1} \mid \mid a_{3}^{-1} \mid\mid a_{5}^{-1} \mid\}=\mid a_{2}^{-1}\mid$.
\begin{itemize}
\item If $\mid a_6 \mid\geq 1$, then $\parallel h_0 A_3\parallel \geq \mid a_6 h_2\mid$ and $\parallel h_0 A_1\parallel \geq \mid a_{2}^{-1} h_{2}^{-1}\mid$. Hence $\parallel h_0A_1 \parallel \parallel h_0A_3\parallel^2 \geq \mid a_{2}^{-1} a_{6}^{2} h_2\mid \geq \mid a_{2}^{-1} h_2\mid$. In particular, \eqref{6.2} holds.
\item If $\mid a_6\mid <1$, then $\mid a_5\mid <1$. In this case, $\parallel h_0 A_3\parallel \geq \mid a_{5}^{-1} h_2\mid$ and $\parallel h_0 A_1\parallel \geq \mid a_{1} h_{2}^{-1}\mid$. Hence $\parallel h_0A_1 \parallel \parallel h_0A_3\parallel^2 \geq \mid a_{5}^{-2} a_{1} h_2\mid \geq \mid a_{1} h_2\mid$. In particular, \eqref{6.2} holds.
\end{itemize}
\item If $h_1\geq 1$, the argument is similar as above, we will skip it here.
\end{itemize}
This finishes the proof of (a) for both the $(G,H)$ and the $(G_0,H_0)$ pairs.

For part (b), the argument for $\sigma_0$ is an easy consequence of the argument for $\sigma$, so we only prove the first one. Still let $h=uh_0$, by the definition of $A_{min}^{+}$, $a^{-1}ua$ is an extension of $u$(i.e. $\sigma(a^{-1}ua)\geq \sigma(u)$), so we can still reduce to the case $h=h_0\in H_0(F)$. For the non-split case, the argument is trivial since $a^{-1}h_0 a=h_0$. For the split case, still let $a=diag(A_1,A_2,A_3)$ as above, it is enough to show that for any $h\in GL_2(F)$,
\begin{equation}\label{6.1}
\parallel h\parallel \leq max\{\parallel A_{i}^{-1} hA_i\parallel, i=1,2,3\}.
\end{equation}
Let $h=\left( \begin{array}{cc} x_{11} & x_{12} \\ x_{21} & x_{22} \end{array} \right)$, may assume that $det(h)\geq 1$, then $\parallel h\parallel=max\{x_{ij}\}$. If $\parallel h\parallel=x_{11}, x_{21}\; or\; x_{22}$, it is easy to see that $\parallel h\parallel \leq \parallel A_{1}^{-1} hA_1\parallel$. If $\parallel h\parallel=x_{12}$, then $\parallel h\parallel \leq \parallel A_{2}^{-1} hA_2\parallel$. Therefore \eqref{6.1} holds, and this finishes the proof of (b).

(4) is already covered in the proof of (1), (2) and (3).
\end{proof}

The above proposition tells us $X=H\backslash G$ is a spherical variety of $G$ and $X_0=H_0\backslash G_0$ is a spherical variety of $G_0$. In \cite{SV12}, the authors have introduced the notion of wavefront spherical variety. In the next proposition, we are going to show that $X_0$ is a wavefront spherical variety of $G_0$, we need to use this result for the weak Cartan decomposition of $(G,H)$ and $(G_0,H_0)$.

\begin{prop}\label{wavefront}
$X_0$ is a wavefront spherical variety of $G_0$.
\end{prop}

\begin{proof}
It's is enough to show that the little Weyl group $W_{X_0}$ of $X_0$ is equal to the Weyl group of $G_0$, which is $(\BZ/2\BZ)^3$. Here we use the method introduced by Knop in \cite{Knop95} to calculate the little Weyl group. To be specific, use the same notation as loc. cit., let $B=B_1\times B_2\times B_3$ be a Borel subgroup of $G_0$. Without loss of generality, may assume each $B_i$ are the standard Borel subgroup of $\GL_2$ consisting of upper triangular matrix. Let $\FB(X_0)$ be the set of all non-empty, closed, irreducible, $B$-stable subsets of $X_0$. It is easy to see that there is a bijection between $\FB(X_0)$ and the set of all non-empty, closed, irreducible, $H_0$-stable subsets of $G_0/B\simeq (\BP^1)^3$. And we can easily write down such orbits: $(\BP^1)^3, \;X_{12},\; X_{13},\; X_{23}$ and $Y$ where $X_{ij}=\{(a_1,a_2,a_3)\in (\BP^1)^3| a_i=a_j \}$ and $Y=\{(a_1,a_2,a_3)\in (\BP^1)^3| a_1=a_2=a_3 \}$. Therefore $\FB(X_0)$ contain five elements
\begin{equation}\label{borel orbit}
\FB(X_0)=\{X_0, Y_1, Y_2, Y_3, Z\}
\end{equation}
where $Z$ is the orbit of the identity element under the action of $B$, so it's an irreducible subset of codimension 2. And all $Y_i$'s are closed, irreducible, $B$-stable subsets of codimension 1, with $Y_1=\{ H_0\backslash (g,g'b,g')\mid b\in B_2, g,g'\in \GL_2\}$, $Y_2=\{ H_0\backslash (gb,g',g)\mid b\in B_1, g,g'\in \GL_2\}$, and $Y_3=\{ H_0\backslash (g,gb,g')\mid b\in B_2, g,g'\in \GL_2\}$. Now we study the action of the Weyl group $W=W_{G_0}$ of $G_0$ on the set $\FB(X_0)$.

Let $\Delta(G_0)=\{\alpha_1,\alpha_2,\alpha_2\}$ be the set of simple roots of $G_0$ with respect to the Borel subgroup $B$, here $\alpha_i$ is the simple root of the i-th $\GL_2$ with respect to $B_i$. For $i=1,2,3$, let $w_i\in W$ be the simple reflection associated to $\alpha_i$, and $P_i$ be the corresponding minimal parabolic subgroup of $G_0$ containing $B$ (i.e. $P_i$ has $B_j$ on the j-th component for $i\neq j$, and has $\GL_2$ in the $i$-th component). Then we know $W$ is generated by $w_i$'s, hence it is enough to study the action of $w_i$ on $\FB(X_0)$.

We first consider the action of $w_1$. It is easy to see that there are two non-empty, closed, irreducible, $P_1$-stable subsets of $X_0$: one is $Y_1$, the other one is $X_0$. Let
$$\FB(Y_1,P)=\{A\in \FB(X)\mid P_1 A=Y_1\}$$
and
$$\FB(X_0,P)=\{A\in \FB(X)\mid P_1 A=X_0\}.$$
We have $\FB(Y_1,P)=\{Y_1,Z\}$ and $\FB(X_0,P)=\{Y_2,Y_3,X_0\}$. By Theorem 4.2 of \cite{Knop95}, the action of $w_1$ on $\FB(X_0)$ is given by
$$w_1\cdot X_0=X_0, w_1\cdot Y_1=Y_1, w_1\cdot Y_2=Y_3, w_1\cdot Y_3=Y_2, w_1\cdot Z=Z$$
Similarly we can get the action of $w_2$ and $w_3$:
$$w_2\cdot X_0=X_0, w_2\cdot Y_1=Y_3, w_2\cdot Y_2=Y_2, w_2\cdot Y_3=Y_1, w_2\cdot Z=Z;$$
$$w_3\cdot X_0=X_0, w_3\cdot Y_1=Y_2, w_3\cdot Y_2=Y_1, w_3\cdot Y_3=Y_3, w_3\cdot Z=Z.$$
Hence the isotropy group of $X_0$ is $W$. By Theorem 6.2 of \cite{Knop95}, the little Weyl group $W_{X_0}$ is just $W$, therefore $X_0$ is a wavefront spherical variety of $G_0$.
\end{proof}

We need the weak Cartan decomposition for $X_0$ and $X$. Let $\bar{P}_0=M_0\bar{U}_0$ be a good minimal parabolic subgroup of $G_0$ and let $A_0=A_{M_0}$ be the maximal split center of $M_0$. Let
$$A_{0}^{+}=\{a\in A_0(F)|\; |\alpha(a)| \geq 1, \; \forall \alpha\in \Psi(A_0,\bar{P}_0) \}$$
where $\Psi(A_0,\bar{P}_{0})$ is the set of positive roots associated to $\bar{P}_{0}$.
Choose a good minimal parabolic subgroup $\bar{P}_{min}=\bar{P}_0 \bar{U}=M_{min}\bar{U}_{min}$ of $G$, $P_{min}$ be its opposite with respect to $M_{min}$, then we know $P_{min}\subset P$. Let $\Delta$ be the set of simple roots of $A_{min}=A_{M_{min}}=A_0$ in $P_{min}$, and let $\Delta_P=\Delta \cap \Psi(A_{min},\bar{P}_{min})$ be the subset of simple roots appeared in $\Fu$. For $\alpha\in \Delta_P$, let $\Fn_{\alpha}$ be the corresponding root space.
\begin{prop}\label{cartan 1}
\begin{enumerate}
\item There is a compact subset $\CK_0\subset G_0(F)$ such that
\begin{equation}\label{cartan 2}
G_0(F)=H_0(F)A_{0}^{+} \CK_0.
\end{equation}
\item There is a compact subset $\CK\subset G(F)$ such that
\begin{equation}\label{cartan 3}
G(F)=H(F)A_{0}^{+} \CK.
\end{equation}
\item The character $\xi$ is nontrivial on $\Fn_{\alpha}$ for all $\alpha\in \Delta_P$.
\end{enumerate}
\end{prop}

\begin{proof}
We first prove that (1) implies (2). By Iwasawa decomposition, there is a  compact subgroup $K$ of $G(F)$ such that $G(F)=P(F)K=U(F)M(F)K$. Now by part (1), there exists an open compact subset $\CK_0$ of $G_0(F)=M(F)$ such that $G_0(F)=H_0(F)A_{0}^{+} \CK_0$. Let $\CK=\CK_0 K$, then $H(F) A_{0}^{+} \CK=U(F)H_0(F)A_{0}^{+} \CK_0 K=U(F)M(F)K=G(F)$, this proves (2).

Now we prove (1): in the non-split case, $A_{0}^{+}=Z_{G_0}$ and $Z_{G_0}\backslash G_0(F)$ is compact, hence (1) is trivial. In the split case, if $F=\BR$, since $(G_0,H_0)$ is a wavefront spherical variety, (2) follows from Theorem 5.13 of \cite{KKSS}. If $F$ is p-adic, we refer the readers to Appendix A for explicit construction.

For part (3), it is easy to see that this is independent of the choice of good minimal parabolic subgroup, so we still use the one defined in Proposition \ref{spherical}. Then (3) just follows from direct computation.
\end{proof}

To end this section, we will show that the homogeneous space $X=H\back G$ has polynomial growth, we first recall the definition for polynomial growth in \cite{Ber88}.
\begin{defn}
We say a homogeneous space $X=H\back G$ of $G$ has polynomial growth if it satisfies the following condition:

Fix a compact neighborhood $K$ of the identity element in $G$, then there exist constants $d,C>0$ such that for every $R>0$, the ball $B(R)=\{x\in X\mid r(x)\leq R\}$ can be covered by less than $C(1+R)^d$ many $K-balls$ of the form $Kx,x\in X$. Here $r$ is a function on $X$ defined by $r(x)=\inf \{\sigma(g)\mid x=gx_0\}$ where $x_0\in X$ is a fixed point.
\end{defn}

\begin{rmk}
In our case, if we set $x_0=1$, then $r(x)=inf_{h\in H(F)} \sigma(hx)$. By Lemma \ref{norm}, $r(x)=\sigma_{H\back G}(x)$.
\end{rmk}

We need a lemma before proving the polynomial growth.

\begin{lem}\label{major 1}
\begin{enumerate}
\item Let $\CK\subset G(F)$ be a compact subset, we have
$\nor(xk)\sim \nor(x)$
for all $x\in H(F)\back G(F), k\in \CK$.
\item There exists a positive constant $d>0$ such that
\begin{equation}\label{major 1.1}
\sigma_{H\back G}(a)\sim \sigma_{Z_G\back G}(a)=\sigma_0(a)
\end{equation}
for all $a\in A_{0}^{+}$, the last equation is just the definition of $\sigma_0$.
\end{enumerate}
\end{lem}

\begin{proof}
(1) is trivial. For (2), since $G\rightarrow H\back G$ has norm descent property(Lemma \ref{norm}), may assume that
\begin{equation}\label{major 1.2}
\sigma_{H\back G}(x)=\inf_{h\in H(F)} \sigma_G(hx).
\end{equation}
Then we obviously have the inequality $\sigma_{H\back G}(g)\ll \sigma_0(g)$ for all $g\in G(F)$, so we only need to show that
$\sigma_0(a)\ll \sigma_{H\back G}(a)$
for all $a\in A_{0}^{+}$. By applying \eqref{major 1.2}, it is enough to show that for all $a\in A_{0}^{+}$ and $h\in H(F)$, we have
\begin{equation}\label{major 1.3}
\sigma_0(a)\ll \sigma_0(ha).
\end{equation}
We can write $h=uh_0$ for $u\in U(F), h_0\in H_0(F)$. Since for all $u\in U(F), g_0\in G_0(F)$, we have $\sigma_0(ug_0)\gg \sigma_0(g_0)$, therefore $\sigma_0(ha)\gg \sigma_0(h_0a)$. So it is enough to show for all $a\in A_{0}^{+}$ and $h_0\in H_0(F)$, we have
$\sigma_0(a)\ll \sigma_0(h_0 a).$
This just follows from Proposition \ref{spherical}(3). This finishes the proof of (2).
\end{proof}

\begin{prop}\label{polynomial growth}
$H(F)\back G(F)$ has polynomial growth as $G(F)$-homogeneous space.
\end{prop}

\begin{proof}
By Proposition \ref{cartan 1}, there exists a compact subset $\CK\subset G(F)$ such that
$G(F)=H(F) A_{0}^{+} \CK.$
Since $H(F)\cap A_{0}^{+}=Z_{G}(F)$, together with the lemma above, there exists a constant $c_0>0$ such that
$$B(R)\subset H(F) \{a\mid a\in A_{0}^{+}/Z_G(F), \sigma_0(a)\leq c_0R \} \CK$$
for all $R\geq 1$. Hence we only need to show that there exists a positive integer $N>0$ such that for all $R\geq 1$, the subset $\{a\in A_{0}^{+}/Z_G(F) \mid \sigma_0(a)<R\}$ can be covered by less than $(1+R)^N$ subsets of the form $C_0a$ with $a\in A_{0}^{+}$ and $C_0\subset A_{0}^{+}$ is a compact subset with nonempty interior. This is trivial.
\end{proof}

\subsection{Some estimations}
In the next two sections, we are going to prove several estimations for various integrals which will be used in later sections. The proof of some estimations are similar to the GGP case in \cite{B15}, we only include them here for completion.
\begin{lem}\label{major 2}
\begin{enumerate}
\item There exists $\epsilon>0$ such that the integral
\begin{equation}\label{strong temper 1}
\int_{Z_{H_0}(F)\backslash H_0(F)} \Xi^{G_0}(h_0)e^{\epsilon \sigma_0(h_0)} dh_0
\end{equation}
is absolutely convergent.
\item There exists $d>0$ such that the integral
\begin{equation}\label{strong temper 2}
\int_{Z_H(F)\backslash H(F)} \Xi^G(h) \sigma_0(h)^{-d} dh
\end{equation}
is absolutely convergent.
\item For all $\delta>0$, there exists $\epsilon>0$ such that the integral
\begin{equation}\label{strong temper 3}
\int_{Z_{G}(F)\backslash H(F)} \Xi^{G}(h)e^{\epsilon \sigma_0(h)} (1+\mid \lambda(h)\mid)^{-\delta}dh
\end{equation}
is absolutely convergent.
\end{enumerate}
\end{lem}

\begin{proof}
(1) If we are in the non-split case, $Z_{H_0}(F)\backslash H_0(F)$ is compact and the argument is trivial. If we are in the split case, $G_0=\GL_2\times \GL_2\times \GL_2$. By the definition of $\Xi^{G_0}$, for $h_0\in H_0(F)$, $\Xi^{G_0}(h_0)=(\Xi^{H_0}(h_0))^3$. But since $\Xi^{H_0}$ is the matrix coefficient of a tempered representation, it belongs to the space $L^{2+t}(Z_{H_0}(F)\backslash H_0(F))$ for any $t>0$. Then we choose $\epsilon>0$ small enough so that $e^{\epsilon \sigma_0(h_0)}\ll \Xi^{H_0}(h_0)^{-1/2}$. For such $\epsilon$, the integral \eqref{strong temper 1} will be absolutely convergent.

(2) Let $d>0$, by Proposition \ref{h-c function}(iv), if $d$ is sufficiently large,
\begin{eqnarray*}
\int_{Z_{H}(F)\backslash H(F)} \Xi^G(h) \sigma_0(h)^{-d} dh&=&\int_{Z_H(F)\backslash H_0(F)} \int_{U(F)} \Xi^G(h_0 u) \sigma_0(h_0 u)^{-d} du dh_0\\
&\ll& \int_{Z_{H_0}(F)\backslash H_0(F)} \delta_P(h_0)^{1/2} \Xi^{G_0}(h_0)dh_0\\
&=& \int_{Z_{H_0}(F)\backslash H_0(F)} \Xi^{G_0}(h_0)dh_0.
\end{eqnarray*}
And the last integral is absolutely convergent by (1).

(3) Since $\sigma_0(h_0 u)\ll\sigma_0(h_0)\sigma_0(u)$ for all $h_0\in H_0(F)$ and $u\in U(F)$, by applying (1), it suffices to prove the following claim.

\begin{claim}\label{major 2.1}
For all $\delta>0$ and $\epsilon_0>0$, there exists $\epsilon>0$ such that the integral
$$I_{\epsilon,\delta}^{0}(h_0)=\int_{U(F)} \Xi^{G}(uh_0)e^{\epsilon \sigma_0(u)} (1+\mid \lambda(u)\mid)^{-\delta}du$$
is absolutely convergent for all $h_0\in H_0(F)$, and we have
$$I_{\epsilon,\delta}^{0}(h_0)\ll \Xi^{G_0}(h_0)e^{\epsilon_0\sigma_0(h_0)}.$$
\end{claim}

Given $\delta,\epsilon,\epsilon_0,b>0$, we have $I_{\epsilon,\delta}^{0}(h_0)=I_{\epsilon,\delta,\leq b}^{0}(h_0)+I_{\epsilon,\delta,>b}^{0}(h_0)$ where
$$I_{\epsilon,\delta,\leq b}^{0}(h_0)=\int_{U(F)} 1_{\sigma_0\leq b}(u) \Xi^{G}(uh_0)e^{\epsilon \sigma_0(u)} (1+\mid \lambda(u)\mid)^{-\delta}du$$
and
$$I_{\epsilon,\delta,>b}^{0}(h_0)=\int_{U(F)} 1_{\sigma_0> b}(u) \Xi^{G}(uh_0)e^{\epsilon \sigma_0(u)} (1+\mid \lambda(u)\mid)^{-\delta}du.$$
For all $d>0$, we have
\begin{equation}\label{major 2.2}
I_{\epsilon,\delta,\leq b}^{0}(h_0)\leq e^{\epsilon b} b^d \int_{U(F)} \Xi^G(uh_0)\sigma_0(u)^{-d} du.
\end{equation}
By Proposition \ref{h-c function}(iv), we can choose $d>0$ such that the last integral of \eqref{major 2.2} is essentially bounded by $\delta_P(h_0)^{-1/2} \Xi^M(h_0)=\Xi^{G_0}(h_0)$ for all $h_0\in H_0(F)$. We fix such $d>0$, and then have
\begin{equation}\label{major 2.3}
I_{\epsilon,\delta,\leq b}^{0}(h_0)\ll e^{\epsilon b} b^d \Xi^{G_0}(h_0)
\end{equation}
for all $h_0\in H_0(F)$ and $b>0$.

On the other hand, there exists $\alpha>0$ such that $\Xi^G(gg')\ll e^{\alpha \sigma_0(g')}\Xi^G(g)$ for all $g,g'\in G(F)$. Therefore
\begin{equation}\label{major 2.4}
I_{\epsilon,\delta,>b}^{0}(h_0)\ll e^{\alpha\sigma_0(h_0)-\sqrt{\epsilon}b} \int_{U(F)} \Xi^G(u) e^{(\epsilon+\sqrt{\epsilon})\sigma_0(u)} (1+\mid \lambda(u)\mid)^{-\delta} du
\end{equation}
for all $h_0\in H_0(F)$ and $b>0$. Assume that we can find $\epsilon>0$ such that the last integral of \eqref{major 2.4} is convergent. Then by $\eqref{major 2.3}$ and $\eqref{major 2.4}$, we have
\begin{equation}\label{major 2.6}
I_{\epsilon,\delta}^{0}(h_0)\ll e^{\epsilon b}b^d \Xi^{G_0}(h_0)+e^{\alpha \sigma_0(h_0)-\sqrt{\epsilon} b}
\end{equation}
for all $h_0\in H_0(F)$ and $b>0$. Choose $\beta>0$ such that $e^{-\beta\sigma_0(h_0)}\ll \Xi^{G_0}(h_0)$ for all $h_0\in H_0(F)$, and
then by letting $b=\frac{\alpha+\beta}{\sqrt{\epsilon}} \sigma_0(h_0)$ in \eqref{major 2.6}, we have
\begin{eqnarray*}
I_{\epsilon,\delta}^{0}(h_0)&\ll& e^{\sqrt{\epsilon}(\alpha+\beta)\sigma_0(h_0)} (\frac{\alpha+\beta}{\sqrt{\epsilon}} \sigma_0(h_0))^d \Xi^{G_0}(h_0)+e^{\alpha \sigma_0(h_0)-(\alpha+\beta)\sigma_0(h_0)} (e^{\beta\sigma_0(h_0)} \Xi^{G_0}(h_0))\\
&\ll& e^{\sqrt{\epsilon}(\alpha+\beta+1)\sigma_0(h_0)} \Xi^{G_0}(h_0)+\Xi^{G_0}(h_0) \\
&\ll& e^{\sqrt{\epsilon}(\alpha+\beta+1)\sigma_0(h_0)} \Xi^{G_0}(h_0)
\end{eqnarray*}
for all $h_0\in H_0(F)$. Note that $\alpha$ and $\beta$ do not depend on the choice of $\epsilon$. Hence we can always choose $\epsilon>0$ small so that $\sqrt{\epsilon}(\alpha+\beta+1)<\epsilon_0$. This proves Claim \ref{major 2.1}.

So it remains to prove that we can find $\epsilon>0$ such that the integral in \eqref{major 2.4} is absolutely convergent. If we are in the non-split case, $P$ is a minimal parabolic subgroup of $G$, then this follows from Corollary B.3.1 of \cite{B15}. If we are in the split case, it is easy to see that the convergence of the integral is independent of the choice of $\lambda$ (under the M-conjugation), so we may temporarily let
$$\lambda(u(X,Y,Z))=x_{12}+x_{21}+y_{12}+y_{21}$$
where
$$X=\left( \begin{array}{cc} x_{11} & x_{12} \\ x_{21} & x_{22} \end{array} \right), Y=\left( \begin{array}{cc} y_{11} & y_{12} \\ y_{21} & y_{22} \end{array} \right).$$
Then we have a decomposition $\lambda=\lambda_{+}-\lambda_{-}$ where
$$\lambda_{+}(u(X,Y,Z))=x_{21}+y_{21}$$
and
$$\lambda_{-}(u(X,Y,Z))=-x_{12}-y_{12}.$$
The additive character $\lambda_{+}$ is the restriction of a generic additive character of a maximal unipotent subgroup contained in $P$ to $U$. In fact we can take the maximal unipotent subgroup to be the upper triangular unipotent matrix, and consider the additive character of the form $(x_{ij})_{1\leq i,j\leq 6}\rightarrow x_{12}+x_{23}+x_{34}+x_{45}+x_{56}$. By applying Corollary B.3.1 of \cite{B15} again, we know the integral
\begin{equation}\label{major 2.5}
\int_{U(F)} \Xi^G(u) e^{\epsilon\sigma_0(u)} (1+\mid \lambda_{+}(u)\mid)^{-\delta} du
\end{equation}
is convergent for $\epsilon$ small.

Fix an embedding $a:\BG_m \hookrightarrow M$ given by $t\rightarrow diag(1,t,1,t,1,t)$. It is easy to see that $\lambda_{+}(a(t)ua(t)^{-1})= t\lambda_{+}(u)$ and $\lambda_{-}(a(t)ua(t)^{-1})=t^{-1}\lambda_{-}(u)$ for all $t\in \BG_m$ and $u\in U(F)$. Let $\CU\subset F^{\times}$ be a compact neighborhood of 1, for all $\epsilon >0$, we have
\begin{eqnarray*}&&\int_{U(F)} \Xi^G(u) e^{\epsilon\sigma_0(u)} (1+\mid \lambda(u)\mid)^{-\delta} du \\
&\ll& \int_{U(F)} \Xi^G(u) e^{\epsilon\sigma_0(u)} (1+\mid \lambda(a(t)ua(t)^{-1})\mid)^{-\delta} du\\
&=&\int_{U(F)} \Xi^G(u) e^{\epsilon\sigma_0(u)} (1+\mid t\lambda_{+}(u)-t^{-1} \lambda_{-}(u)\mid)^{-\delta} du
\end{eqnarray*}
for all $t\in \CU$. Integrating the above inequality over $\CU$, we have
\begin{eqnarray*}
&&\int_{U(F)} \Xi^G(u) e^{\epsilon\sigma_0(u)} (1+\mid \lambda(u)\mid)^{-\delta} du\\
&\ll& \int_{U(F)} \Xi^G(u) e^{\epsilon\sigma_0(u)}\int_{\CU} (1+\mid t\lambda_{+}(u)-t^{-1} \lambda_{-}(u)\mid)^{-\delta} dt du
\end{eqnarray*}
By Lemma B.1.1 of \cite{B15}, there exists $\delta'>0$ only depends on $\delta>0$ such that the last expression above is essentially bounded by
$$\int_{U(F)} \Xi^G(u) e^{\epsilon\sigma_0(u)} (1+\mid t\lambda_{+}(u) \mid)^{-\delta'} du.$$
Then by \eqref{major 2.5}, we can find $\epsilon>0$ such that the integral on \eqref{major 2.4} is absolutely convergent. This finishes the proof of (3).
\end{proof}

\begin{lem}\label{major 3}
Let $\bar{P}_{min}=M_{min}\bar{U}_{min}$ be a good minimal parabolic subgroup of $G$.
\begin{enumerate}
\item For any $\delta>0$, there exist $\epsilon>0$ and $d>0$ such that the integral
$$I_{\epsilon,\delta}^{1}(m_{min})=\int_{Z_H(F)\backslash H(F)} \Xi^G(hm_{min}) e^{\epsilon \sigma_0(h)} (1+\mid \lambda(h)\mid)^{-\delta} dh$$
is absolutely convergent for all $m_{min}\in M_{min}(F)$ and we have
$$I_{\epsilon,\delta}^{1}(m_{min})\ll \delta_{\bar{P}_{min}}(m_{min})^{-1/2} \sigma_0(m_{min})^d$$
for all $m_{min}\in M_{min}(F)$.
\item Assume that $Z_{G_0}(F)$ is contained in $A_{M_{min}}(F)$. Then for any $\delta>0$, there exist $\epsilon>0$ and $d>0$ such that the integral
\begin{eqnarray*}
I_{\epsilon,\delta}^{2}(m_{min})&=&\int_{Z_H(F)\backslash H(F)}\int_{Z_H(F)\backslash H(F)}\\
&&\Xi^G(hm_{min})\Xi^G(h'hm_{min}) e^{\epsilon \sigma_0(h)} e^{\epsilon \sigma_0(h')}(1+\mid \lambda(h')\mid)^{-\delta} dh' dh
\end{eqnarray*}
is absolutely convergent for all $m_{min}\in M_{min}(F)$ and we have
$$I_{\epsilon,\delta}^{2}(m_{min})\ll \delta_{\bar{P}_{min}}(m_{min})^{-1} \sigma_0(m_{min})^d$$
for all $m_{min}\in M_{min}(F)$.
\end{enumerate}
\end{lem}

\begin{proof}\label{major 3.1}
(1) Since $\Xi^G(g^{-1})\sim \Xi^G(g),\sigma_0(g^{-1})\sim\sigma_0(g)$ and $\lambda(h^{-1})=-\lambda(h)$ for all $g\in G(F)$ and $h\in H(F)$, it is equivalent to prove the following Claim.
\begin{claim}
For any $\delta>0$, there exist $\epsilon>0$ and $d>0$, such that the integral
$$J_{\epsilon,\delta}^{1}(m_{min})=\int_{Z_H(F)\backslash H(F)} \Xi^G(m_{min}h) e^{\epsilon \sigma_0(h)} (1+\mid \lambda(h)\mid)^{-\delta} dh$$
is absolutely convergent for all $m_{min}\in M_{min}(F)$ and we have
$$J_{\epsilon,\delta}^{1}(m_{min})\ll \delta_{\bar{P}_{min}}(m_{min})^{1/2} \sigma_0(m_{min})^d$$
for all $m_{min}\in M_{min}(F)$.
\end{claim}

By Proposition \ref{h-c function}(ii), there exists $d>0$ such that
\begin{eqnarray*}
J_{\epsilon,\delta}^{1}(m_{min}) &\ll& \delta_{\bar{P}_{min}}(m_{min})^{1/2} \sigma_0(m_{min})^d \\
&&\times \int_{Z_H(F)\backslash H(F)} \delta_{\bar{P}_{min}}(m_{\bar{P}_{min}}(h))^{1/2} \sigma_0(h)^d e^{\epsilon \sigma_0(h)} (1+\mid \lambda(h)\mid)^{-\delta} dh
\end{eqnarray*}
for all $m_{min}\in M_{min}(F)$. Here $m_{\bar{P}_{min}}:G(F)\rightarrow \bar{P}_{min}(F)$ is the map induced by the Iwasawa decomposition. Since for any $\epsilon'>\epsilon>0$, $\sigma_0(h)^d e^{\epsilon \sigma_0(h)}\ll e^{\epsilon' \sigma_0(h)}$, it is enough to prove that for $\epsilon$ small, the integral
\begin{equation}\label{major 3.2}
\int_{Z_H(F)\backslash H(F)} \delta_{\bar{P}_{min}}(m_{\bar{P}_{min}}(h))^{1/2} e^{\epsilon \sigma_0(h)} (1+\mid \lambda(h)\mid)^{-\delta} dh
\end{equation}
is absolutely convergent. Since $\bar{P}_{min}$ is a good parabolic subgroup, we can find open compact neighborhoods of the identity $\CU_K\subset K,\CU_H\subset Z_H(F)\backslash H(F)$ and $\CU_{\bar{P}}\subset \bar{P}_{min}(F)$ such that $\CU_K\subset \CU_{\bar{P}}\CU_{H}$. We have the estimations
$$e^{\epsilon \sigma_0(k_H h)}\ll e^{\epsilon \sigma_0(h)}, \; (1+\mid \lambda(k_H h)\mid)^{-\delta}\ll (1+\mid \lambda(h)\mid)^{-\delta}$$
for all $h\in H(F)$ and $k_H\in \CU_H$. Therefore
\begin{eqnarray*}
&& \int_{Z_H(F)\backslash H(F)} \delta_{\bar{P}_{min}}(m_{\bar{P}_{min}}(h))^{1/2} e^{\epsilon \sigma_0(h)} (1+\mid \lambda(h)\mid)^{-\delta} dh\\
&\ll& \delta_{\bar{P}_{min}}(k_{\bar{P}})^{1/2}\int_{Z_H(F)\backslash H(F)} \delta_{\bar{P}_{min}}(m_{\bar{P}_{min}}(k_H h))^{1/2} e^{\epsilon \sigma_0(h)} (1+\mid \lambda(h)\mid)^{-\delta} dh\\
&=& \int_{Z_H(F)\backslash H(F)} \delta_{\bar{P}_{min}}(m_{\bar{P}_{min}}(k_{\bar{P}}k_H h))^{1/2} e^{\epsilon \sigma_0(h)} (1+\mid \lambda(h)\mid)^{-\delta} dh
\end{eqnarray*}
for all $k_H\in \CU_H$ and $k_{\bar{P}}\in\CU_{\bar{P}}$. This implies
\begin{eqnarray*}
&& \int_{Z_H(F)\backslash H(F)} \delta_{\bar{P}_{min}}(m_{\bar{P}_{min}}(h))^{1/2} e^{\epsilon \sigma_0(h)} (1+\mid \lambda(h)\mid)^{-\delta} dh\\
&\ll& \int_{Z_H(F)\backslash H(F)}\int_{\CU_K} \delta_{\bar{P}_{min}}(m_{\bar{P}_{min}}(kh))^{1/2}dk e^{\epsilon \sigma_0(h)} (1+\mid \lambda(h)\mid)^{-\delta} dh\\
&\ll& \int_{Z_H(F)\backslash H(F)} \int_K \delta_{\bar{P}_{min}}(m_{\bar{P}_{min}}(kh))^{1/2}dk e^{\epsilon \sigma_0(h)} (1+\mid \lambda(h)\mid)^{-\delta} dh.
\end{eqnarray*}
By Proposition \ref{h-c function}(iii), the inner integral above is equal to $\Xi^G(h)$, then the convergence of \eqref{major 3.2} for $\epsilon$ small just follows from (3) of Lemma \ref{major 2}, this finishes the proof of (1).
\\
\\
(2) By changing the variable $h'\rightarrow h'h^{-1}$ in the integral, it is enough to show that for $\epsilon>0$ small, the integral
\begin{eqnarray*}
I_{\epsilon,\delta}^{3}(m_{min})&=&\int_{Z_H(F)\backslash H(F)}\int_{Z_H(F)\backslash H(F)} \\
&&\Xi^G(hm_{min})\Xi^G(h'm_{min}) e^{\epsilon \sigma_0(h)} e^{\epsilon \sigma_0(h')}(1+\mid \lambda(h')-\lambda(h)\mid)^{-\delta} dh' dh
\end{eqnarray*}
is absolutely convergent for all $m_{min}\in M_{min}(F)$ and there exists $d>0$ such that
\begin{equation}\label{major 3.3}
I_{\epsilon,\delta}^{3}(m_{min})\ll \delta_{\bar{P}_{min}}(m_{min})^{-1} \sigma_0(m_{min})^d
\end{equation}
for all $m_{min}\in M_{min}(F)$. Let $a:\BG_m(F)\rightarrow Z_{G_0}(F)$ be a homomorphism given by $a(t)=diag(t,t,1,1,t^{-1},t^{-1})$ in the split case, and $a(t)=diag(t,1,t^{-1})$ in the non-split case. It is easy to see that $\lambda(a(t)ha(t)^{-1})=t\lambda(h)$ for all $h\in H(F)$ and $t\in \BG_m(F)$. Let $\CU\subset F^{\times}$ be an open compact neighborhood of 1. Since $Z_{G_0}$ is in the center of $M_{min}$, by making the transform $h'\rightarrow a(t)^{-1}h'a(t)$, we have
\begin{eqnarray*}
I_{\epsilon,\delta}^{3}(m_{min}) &\ll& \int_{Z_H(F)\backslash H(F)}\int_{Z_H(F)\backslash H(F)} \Xi^G(hm_{min})\Xi^G(a(t)h'm_{min}a(t)^{-1}) e^{\epsilon \sigma_0(h)} e^{\epsilon \sigma_0(a(t)h'a(t)^{-1})}\\
&& \times \int_{\CU} (1+\mid \lambda(a(t)h'a(t)^{-1})-\lambda(h)\mid)^{-\delta} dtdh'dh \\
&=& \int_{Z_H(F)\backslash H(F)}\int_{Z_H(F)\backslash H(F)} \Xi^G(hm_{min})\Xi^G(h'm_{min}) e^{\epsilon \sigma_0(h)} e^{\epsilon \sigma_0(h')}\\
&& \times \int_{\CU} (1+\mid t\lambda(h')-\lambda(h)\mid)^{-\delta} dtdh'dh
\end{eqnarray*}
for all $m_{min}\in M_{min}(F)$. By Lemma B.1.1 of \cite{B15}, there exists $\delta'>0$ only depends on $\delta$ such that the last integral above is essentially bounded by
\begin{eqnarray*}
&&\int_{Z_H(F)\backslash H(F)}\int_{Z_H(F)\backslash H(F)} \Xi^G(hm_{min})\Xi^G(h'm_{min}) \\
&& e^{\epsilon \sigma_0(h)} e^{\epsilon \sigma_0(h')} (1+\mid \lambda(h')\mid)^{-\delta'}(1+\mid\lambda(h)\mid)^{-\delta'} dh'dh\\
&=& I_{\epsilon,\delta'}^{1}(m_{min})^2
\end{eqnarray*}
for all $m_{min}\in M_{min}(F)$. Therefore the inequality \eqref{major 3.3} follows from part (1), this finishes the proof of (2).
\end{proof}

\subsection{The Harish-Chandra Schwartz spce of $H\backslash G$}
Let $C\subset G(F)$ be a compact subset with nonempty interior. Define the function
$\Xi^{H\backslash G}_{C}(x)=vol_{H\back G}(xC)^{-1/2}$ for $x\in H(F)\back G(F)$.
If $C'$ is another compact subset with nonempty interior, then
$\Xi^{H\backslash G}_{C}(x)\sim \Xi^{H\backslash G}_{C'}(x)$
for all $x\in H(F)\back G(F)$. We will only use the function $\Xi^{H\backslash G}_{C}$ for majorization. From now on, we will fix a particular $C$, and set
$\Xi^{H\backslash G}=\Xi^{H\backslash G}_{C}.$
The next proposition gives the properties for the function $\Xi^{H\backslash G}$, which is quiet similar to Proposition \ref{h-c function} for the group case.

\begin{prop}\label{major 4}
\begin{enumerate}
\item Let $\CK\subset G(F)$ be a compact subset, we have
$\hc(xk)\sim \hc(x)$
for all $x\in H(F)\back G(F)$ and $k\in \CK$.
\item Let $\bar{P}_0=M_0 \bar{U}_0$ be a good minimal parabolic subgroup of $G_0$ and $A_0=A_{M_0}$ be the split center of $M_0$. Set
$$A_{0}^{+}=\{a_0\in A_0(F)\mid \mid \alpha(a)\mid \geq 1 \; for\; all\; \alpha\in \Psi(A_0,\bar{P}_0)\}.$$
Then there exists $d>0$ such that
\begin{equation}\label{major 4.1}
\Xi^{G_0}(a)\delta_{P}(a)^{1/2}\sigma_{Z_{G_0}\back G_0}(a)^{-d}\ll \hc(a)\ll \Xi^{G_0}(a)\delta_{P}(a)^{1/2}
\end{equation}
for all $a\in A_{0}^{+}$.
\item There exists $d>0$ such that the integral
$$\int_{H(F)\back G(F)}\hc(x)^2 \nor(x)^{-d} dx$$
is absolutely convergent.
\item For all $d>0$, there exists $d'>0$ such that
$$\int_{H(F)\back G(F)} 1_{\nor\leq c}(x) \hc(x)^2 \nor(x)^d dx\ll c^{d'}$$
for all $c\geq 1$.
\item There exist $d>0$ and $d'>0$ such that
$$\int_{\zh}\Xi^G(x^{-1}hx)\sigma_0(x^{-1}hx)^{-d}dh\ll \hc(x)^2 \nor(x)^{d'}$$
for all $x\in H(F)\back G(F)$.
\item For all $d>0$, there exists $d'>0$ such that
$$\int_{\zh}\Xi^G(hx)\sigma_0(hx)^{-d'}dh\ll \hc(x) \nor(x)^{d}$$
for all $x\in H(F)\back G(F)$.
\item Let $\delta>0$ and $d>0$, then the integral
\begin{eqnarray*}
I_{\delta,d}(c,x)&=&\int_{\zh} \int_{\zh}1_{\sigma_0\geq c}(h') \Xi^G(hx)\Xi^G(h'hx)\\
&& \sigma_0(hx)^d \sigma_0(h'hx)^d (1+\mid \lambda(h')\mid)^{-\delta} dh'dh
\end{eqnarray*}
is absolutely convergent for all $x\in H(F)\back G(F)$ and $c\geq 1$. Moreover, there exist $\epsilon>0$ and $d'>0$ such that
$$I_{\delta,d}(c,x)\ll \hc(x)^2 \nor(x)^{d'} e^{-\epsilon c}$$
for all $x\in H(F)\back G(F)$ and $c\geq 1$.
\end{enumerate}
\end{prop}

\begin{proof}
The first one is trivial. For (2), let $\bar{P}=M\bar{U}$ be the parabolic subgroup opposite to $P$ with respect to $M$. We fix some compact subsets with nonempty interior for the following groups
$$C_{\bar{U}}\subset \bar{U}(F), C_0\subset G_0(F)=M(F),C_U\subset U(F).$$
By Bruhat decomposition, $C=C_{U}C_0C_{\bar{U}}$ is a compact subset of $G(F)$ with nonempty interior. By the definition of $\hc$, we have
$$\hc(g)\sim vol_{H\back G}(H(F)gC)^{-1/2},\; \forall g\in G(F).$$
By the definition of $\Xi^{G_0}$, there exists $d>0$ such that
$$\Xi^{G_0}(g_0)\sigma_{Z_{G_0}\back G_0}(g_0)^{-d}\ll vol_{G_0}(C_0g_0C_0)^{-1/2}\ll \Xi^{G_0}(g_0),\;\forall g_0\in G_0(F).$$
So in order to prove \eqref{major 4.1}, it is enough to show that
$$\delta_P(a)^{-1}vol_{G_0}(C_0a C_0)^{-1/2}\sim vol_{H\back G}(H(F)a C)$$
for all $a\in A_{0}^{+}$. By the definition of $C$, we know
$$H(F)aC=H(F)a C_{\bar{P}}$$
where $C_{\bar{P}}=C_0C_{\bar{U}}$. Thus we only need to prove
\begin{equation}\label{major 4.3}
\delta_P(a)^{-1}vol_{G_0}(C_0aC_0)^{-1/2}\sim vol_{H\back G}(H(F)a C_{\bar{P}})
\end{equation}
for all $a\in A_{0}^{+}$.

Let $C_{H_0}\subset H_0(F)$ be a compact subset with nonempty interior and let $C_{H}=C_{U}C_{H_0}$, it is a compact subset of $H(F)$ with nonempty interior. We claim that
\begin{equation}\label{major 4.4}
vol_{H\back G}(H(F)a C_{\bar{P}}) \sim vol_G(C_H aC_{\bar{P}})
\end{equation}
for all $a\in A_{0}^{+}$. In fact, we have
$$vol_G(C_H aC_{\bar{P}})=\int_{H(F)\back G(F)}\int_{H(F)} 1_{C_H aC_{\bar{P}}}(hx)dh dx.$$
The inner integral above is nonzero if and only if $x\in H(F) a C_{\bar{P}}$. If this holds, the inner integral is equal to
$$\vol_H(H(F)\cap C_H aC_{\bar{P}} x^{-1})=vol_H(C_H(H(F)\cap a C_{\bar{P}} x^{-1})).$$
Therefore in order to prove \eqref{major 4.4}, it is enough to show that
$$vol_H(C_H(H(F)\cap a C_{\bar{P}} x^{-1}))\sim 1$$
for all $a\in A_{0}^{+}$ and $x\in aC_{\bar{P}}$. For such $x$, $C_H\subset C_H(H(F)\cap a C_{\bar{P}} x^{-1})$, so we only need to show
$$vol_H(C_H(H(F)\cap a C_{\bar{P}} x^{-1}))\ll 1.$$
In order to prove this, it is enough to show that the set $H(F)\cap a C_{\bar{P}}' a^{-1}$ remains uniformly bounded for all $a\in A_{0}^{+}$, here $C_{\bar{P}}'=C_{\bar{P}}C_{\bar{P}}^{-1}$. Since $\bar{P}\cap H=H_0$, $H(F)\cap a C_{\bar{P}}' a^{-1}=H_0(F)\cap aC_0 'a^{-1}$ where $C_{0}'=C_{\bar{P}}' \cap G_0(F)$. For $h_0\in H_0(F)\cap aC_0 'a^{-1}$, $a^{-1}h_0a\in C_0 '$ is bounded. By Proposition \ref{spherical}(3), $\sigma(h_0)\ll \sigma(a^{-1} h_0a)$. Hence $H_0(F)\cap aC_0 'a^{-1}$ is uniformly bounded for $a_0\in A_{0}^{+}$, this finishes the proof of \eqref{major 4.4}.

Now by applying \eqref{major 4.4}, \eqref{major 4.3} is equivalent to
\begin{equation}\label{major 4.5}
\delta_P(a)^{-1}vol_{G_0}(C_0aC_0)\sim vol_G(C_H a C_{\bar{P}}),\; \forall a\in A_{0}^{+}.
\end{equation}
By the definition of $C_H$ and $C_{\bar{P}}$, $C_H a C_{\bar{P}}=C_U (C_{H_0} a C_0)C_{\bar{U}}$. Since we have a decomposition of the Haar measure on $G(F)$: $dg=\delta_P(g_0)^{-1} dudg_0d\bar{u}$ where $du,dg_0$ and $d\bar{u}$ are Haar measures on respectively $U(F),G_0(F)$ and $\bar{U}(F)$, we have
$$vol_G(C_U (C_{H_0} a C_0)C_{\bar{U}})\sim \delta_P(a)^{-1} vol_{G_0}(C_{H_0} aC_0).$$
Hence the last thing to show is for all $a_0\in A_{0}^{+}$, we have
\begin{equation}\label{major 4.6}
vol_{G_0}(C_0aC_0)\sim vol_{G_0}(C_{H_0} aC_0).
\end{equation}

The inequality $vol_{G_0}(C_0aC_0)\gg vol_{G_0}(C_{H_0} aC_0)$ is trivial. For the other direction, since $H_0(F)\bar{P}_0(F)$ is open in $G_0(F)$, may assume $C_0=C_{H_0}C_{\bar{P}_0}$ where $C_{\bar{P}_0}$ is a compact subset in $\bar{P}_0(F)$ with nonempty interior. By the definition of $A_{0}^{+}$, $a^{-1}C_{\bar{P}_0} a$ is uniformly bounded since the action on the unipotent part is a contraction and the action preserves the Levi part. Hence there exists a compact subset $C'\subset G_0(F)$ such that
$$a^{-1}C_{\bar{P}_0} a C_0\subset C'$$
for all $a\in A_{0}^{+}$. This implies
$$vol_{G_0}(C_0aC_0)\ll vol_{G_0}(C_{H_0} aC')\ll vol_{G_0}(C_{H_0} aC_0)$$
for all $a\in A_{0}^{+}$. This finishes the proof of \eqref{major 4.6} and hence the proof of (2).
\\
\\
(3) Set $B(R)=\{x\in H(F)\back G(F)\mid \sigma_{H\back G}(x)< R\}$. By Proposition \ref{polynomial growth}, there exists $N>0$ such that for all $R\geq 1$, the subset $B(R)$ can be covered by less than $(1+R)^N$ many subsets of the the form $xC$ for $x\in H(F)\back G(F)$ and $C\subset G(F)$ be a compact subset with non-empty interior. Let
$$I(R,d)=\int_{B(R+1)\back B(R)} \hc(x)^2 \nor(x)^{-d} dx.$$
We have
\begin{equation}\label{major 4.9}
\int_{H(F)\back G(F)}\hc(x)^2 \nor(x)^{-d} dx=\Sigma_{R=1}^{\infty} I(R,d).
\end{equation}
Since for all $R\geq 1$, $B(R+1)\back B(R)$ can be covered by some subsets $x_1 C,\cdots, x_{k_R}C$ with $k_R\leq (R+2)^N$, we have
\begin{equation}\label{major 4.7}
I(R,d)\leq \Sigma_{i=1}^{k_R} \int_{x_i C} \hc(x)^2 \nor(x)^{-d} dx
\end{equation}
for all $d>0$ and $R\geq 1$. Since $C$ is compact, together with the definition of $\hc$, we have
\begin{eqnarray*}
&&\int_{yC}\hc(x)^2 \nor(x)^{-d} dx\\
&\ll& vol_{H\back G}(yC) \hc(y)^2 \nor(y)^{-d}\\
&\ll& vol_{H\back G}(yC) vol_{H\back G}(yC)^{-1} \nor(y)^{-d}=\nor(y)^{-d}
\end{eqnarray*}
for all $y\in H(F)\back G(F)$. Combining with \eqref{major 4.7}, we have
\begin{equation}\label{major 4.8}
I(R,d)\ll \Sigma_{i=1}^{k_R} \nor(x_i)^{-d}
\end{equation}
for all $d>0$ and $R\geq 1$. Since $x_i C\cap (B(R+1)\back B(R)) \neq \emptyset$, $\nor(x_i)\gg R$. Combining with \eqref{major 4.8}, we have
$$I(R,d)\ll R^{-d}k_R\leq (R+2)^N R^{-d}$$
for all $d>0$ and $R\geq 1$. So once we let $d>N+1$, \eqref{major 4.9} is absolutely convergent, this finishes the proof of (3).
\\
\\
The proof of (4) is very similar to (3), we will skip it here. For (5), by the Cartan decomposition in Proposition \ref{cartan 1}, we may assume that $x\in A_{0}^{+}$. Then by applying part (2) and Lemma \ref{major 1}, we only need to show that there exists $d>0$ such that for all $a\in A_{0}^{+}$, we have
\begin{equation}\label{major 4.10}
\int_{\zh}\Xi^G(a^{-1}ha)\sigma_0(a^{-1}ha)^{-d}dh\ll \Xi^{G_0}(a)^2 \delta_P(a).
\end{equation}
But we know
\begin{eqnarray*}
&&\int_{\zh}\Xi^G(a^{-1}ha)\sigma_0(a^{-1}ha)^{-d}dh\\
&=&\int_{Z_{H_0}(F)\back H_0(F)}\int_{U(F)}\Xi^G(a^{-1}h_0 ua)\sigma_0(a^{-1}h_0 ua)^{-d}dudh_0\\
&=&\delta_P(a)\int_{Z_{H_0}(F)\back H_0(F)}\int_{U(F)}\Xi^G(a^{-1}h_0 au)\sigma_0(a^{-1}h_0 au)^{-d}dudh_0.
\end{eqnarray*}
By Proposition \ref{h-c function}(4), for $d>0$ large, we have
$$\int_{U(F)}\Xi^G(a^{-1}h_0 au)\sigma_0(a^{-1}h_0 au)^{-d}du\ll \Xi^{G_0}(a^{-1}h_0a)$$
for all $a\in A_0(F)$ and $h_0\in H_0(F)$. Thus for $d>0$ large, the left hand side of \eqref{major 4.10} is essentially bounded by
$$\delta_P(a)\int_{Z_{H_0}(F)\back H_0(F)}  \Xi^{G_0}(a^{-1}h_0a) dh_0.$$
So in order to prove \eqref{major 4.10}, it is enough to show that
\begin{equation}\label{major 4.11}
\int_{Z_{H_0}(F)\back H_0(F)}  \Xi^{G_0}(a^{-1}h_0a)dh_0 \ll \Xi^{G_0}(a)^2
\end{equation}
for all $a\in A_{0}^{+}$. If we are in the non-split case, $A_0=Z_{G_0}$, so $\Xi^{G_0}(a)=\Xi^{G_0}(1)$. Then \eqref{major 4.11} holds since $Z_{H_0}(F)\back H_0(F)$ is compact. In the split case, let $\CU_{H_0(F)}\subset H_0(F)$ and $\CU_{\bar{P}_0}\subset \bar{P}_0(F)$ be some compact neighborhoods of the identity. By the definition $A_{0}^{+}$, the subsets $a^{-1}\CU_{\bar{P}_0} a$ remain uniformly bounded as $a\in A_{0}^{+}$. So we have
$$\int_{Z_{H_0}(F)\back H_0(F)}  \Xi^{G_0}(a^{-1}h_0a) dh_0\ll \int_{Z_{H_0}(F)\back H_0(F)}  \Xi^{G_0}(a^{-1}p_1 h_1h_0h_2 p_2a) dh_0$$
for all $a\in A_{0}^{+}$, $h_1,h_2\in \CU_{H_0}$ and $p_1,p_2\in \CU_{\bar{P}_0}$. Let $K_0$ be a maximal compact subgroup of $G_0(F)$, since $\bar{P}_0$ is a good parabolic subgroup, there exist a compact neighborhood of the identity $\CU_{K_0}\subset K_0$ such that $\CU_{K_0}\subset \CU_{\bar{P}_0}\CU_{H_0}\cap \CU_{H_0}\CU_{\bar{P}_0}$. So we have
\begin{eqnarray*}
&&\int_{Z_{H_0}(F)\back H_0(F)}  \Xi^{G_0}(a^{-1}h_0a)\\
&\ll& \int_{Z_{H_0}(F)\back H_0(F)}\int_{\CU_{K_0}^{2}}  \Xi^{G_0}(a^{-1}k_1h_0k_2a)dk_1 dk_2 dh_0\\
&\ll& \int_{Z_{H_0}(F)\back H_0(F)}\int_{K_{0}^{2}}  \Xi^{G_0}(a^{-1}k_1h_0k_2a)dk_1 dk_2 dh_0
\end{eqnarray*}
for all $a\in A_{0}^{+}$. By Proposition \ref{h-c function}(6), the last integral above is bounded by
$$\Xi^{G_0}(a)^2\int_{Z_{H_0}(F)\back H_0(F)} \Xi^{G_0}(h_0) dh_0$$
for all $a\in A_{0}^{+}$. Then \eqref{major 4.11} follows from Lemma \ref{major 2}(1) and this finishes the proof of (5).
\\
\\
For (6), by applying the same reduction as in (5), we only need to show that there exists $d>0$ such that for all $a\in A_{0}^{+}$, we have
$$\int_{\zh}\Xi^G(ha)\sigma_0(ha)^{-d'}dh\ll \delta_P(a)^{1/2} \Xi^{G_0}(a).$$
Again we decompose $dh=dudh_0$ and by applying Proposition \ref{h-c function}(4), we only need to show that for all $a\in A_{0}^{+}$, we have
$$\int_{Z_{H_0}(F)\back H_0(F)} \Xi^{G_0}(h_0 a)dh_0\ll \Xi^{G_0}(a).$$
Then using the same argument as (5), together with Proposition \ref{h-c function}(6), we are reduced to show that the integral
$$\int_{Z_{H_0}(F)\back H_0(F)} \Xi^{G_0}(h_0) dh_0$$
is absolutely convergent, which is just Lemma \ref{major 2}(1). This finishes the proof of (6).
\\
\\
For (7), by applying the same reduction as in (5), together with the fact that for all $d>0$ and $\epsilon>0$, we have $1_{\sigma_0 \geq c}(h)\sigma_0(h)^d \ll e^{\epsilon \sigma_0(h)} e^{-\epsilon c/2}$, we are reduced to show that for all $\delta>0$, there exist $d>0$ and $\epsilon >0$ such that for all $a\in A_{0}^{+}$, we have
\begin{equation}\label{major 4.12}
\int_{\zh} \int_{\zh}\Xi^G(ha)\Xi^G(h'ha) e^{\epsilon \sigma_0(h)}e^{\epsilon \sigma_0(h')} (1+\mid \lambda(h')\mid)^{-\delta} dh'dh
\end{equation}
$$\ll \delta_P(a) \Xi^{G_0}(a)^2 \sigma_0(a)^d.$$
Let $\bar{P}_{min}=\bar{P}_0 \bar{U}$ and $M_{min}=M_0$, then $\bar{P}_{min}$ is a good parabolic subgroup of $G$, and $M_{min}$ is a Levi subgroup of it which contains $A_0$. By Lemma \ref{major 3}(2), there exists $\epsilon>0$ and $d>0$ such that
$$\int_{\zh} \int_{\zh}\Xi^G(ha)\Xi^G(h'ha) e^{\epsilon \sigma_0(h)}e^{\epsilon \sigma_0(h')} (1+\mid \lambda(h')\mid)^{-\delta} dh'dh$$
$$\ll \delta_{\bar{P}_{min}}(a)^{-1} \sigma_0(a)^d$$
for all $a\in A_{0}^{+}$. But we know $\delta_{\bar{P}_{min}}(a)^{-1}=\delta_P(a)\delta_{P_0}(a)$. By Proposition \ref{h-c function}(1), $\delta_{P_0}(a)\ll \Xi^{G_0}(a)^2$ for all $a\in A_{0}^{+}$. Therefore the inequality \eqref{major 4.12} holds for such $d$ and $\epsilon$, this finishes the proof of (7).
\end{proof}

\begin{lem}\label{major 5}
Let $\bar{Q}=M_Q \bar{U}_Q$ be a good parabolic subgroup of $G$, $H_{\bar{Q}}=H\cap \bar{Q}$, and let $G_{\bar{Q}}=\bar{Q}/U_{\bar{Q}}$ be the reductive quotient of $\bar{Q}$. Then we have
\begin{enumerate}
\item $H_{\bar{Q}}\cap U_{\bar{Q}}=\{1\}$, hence we can view $H_{\bar{Q}}$ as a subgroup of $G_{\bar{Q}}$. We also have $\delta_{\bar{Q}}(h_{\bar{Q}})=\delta_{H_{\bar{Q}}} (h_{\bar{Q}})$ for all $h_{\bar{Q}}\in H_{\bar{Q}}(F)$.
\item There exists $d>0$ such that the integral
$$\int_{Z_{H}(F)\back H_{\bar{Q}}(F)} \Xi^{G_{\bar{Q}}}(h_{\bar{Q}})\sigma_0(h_{\bar{Q}})^{-d} \delta_{H_{\bar{Q}}}(h_{\bar{Q}})^{1/2} dh_{\bar{Q}}$$
is absolutely convergent. Moreover, if we are in the $(G_0,H_0)$-case (this means we replace the pair $(G,H)$ in the statement by the pair $(G_0,H_0)$), for all $d>0$, the integral
$$\int_{Z_{H}(F)\back H_{\bar{Q}}(F)} \Xi^{G_{\bar{Q}}}(h_{\bar{Q}})\sigma_0(h_{\bar{Q}})^d \delta_{H_{\bar{Q}}}(h_{\bar{Q}})^{1/2} dh_{\bar{Q}}$$
is absolutely convergent.
\item Let $\bar{P}_{min}=M_{min}\bar{U}_{min}\subset \bar{Q}$ be a good minimal parabolic subgroup of $G$, let $A_{min}=A_{M_{min}}$ and $A_{min}^{+}=\{ a\in A_{min}(F)|\; | \alpha(a)| \geq 1, \forall \alpha\in \Psi(A_{min},\bar{P}_{min})\}$. Then there exists $d>0$ such that
$$\int_{Z_{H}(F)\back H_{\bar{Q}}(F)} \Xi^{G_{\bar{Q}}}(a^{-1}h_{\bar{Q}}a)\sigma_0(a^{-1}h_{\bar{Q}}a)^{-d} \delta_{H_{\bar{Q}}}(h_{\bar{Q}})^{1/2} dh_{\bar{Q}} \ll \Xi^{G_{\bar{Q}}}(a)^2$$
for all $a\in A_{min}^{+}$
\end{enumerate}
\end{lem}

\begin{proof}
(1) $H_{\bar{Q}}\cap U_{\bar{Q}}=\{1\}$ just follows from Proposition \ref{spherical}. For the second part, we only need to show that
$$\det(Ad(h_{\bar{Q}})\mid_{\bar{\Fq}/\Fh_{\bar{Q}}})=1$$
for all $h_{\bar{Q}}\in H_{\bar{Q}}(F)$. Since $\bar{Q}$ is a good parabolic subgroup, $\bar{\Fq}+\Fh=\Fg$ and $\Fh_{\bar{Q}}=\Fh\cap \bar{\Fq}$, so we have an isomorphism $\bar{\Fq}/\Fh_{\bar{Q}}\simeq \Fg/\Fh$, this implies
$$\det(Ad(h_{\bar{Q}})\mid_{\bar{\Fq}/\Fh_{\bar{Q}}})=\det(Ad(h_{\bar{Q}})\mid_{\Fg/\Fh})=\det(Ad(h_{\bar{Q}})\mid_{\Fg})\det(Ad(h_{\bar{Q}})\mid_{\Fh})^{-1}.$$
Since $G$ and $H$ are unimodular, $\det(Ad(h_{\bar{Q}})\mid_{\Fg})=\det(Ad(h_{\bar{Q}})\mid_{\Fh})=1$ for all $h_{\bar{Q}}\in H_{\bar{Q}}(F)$. This finishes the proof of (1).
\\
\\
(2) By Proposition \ref{spherical}, we can find a good minimal parabolic subgroup $\bar{P}_{min}=M_{min}\bar{U}_{min}\subset \bar{Q}$. Let $L$ be the Levi subgroup of $Q$ containing $M_{min}$, we have $L\simeq G_{\bar{Q}}$. Let $K$ be a hyperspecial maximal compact subgroup of $G(F)$ with respect to $L$, and $K_L=K\cap L(F)$. Define $\tau=I_{\bar{P}_{min}\cap L}^{L}(1)$ and $\pi=I_{\bar{Q}}^{G}(\tau)=I_{\bar{P}_{min}}^{G}(1)$. Let $(\;,\;)$ (resp. $(\;,\;)_{\tau}$) be the inner product on $\pi$ (resp. $\tau$). We fix $e_K\in \pi^{\infty}$ (resp. $e_{K_L}\in \tau^{\infty}$) to be the unique $K$-invariant(resp. $K_L$-invariant) vector. Then by the definition of Harish-Chandra function, we may assume that
\begin{equation}\label{major 5.1}
\Xi^G(g)=(\pi(g)e_K,e_K), \Xi^L(l)=(\tau(l)e_{K_L},e_{K_L})_{\tau}, g\in G(F),l\in L(F).
\end{equation}
So by choosing a suitable Haar measure, we have
$$\Xi^G(g)=\int_{\bar{Q}(F)\back G(F)} (e_K(g'g),e_K(g'))_{\tau} dg'.$$
Since $\bar{Q}$ is a good parabolic, by part(1) and Proposition \ref{spherical}(1), we have
$$\int_{\bar{Q}(F)\back G(F)}\varphi(g)dg=\int_{H_{\bar{Q}}(F)\back H(F)} \varphi(h)dh$$
for all $\varphi\in L^1(\bar{Q}(F)\back G(F),\delta_{\bar{Q}})$. So for all $g\in G(F)$, we have
$$\Xi^G(g)=\int_{H_{\bar{Q}}(F)\back H(F)} (e_K(hg),e_K(h))_{\tau} dh.$$
By Lemma \ref{major 2}(2), there exists $d>0$ such that the integral
\begin{eqnarray*}
&&\int_{\zh} \Xi^G(h)\sigma_0(h)^{-d} dh \\
&=&\int_{\zh}\int_{H_{\bar{Q}}(F)\back H(F)}(e_K(h'h),e_K(h'))_{\tau} \sigma_0(h)^{-d}dh' dh
\end{eqnarray*}
converges. Since $(e_K(h'h),e_K(h'))_{\tau}$ equals to some value of $\Xi^L$, it is positive, so the double integral above is absolutely convergent. By switching the order of the integral, changing the variable $h\mapsto h'^{-1}h$ and decomposing the integral over $\zh$ as a double integral over $H_{\bar{Q}}\back H(F)$ and $Z_H(F)\back H_{\bar{Q}}(F)$, we know the integral
$$\int_{(H_{\bar{Q}}(F)\back H(F))^2} \int_{Z_H(F)\back H_{\bar{Q}}(F)} (\tau(h_{\bar{Q}})e_K(h),e_K(h'))_{\tau} \sigma_0(h'^{-1} h_{\bar{Q}}h)^{-d} \delta_{H_{\bar{Q}}}(h_{\bar{Q}})^{1/2} dh_{\bar{Q}}dh dh'$$
is absolutely convergent. Here we also use the fact that $\delta_{\bar{Q}}(h_{\bar{Q}})=\delta_{H_{\bar{Q}}}(h_{\bar{Q}})$. Then by Fubini Theorem, there exist $h,h'\in H(F)$ such that the integral
$$\int_{Z_H(F)\back H_{\bar{Q}}(F)} (\tau(h_{\bar{Q}})e_K(h),e_K(h'))_{\tau} \sigma_0(h'^{-1} h_{\bar{Q}}h)^{-d} \delta_{H_{\bar{Q}}}(h_{\bar{Q}})^{1/2} dh_{\bar{Q}}$$
is absolutely convergent. Let $h=luk,h'=l'u'k'$ be the Iwasawa decomposition with $l,l'\in L(F),u,u'\in U_{\bar{Q}}(F)$ and $k,k'\in K$. Then by \eqref{major 5.1}, for all $h_{\bar{Q}}\in H_{\bar{Q}}(F)$, we have
$$(\tau(h_{\bar{Q}})e_K(h),e_K(h'))_{\tau}=\delta_{\bar{Q}}(l'l)^{1/2}\Xi^L(l'^{-1} h_{\bar{Q}}l).$$
For the given $h,h',l,l'$ as above, $\Xi^L(h_{\bar{Q}})\ll \Xi^L(l'^{-1}h_{\bar{Q}}l)$ and $\sigma_0(h'^{-1}h_{\bar{Q}}h)\ll \sigma_0(h_{\bar{Q}})$ for all $h_{\bar{Q}}\in H_{\bar{Q}}(F)$, so the integral
$$\int_{Z_H(F)\back H_{\bar{Q}}(F)} \Xi^L(h_{\bar{Q}}) \sigma_0(h_{\bar{Q}})^{-d} \delta_{H_{\bar{Q}}}(h_{\bar{Q}})^{1/2} dh_{\bar{Q}}$$
is absolutely convergent. This finishes the first part of (2) since $\Xi^L=\Xi^{G_{\bar{Q}}}$. The second part of (2) just follows from the same argument except we use Lemma \ref{major 2}(1) instead of Lemma \ref{major 2}(2).
\\
\\
(3) By Proposition \ref{spherical}(3), for all $d>0$ and $a\in A_{min}^{+}$, we have
\begin{eqnarray*}
&&\int_{Z_{H}(F)\back H_{\bar{Q}}(F)} \Xi^{G_{\bar{Q}}}(a^{-1}h_{\bar{Q}}a)\sigma_0(a^{-1}h_{\bar{Q}}a)^{-d} \delta_{H_{\bar{Q}}}(h_{\bar{Q}})^{1/2} dh_{\bar{Q}}\\
&\ll& \int_{Z_{H}(F)\back H_{\bar{Q}}(F)} \Xi^{G_{\bar{Q}}}(a^{-1}h_{\bar{Q}}a)\sigma_0(h_{\bar{Q}})^{-d} \delta_{H_{\bar{Q}}}(h_{\bar{Q}})^{1/2} dh_{\bar{Q}}. \end{eqnarray*}
So we only need to prove that there exists $d>0$ such that for all $a\in A_{min}^{+}$, we have
$$\int_{Z_{H}(F)\back H_{\bar{Q}}(F)} \Xi^{G_{\bar{Q}}}(a^{-1}h_{\bar{Q}}a)\sigma_0(h_{\bar{Q}})^{-d} \delta_{H_{\bar{Q}}}(h_{\bar{Q}})^{1/2} dh_{\bar{Q}} \ll \Xi^{G_{\bar{Q}}}(a)^2.$$
Let $\bar{P}_{min,\bar{Q}}$ be the image of $\bar{P}_{min}$ under the projection $Q\rightarrow G_{\bar{Q}}$, it is a minimal parabolic subgroup of $G_{\bar{Q}}$ and $\bar{P}_{min,\bar{Q}} H_{\bar{Q}}$ is open in $G_{\bar{Q}}$. By applying the same argument as in the proof of Proposition \ref{major 4}(5), we can show that
\begin{eqnarray*}
&&\int_{Z_{H}(F)\back H_{\bar{Q}}(F)} \Xi^{G_{\bar{Q}}}(a^{-1}h_{\bar{Q}}a)\sigma_0(h_{\bar{Q}})^{-d} \delta_{H_{\bar{Q}}}(h_{\bar{Q}})^{1/2} dh_{\bar{Q}}\\
&\ll& \Xi^{G_{\bar{Q}}}(a)^2\int_{Z_{H}(F)\back H_{\bar{Q}}(F)} \Xi^{G_{\bar{Q}}}(h_{\bar{Q}})\sigma_0(h_{\bar{Q}})^{-d} \delta_{H_{\bar{Q}}}(h_{\bar{Q}})^{1/2} dh_{\bar{Q}}
\end{eqnarray*}
for all $a\in A_{min}^{+}$. Then we just need to choose $d>0$ large so that part (2) holds. This finishes the proof of (3).
\end{proof}

\section{Explicit intertwining operator}
In this section, we study an explicit element $\CL_{\pi}$ in the Hom space given by the (normalized) integral of the matrix coefficients. The main result of this section is to show that the Hom space is nonzero if and only if $\CL_{\pi}\neq 0$ (i.e. Theorem \ref{thm 1}). In Sections 5.1 and 5.2, we define $\CL_{\pi}$ and prove some basic properties of it. In Sections 5.3 and 5.4, we study the behavior of $\CL_{\pi}$ under parabolic induction. Since we can not always reduce to the strongly tempered case, we have to treat the p-adic case and the archimedean case separately. In Section 5.5, we prove Theorem \ref{thm 1}. Then in Section 5.6, we discuss some applications of Theorem \ref{thm 1}, which are Corollary \ref{parabolic induction 6} and Corollary \ref{parabolic induction 5}. These two results will play essential roles in our proof of the main results of this paper.
\subsection{A normalized integral}
Let $\eta=\chi^2$ be two unitary characters of $F^{\times}$. In Section 1, we define the character $\omega$ and $\xi$ on $H(F)$. By Lemma \ref{major 2}, for all $f\in \CC(\zg,\eta^{-1})$, the integral
$$\int_{\zh}f(h)\xi(h)\omega(h)dh$$
is absolutely convergent and defines a continuous linear form on the space $\CC(\zg,\eta^{-1})$, which can be extended to a
linear form on $\CC^w(\zg,\eta^{-1})$ as follows.

\begin{prop}\label{major 8}
The linear form
$$f\in \CC(\zg,\eta^{-1}) \rightarrow \int_{\zh}f(h)\xi(h)\omega(h)dh$$
can be extended continuously to $\CC^w(\zg,\eta^{-1})$.
\end{prop}

\begin{proof}
Let $a:\BG_m(F)\rightarrow Z_{G_0}(F)$ be a homomorphism defined by $a(t)=diag(t,t,1,1,t^{-1},t^{-1})$ in the split case, and $a(t)=diag(t,1,t^{-1})$ in the non-split case. Then we know $\lambda(a(t)ha(t)^{-1})=t\lambda(h)$ for all $h\in H(F)$ and $t\in \BG_m(F)$.

\textbf{If F is p-adic}, fix an open compact subgroup $K\subset G(F)$ (not necessarily maximal), it is enough to prove that the linear form
$$f\in \CC_K(\zg,\eta^{-1}) \rightarrow \int_{\zh}f(h)\xi(h)\omega(h)dh$$
extends continuously to $\CC_{K}^{w}(\zg,\eta^{-1})$ for all $K$. Here we define $\CC_{K}^{w}(\zg,\eta^{-1})$ to be the space of bi-$K$-invariant elements in $\CC^w(\zg,\eta^{-1})$. Let $K_a=a^{-1}(K\cap Z_{G_0}(F))$, it is an open compact subset of $F^{\times}$. Then for $f\in \CC_K(\zg,\eta^{-1})$, we have
\begin{eqnarray*}
&&\int_{\zh} f(h)\xi(h)\omega(h)dh\\
&=&mes(K_a)^{-1} \int_{K_a}\int_{\zh} f(a(t)^{-1}ha(t))\xi(h) \omega(a(t)^{-1}ha(t)) dh d^{\times}t\\
&=&mes(K_a)^{-1}\int_{\zh} f(h) \omega(h)\int_{K_a}\xi(a(t)ha(t)^{-1})d^{\times}t dh\\
&=&mes(K_a)^{-1}\int_{\zh} f(h)\omega(h) \int_{K_a}\psi(t\lambda(h))\mid t\mid^{-1}  dt dh.
\end{eqnarray*}
The function $x\in F\rightarrow \int_{K_a} \psi(tx)\mid t\mid^{-1} dt$ is the Fourier transform of the function $\mid \cdot\mid^{-1} 1_{K_a}\in C_{c}^{\infty}(F)$, so it also belongs to $C_{c}^{\infty}(F)$. So the last integral above is essentially bounded by
$$\int_{\zh} |f(h)| (1+\mid\lambda(h)\mid)^{-\delta} dh$$
for all $\delta>0$. Then by applying Lemma \ref{major 2}, we know the integral above is also absolutely convergent for $f\in \CC_{K}^{w}(\zg,\eta^{-1})$. Thus the linear form can be extended continuously to $\CC_{K}^{w}(\zg,\eta^{-1})$.

\textbf{If F=$\BR$}, recall that for $g\in G(F)$ and $f\in C^{\infty}(G(F))$, we have defined ${}^g f(x)=f(g^{-1}xg)$. Let $Ad_a$ be a smooth representation of $F^{\times}$ on $\CC^{w}(\zg,\eta^{-1})$ given by
$Ad_a(t)(f)={}^{a(t)} f.$
This induces an action of $\CU(\Fg\Fl_1(F))$ on $\CC^{w}(\zg,\eta^{-1})$, which is still denoted by $Ad_a$. Let $\Delta=1-(t\frac{d}{dt})^2\in \CU(\Fg\Fl_1(F))$. By elliptic regularity (see Lemma 3.7 of \cite{BK14}), for all integer $m\geq 1$, there exist $\varphi_1\in C_{c}^{2m-2}(F^{\times})$ and $\varphi_2\in C_{c}^{\infty}(F^{\times})$ such that
$\varphi_1 \ast \Delta^m+\varphi_2=\delta_1.$
This implies
$$Ad_a(\varphi_1)Ad_a(\Delta^m)+Ad_a(\varphi_2)=Id.$$
Therefore for all $f\in \CC(\zg,\eta^{-1})$, we have
\begin{eqnarray*}
&&\int_{\zh} f(h)\xi(h)\omega(h)dh\\
&=&\int_{\zh} (Ad_a(\varphi_1)Ad_a(\Delta^m)f)(h)\xi(h) \omega(h)dh \\
&&+\int_{\zh}(Ad_a(\varphi_2)f)(h)\xi(h)\omega(h)dh\\
&=&\int_{\zh}(Ad_a(\Delta^m)f)(h) \omega(h)\int_{F^{\times}} \varphi_1(t)\xi(a(t)ha(t)^{-1})\delta_P(a(t))d^{\times}t dh\\
&&+\int_{\zh} f(h)\omega(h) \int_{F^{\times}} \varphi_2(t)\xi(a(t)ha(t)^{-1})\delta_P(a(t))d^{\times}t dh\\
&=&\int_{\zh}(Ad_a(\Delta^m)f)(h) \omega(h)\int_{F^{\times}} \varphi_1(t)\psi(t\lambda(h))\delta_P(a(t))\mid t\mid^{-1}dt dh\\
&&+\int_{\zh} f(h)\omega(h) \int_{F^{\times}} \varphi_2(t)\psi(t\lambda(h))\delta_P(a(t)) \mid t\mid^{-1} dt dh.
\end{eqnarray*}
Here the second equation is to take the transform $h\mapsto a(t)^{-1}ha(t)$ in both integrals and the extra $\delta_P(a(t))$ is its Jacobian.
For $i=1,2$, the functions $f_i: x\in F\rightarrow \int_{F}\varphi_i(t)\delta_P(a(t))|t|^{-1} \psi(tx)dt$ are the Fourier transforms of the functions $t\rightarrow \varphi_i(t)\delta_P(a(t)) |t|^{-1} \in C_{c}^{2m-2}(F)$. Hence $f_1$ and $f_2$ are essentially bounded by
$(1+|x|)^{-2m+2}$. By applying Lemma \ref{major 2} again, we know for all $m\geq 2$, the last two integrals above are absolutely convergent for all $f\in \CC^{w}(\zg,\eta^{-1})$. Therefore the linear form can be extend continuously to $\CC^{w}(\zg,\eta^{-1})$.
\end{proof}

Denote by $\CP_{H,\xi}$ the continuous linear form on $\CC^{w}(\zg,\eta^{-1})$ defined above. i.e.
$$f\in \CC^{w}(\zg,\eta^{-1})\rightarrow \int_{\zh}^{\ast}f(h)\xi(h)\omega(h) dh.$$

\begin{lem}\label{linear form}
\begin{enumerate}
\item For all $f\in \CC^{w}(\zg,\eta^{-1})$, and $h_0,h_1\in H(F)$, we have
$$\CP_{H,\xi}(L(h_0)R(h_1) f)=\xi(h_0)\omega(h_0)\xi(h_1)^{-1}\omega(h_1)^{-1} \CP_{H,\xi}(f)$$
where $R$ (resp. $L$) is the right (resp. left) translation.
\item Let $\varphi\in C_{c}^{\infty}(F^{\times})$, and set $\varphi'(t)=\mid t\mid^{-1}\delta_P(a(t))\varphi(t)$, then both $\varphi$ and $\varphi'$ can be viewed as elements in $C_{c}^{\infty}(F)$. Let $\hat{\varphi'}$ be the Fourier transform of $\varphi'$ with respect to $\psi$. Then we have
$$\CP_{H,\xi}(Ad_a(\varphi)f)=\int_{\zh} f(h) \omega(h) \hat{\varphi'}(\lambda(h)) dh$$
for all $f\in \CC^{w}(\zg,\eta^{-1})$. Note that the last integral is absolutely convergent by Lemma \ref{major 2}
\end{enumerate}
\end{lem}

\begin{proof}
Since both sides of the equality is continuous in $\CC^{w}(\zg)$, it is enough to check (1) and (2) for $f\in \CC(\zg,\eta^{-1})$. In this case, $\CP_{H,\xi}(f)=\int_{\zh} f(h)\xi(h)\omega(h)dh$. Then (1) follows from change variables in the integral. For (2), we have
\begin{eqnarray*}
\CP_{H,\xi}(Ad_a(\varphi)f)&=&\int_{\zh} Ad_a(\varphi)(f) \xi(h)\omega(h) dh\\
&=&\int_{\zh} f(h)\omega(h)\int_{F^{\times}} \varphi(t)\xi(a(t)ha(t)^{-1}) \delta_P(a(t)) d^{\times}t dh\\
&=&\int_{\zh} f(h)\omega(h)\int_F \varphi(t)\psi(t\lambda(h)) \delta_P(a(t))|t|^{-1} dtdh\\
&=&\int_{\zh} f(h) \omega(h) \hat{\varphi'}(\lambda(h)) dh.
\end{eqnarray*}
This finishes the proof of the Lemma.
\end{proof}

\subsection{Definition and properties of $\CL_{\pi}$}
Let $\pi$ be a tempered representation of $G(F)$ with central character $\eta$. For all $T\in End(\pi)^{\infty}$, define
$$\CL_{\pi}(T)=\CP_{H,\xi}(Trace(\pi(g^{-1})T))=\int_{\zh}^{\ast} Trace(\pi(h^{-1})T)\xi(h)\omega(h)dh.$$
By Proposition \ref{major 8}, together with the fact that the map $T\in End(\pi)^{\infty}\rightarrow (g\rightarrow Trace(\pi(g^{-1}) T)\in \CC^w(\zg,\eta^{-1})$ is continuous, we know $\CL_{\pi}: End(\pi)^{\infty}\rightarrow \BC$ is a continuous linear form. By Lemma \ref{linear form}, for any $h,h'\in H(F)$, we have
\begin{equation}\label{L 1}
\CL_{\pi}(\pi(h)T\pi(h'))=\xi(hh')\omega(hh')\CL_{\pi}(T).
\end{equation}
For $e\in \pi^{\infty},e'\in \bar{\pi}^{\infty}$, define $T_{e,e'}\in End(\pi)^{\infty}$ to be
$e_0\in \pi\mapsto (e_0,e')e.$
Set $\CL_{\pi}(e,e')=\CL_{\pi}(T_{e,e'})$, then we have
$$\CL_{\pi}(e,e')=\int_{\zh}^{\ast} (e,\pi(h)e')\omega(h)\xi(h) dh.$$
If we fix $e'$, by \eqref{L 1}, the map $e\in \pi^{\infty}\rightarrow \CL_{\pi}(e,e')$ belongs to $Hom_H(\pi^{\infty},\omega\otimes \xi)$. Since $Span \{T_{e,e'}\mid e\in \pi^{\infty},e'\in \bar{\pi}^{\infty}\}$ is dense in $End(\pi)^{\infty}$ (in p-adic case, they are equal), we have that
$\CL_{\pi}\neq 0\Rightarrow m(\pi)\neq 0.$
The purpose of this section is to prove the other direction.
\begin{thm}\label{thm 1}
For all $\pi\in Temp(G,\eta)$, we have
$$\CL_{\pi}\neq 0\iff m(\pi)\neq 0.$$
\end{thm}

Our proof for this result is based on the method developed by Waldspurger (\cite[Proposition 5.7]{W12}) and by Beuzart-Plessis
(\cite[Theorem 8.2.1]{B15}) for the GGP models. See also \cite[Theorem 6.2.1]{SV12}. The key ingredient in the proof is the Plancherel formula together with the fact that the nonvanishing property of $\CL_{\pi}$ is invariant under the parabolic induction and the unramified twist. For the rest of this section, we discuss some basic properties of $\CL_{\pi}$.

The operator $\CL_{\pi}$ defines a continuous linear map
$$L_{\pi}:\pi^{\infty}\rightarrow \bar{\pi}^{-\infty},\; e\rightarrow \CL_{\pi}(e,\cdot)$$
where $\bar{\pi}^{-\infty}$ is the topological dual of $\pi^{\infty}$ endowed with the strong topology. The image of $L_{\pi}$ belongs to $(\bar{\pi}^{-\infty})^{H,\xi}=Hom_H(\pi^{\infty},\xi)$. So if $\pi$ is irreducible, the image is of dimension less or equal to 1. Let $T\in End(\pi)^{\infty}$, it can be uniquely extended to a continuous operator $T:\bar{\pi}^{-\infty}\rightarrow \pi^{\infty}$. Then we have the following two operators which are both of finite rank:
$$TL_{\pi}:\pi^{\infty}\rightarrow \pi^{\infty}, \;L_{\pi}T:\bar{\pi}^{-\infty}\rightarrow \bar{\pi}^{-\infty}.$$
In particular, they are of trace class. It is easy to see that
\begin{equation}\label{8.1}
Trace(TL_{\pi})=Trace(L_{\pi}T)=\CL_{\pi}(T).
\end{equation}

\begin{lem}\label{lemma 1}
With the notation above, the following hold.
\begin{enumerate}
\item The map
$\pi \in \CX_{temp}(G,\eta)\rightarrow L_{\pi}\in Hom(\pi^{\infty},\bar{\pi}^{-\infty})$
is smooth in the following sense: For all parabolic subgroup $Q=LU_{Q}$ of $G$, $\sigma\in \Pi_2(L)$, and for all maximal compact subgroup $K$ of $G(F)$, the map
$\lambda\in i\Fa_{L,0}^{\ast}\rightarrow \CL_{\pi_{\lambda}}\in End(\pi_{\lambda})^{-\infty}\simeq End(\pi_K)^{-\infty}$
is smooth, here $\pi_{\lambda}=I_{Q}^{G}(\sigma_{\lambda})$ and $\pi_K=I_{Q\cap K}^{K}(\sigma)$.
\item For $\pi\in Temp(G,\eta)$, and for all $S,T\in End(\pi)^{\infty}$, we have $SL_{\pi}T\in End(\pi)^{\infty}$, and
$\CL_{\pi}(S)\CL_{\pi}(T)=\CL_{\pi}(SL_{\pi} T).$
\item Let $S,T\in \CC(Temp(G,\eta))$, then the section $\pi\in Temp(G,\eta)\mapsto S_{\pi}L_{\pi}T_{\pi}\in End(\pi)^{\infty}$ belongs to $C^{\infty}(Temp(G,\eta))$.
\item Let $f\in \CC(\zg,\eta^{-1})$ and assume that its Fourier transform $\pi\in Temp(G,\eta)\rightarrow \pi(f)$ is compactly supported (this is always true in p-adic case). Then we have
$$\int_{\zh}f(h)\xi(h) \omega(h)dh=\int_{Temp(G,\eta)} \CL_{\pi}(\pi(f))\mu(\pi) d\pi$$
with both integrals being absolutely convergent.
\item For $f\in \CC(\zg,\eta^{-1})$ and $f'\in \CC(\zg,\eta)$, assume that the Fourier transform of $f$ is compactly supported. Then we have
\begin{eqnarray*}
&&\int_{\CX_{temp}(G,\eta)} \CL_{\pi}(\pi(f)) \overline{\CL_{\pi}(\pi(\bar{f'}))} \mu(\pi) d\pi\\
&=&\int_{\zh}\int_{\zh}\int_{\zg} f(hgh')f'(g)dg \xi(h')\omega(h') dh'\xi(h)\omega(h)dh
\end{eqnarray*}
where the left hand side is absolutely convergent and the right hand side is convergent in that order but is not necessarily absolutely convergent.
\end{enumerate}
\end{lem}

\begin{proof}
(1), (2) and (3) follow from the same argument as Lemma 8.2.1 of \cite{B15}, we will skip it here. The proof of (4) and (5) is also similar to the loc. cit. (except that we need to take care of the center of the group), we include the proof for completion.

For (4), by Lemma \ref{major 2}, the left hand side is absolutely convergent. Since the Fourier transform of $f$ is compactly supported, the right hand side is also absolutely convergent. Let $\varphi(f,\pi)(g)=Trace(\pi(g^{-1})\pi(f))$, which is a function in $\CC^w(\zg,\eta^{-1})$. By the Plancherel formula in Section 2.8, we have
$$f=\int_{Temp(G,\eta)} \varphi(f,\pi)\mu(\pi) d\pi.$$
By applying the operator $\CP_{H,\xi}$ on both side, we have
$$\CP_{H,\xi}(f)=\int_{Temp(G,\eta)} \CP_{H,\xi}(\varphi(f,\pi))\mu(\pi) d\pi.$$
This proves (4).

For (5), let $f'^{\vee}(g)=f'(g^{-1})$. Then right hand side is equal to
\begin{equation}\label{8.11}
\int_{\zh}\int_{\zh}(f'^{\vee}\ast L(h^{-1})f)(h') \xi(h')\omega(h')dh'\xi(h)\omega(h)dh.
\end{equation}
The Fourier transform of $f$ is compactly supported, so is $f'^{\vee}\ast L(h^{-1})f$. By applying part (4) to $f'^{\vee}\ast L(h^{-1})f$, we know the inner integral in \eqref{8.11} is absolutely convergent and we have
\begin{eqnarray*}
&&\int_{\zh}(f'^{\vee}\ast L(h^{-1})f)(h') \xi(h')\omega(h')dh' \\
&=&\int_{Temp(G,\eta)} \CL_{\pi}(\pi(f'^{\vee}) \pi(h^{-1})\pi(f)) \mu(\pi) d\pi \\
&=&\int_{Temp(G,\eta)} Trace(\pi(h^{-1})\pi(f)L_{\pi}\pi(f'^{\vee})) \mu(\pi)d\pi.
\end{eqnarray*}
The last equality is because of \eqref{8.1}. By part (3), the section $\pi\in Temp(G,\eta)\mapsto \pi(f)L_{\pi}\pi(f'^{\vee})$ is smooth,
and is also compactly supported, and hence it belongs to $\CC(Temp(G,\eta))$. By the matrical Paley-Wiener Theorem in Section 2.8, it is a Fourier transform of a Harish-Chandra-Schwartz function. Applying part (4) to such function, we know the exterior integral of \eqref{8.11} is absolutely convergent and the whole expression is equal to
$$\int_{Temp(G,\eta)} \CL_{\pi}(\pi(f)L_{\pi}\pi(f'^{\vee}))\mu(\pi) d\pi.$$
By part (2) and the fact that $\CL_{\pi}(\pi(f'^{\vee}))=\overline{\CL_{\pi}(\pi(\bar{f'}))}$, \eqref{8.11} is then equal to
$$\int_{\CX_{temp}(G,\eta)} \CL_{\pi}(\pi(f)) \overline{\CL_{\pi}(\pi(\bar{f'}))} \mu(\pi) d\pi.$$
This finishes the proof of the lemma.
\end{proof}

The next lemma is about the asymptotic properties for elements in $Hom_H(\pi,\omega\otimes\xi)$.
\begin{lem}\label{lemma 2}
\begin{enumerate}
\item Let $\pi$ be a tempered representation of $G(F)$ with central character $\eta$ and $l\in Hom_H(\pi,\omega\otimes\xi)$ be a continuous $(H,\omega\otimes\xi)$-equivariant linear form. Then there exist $d>0$ and a continuous semi-norm $\nu_d$ on $\pi$ such that
$$| l(\pi(x)e)| \leq \nu_d(e)\hc(x)\nor(x)^d$$
for all $e\in \pi$ and $x\in H(F)\back G(F)$.
\item For all $d>0$, there exist $d'>0$ and a continuous semi-norm $\nu_{d,d'}$ on $\CC_{d}^{w}(\zg,\eta^{-1})$ such that
$$\mid \CP_{H,\xi}(R(x)L(y)\varphi)\leq \nu_{d,d'}(\varphi) \hc(x)\hc(y)\nor(x)^{d'}\nor(y)^{d'}$$
for all $\varphi \in \CC_{d}^{w}(\zg,\eta^{-1})$ and $x,y\in H(F)\back G(F)$.
\end{enumerate}
\end{lem}

\begin{proof}
The proof for GGP model can also be applied to our model. See \cite[Lemma 8.3.1]{B15}. We refer the readers to \cite{Wan17} for details of the proof.
\end{proof}

\subsection{Parabolic induction for the p-adic case}
Assume that $F$ is p-adic in this section. Let $\pi$ be a tempered representation of $G(F)$ with central character $\eta$. There exists a parabolic subgroup $\bar{Q}=LU_{\bar{Q}}$ of $G$, together with a discrete series $\tau\in \Pi_2(L)$ such that $\pi=I_{\bar{Q}}^{G}(\tau)$. By Proposition \ref{spherical}, we may assume that $\bar{Q}$ is a good parabolic subgroup. We can further assume that the inner product on $\pi$ is given by
\begin{equation}\label{para 1}
(e,e')=\int_{Q(F)\back G(F)}(e(g),e'(g))_{\tau}dg,\; \forall e,e'\in  \pi=I_{Q}^{G}(\tau).
\end{equation}
Let $H_{\bar{Q}}=H\cap \bar{Q}$, for $T\in End(\tau)^{\infty}$, define
$$\CL_{\tau}(T_{\tau})=\int_{Z_H(F)\back H_{\bar{Q}}(F)} Trace(\tau(h_{\bar{Q}}^{-1})T)\delta_{H_{\bar{Q}}}(h_{\bar{Q}})^{1/2} \omega(h_{\bar{Q}}) \xi(h_{\bar{Q}}) dh_{\bar{Q}}.$$
The integral above is absolutely convergent by Proposition \ref{major 5}(2) together with the assumption that $\tau$ is a discrete series.

\begin{prop}\label{parabolic induction}
With the notation above, we have
$$\CL_{\pi}\neq 0 \iff \CL_{\tau}\neq 0.$$
\end{prop}

\begin{proof}
For $e,e'\in \pi^{\infty}$, by \eqref{para 1}, we have
$$\CL_{\pi}(e,e')=\int_{\zh}^{\ast}\int_{\bar{Q}(F)\back G(F)} (e(g),e'(gh))_{\tau}dg \omega(h)\xi(h)dh.$$
Same as in previous sections, let $a:\BG_m(F)\rightarrow Z_{G_0}(F)$ be a homomorphism defined by $a(t)=diag(t,t,1,1,t^{-1},t^{-1})$ in the split case, and $a(t)=diag(t,1,t^{-1})$ in the non-split case. Since $e,e'\in \pi^{\infty}$, there exists an open compact subgroup $K_0$ of $G(F)$ such that the functions $e,e': G(F)\rightarrow \tau$ is bi-$K_0$-invariant. Let $K_a=a^{-1}(K_0\cap Z_{G_0}(F))\subset F^{\times}$, it is an open compact subset. By Proposition \ref{major 8}, we have
\begin{eqnarray}\label{para 2}
\CL_{\pi}(e,e')&=&\int_{\zh}^{\ast}\int_{\bar{Q}(F)\back G(F)} (e(g),e'(gh))_{\tau}dg \xi(h) \omega(h)dh \nonumber\\
&=&meas(K_a)^{-1} \int_{\zh}\int_{\bar{Q}(F)\back G(F)} (e(g),e'(gh))_{\tau}dg \\
&&\times \int_{K_a} \psi(t\lambda(h)) \mid t\mid^{-1} dt \omega(h)dh.\nonumber
\end{eqnarray}
By the same proposition, the last two integrals $\int_{\zh}\int_{\bar{Q}(F)\back G(F)}$ above is absolutely convergent. Since $\bar{Q}$ is a good parabolic subgroup, by Proposition \ref{spherical}, we can choose the Haar measures compatibly so that for all $\varphi\in L_1(\bar{Q}(F)\back G(F),\delta_{\bar{Q}})$, we have
$$\int_{\bar{Q}(F)\back G(F)}\varphi(g) dg=\int_{H_{\bar{Q}}(F)\back H(F)}\varphi(h)dh.$$
Then \eqref{para 2} becomes
\begin{eqnarray*}
\CL_{\pi}(e,e')&=&meas(K_a)^{-1} \int_{\zh}\int_{H_{\bar{Q}}(F)\back H(F)} (e(h'),e'(h'h))_{\tau}dh' \\
&&\times \int_{K_a} \psi(t\lambda(h)) \mid t\mid^{-1} dt \omega(h)dh.
\end{eqnarray*}
The integral $\int_{\zh}\int_{\bar{Q}(F)\back G(F)}$ above is absolutely convergent because of \eqref{para 2}. By switching the two integrals, making the transform $h\rightarrow h'h$ and decomposing $\int_{\zh}$ as $\int_{H_{\bar{Q}}(F)\back H(F)} \int_{Z_H(F)\back H_{\bar{Q}}(F)}$, we have
$$\CL_{\pi}(e,e')=meas(K_a)^{-1}\int_{(H_{\bar{Q}}(F)\back H(F))^2} f(h,h') dh dh'$$
where
\begin{eqnarray}\label{para 3}
f(h,h')&=&\int_{Z_H(F)\back H_{\bar{Q}}(F)}  (e(h),e'(h_{\bar{Q}}h'))_{\tau} \omega(h_{\bar{Q}})\omega(h^{-1}h') \nonumber\\
&&\times \int_{K_a} \psi(t\lambda(h'))\psi(t\lambda(h_{\bar{Q}}))\psi(-t\lambda(h)) \mid t\mid^{-1} dt dh_{\bar{Q}} \\
&=& \int_{Z_H(F)\back H_{\bar{Q}}(F)} \delta_{H_{\bar{Q}}}(h_{\bar{Q}})^{1/2} (e(h),\tau(h_{\bar{Q}})e'(h'))_{\tau} \omega(h_{\bar{Q}})\omega(h^{-1}h') \nonumber\\
&&\times \int_{K_a} \psi(t\lambda(h'))\psi(t\lambda(h_{\bar{Q}}))\psi(-t\lambda(h)) \mid t\mid^{-1} dt dh_{\bar{Q}}. \nonumber
\end{eqnarray}
Here we use the equation $\delta_{H_{\bar{Q}}}(h_{\bar{Q}})=\delta_{\bar{Q}}(h_{\bar{Q}})$ in the second equality. We first show that the integral \eqref{para 3} is absolutely convergent for any $h,h'\in H_{\bar{Q}}(F)\back H(F)$. In fact, since $K_a$ is compact, it is enough to show that for any $h,h'\in H_{\bar{Q}}(F)\back H(F)$, the integral
$$\int_{Z_H(F)\back H_{\bar{Q}}(F)} \delta_{H_{\bar{Q}}}(h_{\bar{Q}})^{1/2} (e(h),\tau(h_{\bar{Q}})e'(h'))_{\tau} dh_{\bar{Q}}$$
is absolutely convergent. This just follows from Proposition \ref{major 5}(2) together with the assumption that $\tau$ is discrete series. Then by switching the two integrals in \eqref{para 3}, we have
\begin{eqnarray*}
f(h,h')&=&\int_{K_a} \int_{Z_H(F)\back H_{\bar{Q}}(F)} \delta_{H_{\bar{Q}}}(h_{\bar{Q}})^{1/2} (e(h),\tau(h_{\bar{Q}})e'(h'))_{\tau} \omega(h_{\bar{Q}})\\
&&\times\; \psi(t\lambda(h_{\bar{Q}})) dh_{\bar{Q}}  \omega(h^{-1}h') \psi(t\lambda(h'))\psi(-t\lambda(h)) \mid t\mid^{-1} dt.
\end{eqnarray*}
By changing the variable $h_{\bar{Q}}\rightarrow a(t)h_{\bar{Q}}a(t)^{-1}$ in the inner integral (note that the Jacobian of such transform is 1 since $a(t)\in K_0$), we have
\begin{eqnarray*}
&&\int_{Z_H(F)\back H_{\bar{Q}}(F)} \delta_{H_{\bar{Q}}}(h_{\bar{Q}})^{1/2} (e(h),\tau(h_{\bar{Q}})e'(h'))_{\tau} \omega(h_{\bar{Q}}) \psi(t\lambda(h_{\bar{Q}})) dh_{\bar{Q}}\\
&=&\int_{Z_H(F)\back H_{\bar{Q}}(F)} \delta_{H_{\bar{Q}}}(h_{\bar{Q}})^{1/2} (e(h),\tau(a(t)^{-1}h_{\bar{Q}}a(t))e'(h'))_{\tau} \omega(h_{\bar{Q}}) \psi(\lambda(h_{\bar{Q}})) dh_{\bar{Q}}\\
&=&\int_{Z_H(F)\back H_{\bar{Q}}(F)} \delta_{H_{\bar{Q}}}(h_{\bar{Q}})^{1/2} (e(h),\tau(h_{\bar{Q}})e'(h'))_{\tau} \omega(h_{\bar{Q}}) \psi(\lambda(h_{\bar{Q}})) dh_{\bar{Q}}\\
&=& \CL_{\tau}(e(h),e'(h')).
\end{eqnarray*}
Here we use the fact that $e'$ is bi-$K_0$-invariant. Then we have
$$f(h,h')=\int_{K_a} \CL_{\tau}(e(h),e(h')) \omega(h^{-1}h') \psi(t\lambda(h'))\psi(-t\lambda(h)) \mid t\mid^{-1} dt.$$
If $\CL_{\pi}(e,e')\neq 0$, there exist $h,h'\in H_{\bar{Q}}(F)\back H(F)$ such that $f(h,h')\neq 0$, and hence $\CL_{\tau}(e(h),e(h'))\neq0$. This proves
$\CL_{\pi}\neq 0 \Rightarrow \CL_{\tau}\neq 0.$

For the other direction, if $\CL_{\tau}\neq 0$, we can find $v_0,v_0'\in \tau^{\infty}$ such that $\CL_{\tau}(v_0,v_0')\neq 0$. We choose a small open subset $\CU\subset H_{\bar{Q}}(F)\back H(F)$ and let $s: \CU\rightarrow H(F)$ be an analytic section of the map $H(F)\rightarrow H_{\bar{Q}}(F)\back H(F)$. For $f,f'\in C_{c}^{\infty}(\CU)$, define $\varphi,\varphi'\in C_{c}^{\infty}(\CU,\tau^{\infty})$ to be $\varphi(h)=f(h)v_0,\varphi'(h)=f'(h)v_0'$, then set
$$e_{\varphi}(g)=\begin{array}{cc}\left\{ \begin{array}{ccl} \delta_{\bar{Q}}^{1/2}(l)\tau(l)\varphi(h) & if & g=lus(h)\; with\; l\in L(F),u\in U_{\bar{Q}}(F), h \in \CU \\ 0 & else & \\ \end{array}\right. \end{array}$$
This is an element of $\pi^{\infty}$. Similarly we can define $e_{\varphi'}$. Then by the above discussion, we have
$$\CL_{\pi}(e_{\varphi},e_{\varphi'})=meas(K_a)^{-1}\int_{(H_{\bar{Q}}(F)\back H(F))^2} f(h,h') dh dh'$$
where
$$f(h,h')=\int_{K_a} \CL_{\tau}(e_{\varphi}(h),e_{\varphi'}(h')) \omega(h^{-1}h') \psi(t\lambda(h'))\psi(-t\lambda(h)) \mid t\mid^{-1} dt.$$
Combining with the definition of $e_{\varphi}$ and $e_{\varphi'}$, we have
\begin{eqnarray*}
\CL_{\pi}(e_{\varphi},e_{\varphi'})&=&meas(K_a)^{-1}\CL_{\tau}(v_0,v_0')\\
&&\times \int_{\CU^2} \int_{K_a} f(h)\overline{f'(h')} \omega(s(h)^{-1}s(h')) \psi(t\lambda(s(h)))\psi(-t\lambda(s(h))) \mid t\mid^{-1} dt dhdh'.
\end{eqnarray*}
Now if we take $\CU$ small enough, we can choose suitable section $s: \CU\rightarrow H(F)$ such that for all $t\in K_a$ and $h\in s(\CU)$, we have $\psi(t\lambda(h))=\omega(h)=1$. Also by taking $K_0$ small, we may assume that $\mid t\mid=1$ for all $t\in K_a$. Then the integral above becomes
\begin{eqnarray*}
\CL_{\pi}(e_{\varphi},e_{\varphi'})&=&meas(K_a)^{-1}\CL_{\tau}(v_0,v_0')\int_{\CU^2} \int_{K_a} f(h)\overline{f'(h')} dt dhdh'\\
&=& \CL_{\tau}(v_0,v_0')\int_{\CU^2} f(h)\overline{f'(h')} dhdh'.
\end{eqnarray*}
Then we can easily choose $f$ and $f'$ so that $\CL_{\pi}(e_{\varphi},e_{\varphi'})\neq 0$. Therefore we have proved that
$\CL_{\tau}\neq 0\Rightarrow \CL_{\pi}\neq 0.$
\end{proof}

\subsection{Parabolic induction for $F=\BR$}
Assume that $F=\BR$ in this section. It is very hard to directly study any arbitrary parabolic induction because of the way we normalizing the integral. Instead, we first study the parabolic induction for $\bar{P}$, then study all other parabolic subgroups contained in $\bar{P}$. This is allowable since in the archimedean case, the discrete series only appear on $\GL_1(\BR)$, $\GL_2(\BR)$ and $\GL_1(D)$. Let $\pi$ be a tempered representation of $G$ with central character $\eta$. Since we are in archimedean case, there exists a tempered representation $\pi_0$ of $G_0$ such that $\pi=I_{\bar{P}}^{G}(\pi_0)$. We assume that the inner product on $\pi$ is given by
\begin{equation}
(e,e')=\int_{\bar{P}(F)\back G(F)}(e(g),e'(g))_{\pi_0}dg,\; e,e'\in  \pi=I_{\bar{P}}^{G}(\pi_0).
\end{equation}
For $T\in End(\pi_0)^{\infty}$, define
$$\CL_{\pi_0}(T)=\int_{Z_H(F)\back H_0(F)} Trace(\pi_0(h_{0}^{-1})T) \omega(h_0) dh_0.$$
The integral above is absolutely convergent by Lemma \ref{major 2}(1) together with the fact that $\pi_0$ is tempered.

\begin{prop}\label{parabolic induction 2}
With the notation above, we have
$$\CL_{\pi}\neq 0 \iff \CL_{\pi_0}\neq 0.$$
\end{prop}

\begin{proof}
For $e,e'\in \pi^{\infty}$, we have
$$\CL_{\pi}(e,e')=\int_{\zh}^{\ast}\int_{\bar{P}(F)\back G(F)} (e(g),e'(gh))dg \xi(h) \omega(h)dh.$$
Same as in Proposition \ref{major 8}, we can find $\varphi_1\in C_{c}^{2m-2}(F^{\times})$ and $\varphi_2\in C_{c}^{\infty}(F^{\times})$ such that
$\varphi_1\ast \Delta^m+\varphi_2=\delta_1$, and we have
\begin{eqnarray}\label{para 4}
\CL_{\pi}(e,e')&=&\int_{\zh} Ad_a(\Delta^m)(\int_{\bar{P}(F)\back G(F)} (e(g),e'(gh))dg)\nonumber \\
&&\times  \int_F \varphi_1(t) \delta_P(a(t))\mid t\mid^{-1}\psi(t\lambda(h)) \omega(h) dt dh \\
&&+\int_{\zh}  \int_{\bar{P}(F)\back G(F)} (e(g),e'(gh)) \nonumber\\
&&\times \int_F \varphi_2(t) \delta_P(a(t))\mid t\mid^{-1}\psi(t\lambda(h)) \omega(h) dt dg dh. \nonumber
\end{eqnarray}
Here $Ad_a(\Delta^m)$ acts on the function $\int_{\bar{P}(F)\back G(F)} (e(g),e'(gh))dg$ for the variable $h$, hence it commutes with the integral $\int_{\bar{P}(F)\back G(F)}$. Also since $\bar{P}$ is a good parabolic subgroup, by Proposition \ref{spherical}, we can choose Haar measure compatibly so that for all $\varphi\in L_1(\bar{P}(F)\back G(F),\delta_{\bar{P}})$, we have
$$\int_{\bar{P}(F)\back G(F)}\varphi(g) dg=\int_{U(F)}\varphi(h)dh.$$
Therefore \eqref{para 4} becomes
\begin{eqnarray*}
\CL_{\pi}(e,e')&=&\int_{\zh} \int_{U(F)}Ad_a(\Delta^m)( (e(u),e'(uh)) )du\\
&&\times  \int_F \varphi_1(t) \delta_P(a(t))\mid t\mid^{-1}\psi(t\lambda(h)) \omega(h) dt dh\\
&&+\int_{\zh}  \int_{U(F)} (e(u),e'(uh)) \\
&&\times \int_F \varphi_2(t) \delta_P(a(t))\mid t\mid^{-1}\psi(t\lambda(h)) \omega(h)dt du dh.
\end{eqnarray*}
Here $Ad_a(\Delta^m)$ acts on the function $(e(u),e'(uh))$ for the variable $h$. By changing the order of integration $\int_{\zh} \int_{U(F)}$ and decomposing the integral $\int_{\zh}$ by $\int_{U(F)}\int_{Z_H(F)\back H_0(F)}$ (this is allowable since the outer two integrals are absolutely convergent by Proposition \ref{major 8}), together with the fact that $Ad_a$ is the identity map on $H_0$, we have
\begin{eqnarray*}
\CL_{\pi}(e,e')&=&\int_{U(F)} \int_{U(F)}Ad_a(\Delta^m)(\CL_{\pi_0}(e(u),e'(uu'))) \varphi_1'(\lambda(u')) du'du \\
&&+\int_{U(F)}  \int_{U(F)} \CL_{\pi_0}(e(u),e'(uu')) \varphi_2'(\lambda(u')) du' du
\end{eqnarray*}
where
$\varphi_i '(s)=\int_F \varphi_i(t)\delta_P(a(t)) \mid t\mid^{-1} \psi(ts) dt$
is the Fourier transforms of the function $\varphi_i(t)\delta_P(a(t))\mid t\mid^{-1}$ for $i=1,2$. Here $Ad_a(\Delta^m)$ acts on the function $\CL_{\pi_0}(e(u),e'(uu'))$ for the variable $u'$. In particular, this implies $\CL_{\pi}\neq 0 \Rightarrow \CL_{\pi_0}\neq 0$.

For the other direction, if $\CL_{\pi_0}\neq 0$, we can choose $v_1,v_2\in \pi_{0}^{\infty}$ such that $\CL_{\pi_0}(v_1,v_2)\neq 0$. Choose $f_1,f_2\in C_{c}^{\infty}(U(F))$, for $i=1,2$, similarly as in the p-adic case, define
$$e_{f_i}(g)=\begin{array}{cc}\left\{ \begin{array}{ccl} \delta_{\bar{P}}(l)\pi_0(l)f_i(u)v_i & if & g=l\bar{u}u\; with\; l\in G_0(F),u\in U(F),\bar{u}\in \bar{U}(F) \\ 0 & else & \\ \end{array}\right. \end{array}$$
These are elements in $\pi^{\infty}$, and we have
\begin{eqnarray}\label{para 5}
\CL_{\pi}(e_{f_1},e_{f_2})&=&\int_{U(F)} \int_{U(F)} \CL_{\pi_0}(v_1,v_2) f_1(u) Ad_a(\Delta^m)( f_2(uu'))) \varphi_1 '(\lambda(u')) du' du\nonumber \\
&&+\int_{U(F)} \int_{U(F)} \CL_{\pi_0}(v_1,v_2) f_1(u)f_2(uu') \varphi_2 '(\lambda(u')) du'du.
\end{eqnarray}
Here $Ad_a(\Delta^m)$ acts on the function $f_2(uu')$ for the variable $u'$. Then we can easily find $f_1,f_2$ such that \eqref{para 5} is non-zero. This proves $\CL_{\pi_0}\neq 0 \Rightarrow \CL_{\pi}\neq 0$, and finishes the proof of the Proposition.
\end{proof}

Now for a tempered representation $\pi_0$ of $G_0(F)$ whose central character equals to $\eta$ when restricting on $Z_G$, we can find a good parabolic subgroup $\bar{Q}_0=L_0 U_0$ of $G_0(F)$ and a discrete series $\tau$ of $L_0$ such that $\pi_0=I_{\bar{Q}_0}^{G_0}(\tau)$. We still assume that the inner product on $\pi_0$ is given by
\begin{equation}\label{para 6}
(e,e')=\int_{\bar{Q}_0(F)\back H_0(F)}(e(g),e'(g))_{\tau}dg,\; e,e'\in  \pi_0=I_{\bar{Q}_0}^{G_0}(\tau).
\end{equation}
Let $H_{\bar{Q}}=H_0\cap \bar{Q}_0$, for $T\in End(\tau)^{\infty}$, define
$$\CL_{\tau}(T_{\tau})=\int_{Z_H(F)\back H_{\bar{Q}}(F)} Trace(\tau(h_{\bar{Q}}^{-1})T)\delta_{H_{\bar{Q}}}(h_{\bar{Q}})^{1/2} \omega(h_{\bar{Q}}) dh_{\bar{Q}}.$$
The integral above is absolutely convergent by Proposition \ref{major 5}(2) together with the assumption that $\tau$ is discrete series.

\begin{prop}\label{parabolic induction 3}
With the notation above, we have
$$\CL_{\pi_0}\neq 0 \iff \CL_{\tau}\neq 0.$$
\end{prop}

\begin{proof}
Since we are in $(G_0,H_0)$ case, the integral defining $\CL_{\pi_0}$ is absolutely convergent, together with \eqref{para 6}, we have
$$\CL_{\pi_0}(e,e')=\int_{Z_H(F)\back H_0(F)}\int_{\bar{Q}_0(F)\back G_0(F)} (e(g),e'(gh))_{\tau} \omega(h)dg dh.$$
The integral above is absolutely convergent by Lemma \ref{major 2}. Same as in the previous Propositions, the integral $\bar{Q}_0(F)\back G_0(F)$ can be replaced by $H_{\bar{Q}}(F)\back H_0(F)$, hence we have
$$\CL_{\pi_0}(e,e')=\int_{Z_H(F)\back H_0(F)}\int_{H_{\bar{Q}}(F)\back H_0(F)} (e(h'),e'(h'h))_{\tau} \omega(h) dh' dh.$$
By switching the two integrals, changing the variable $h\rightarrow h'h$ and decomposing the integral $\int_{Z_H(F)\back H_0(F)}$ by $\int_{H_{\bar{Q}}(F)\back H_0(F)} \int_{Z_H(F)\back H_{\bar{Q}}(F)}$, we have
$$\CL_{\pi}(e,e')=\int_{H_{\bar{Q}}(F)\back H_0(F)}\int_{H_{\bar{Q}}(F)\back H_0(F)} \CL_{\tau}(e(h),e'(h')) \omega(h)^{-1} \omega(h') dh dh'.$$
This proves $\CL_{\pi_0}\neq 0\Rightarrow \CL_{\tau}\neq 0$.

For the other direction, if $\CL_{\tau}\neq 0$, there exist $v_1,v_2\in \tau^{\infty}$ such that $\CL_{\tau}(v_1,v_2)\neq 0$. Let $s:\CU\rightarrow H_0(F)$ be an analytic section over an open subset $\CU$ of $H_{\bar{Q}}(F)\back H(F)$ of the map $H(F)\rightarrow H_{\bar{Q}}(F)\back H(F)$. Choose $f_1,f_2\in C_{c}^{\infty}(\CU)$, for $i=1,2$, define
$$e_{f_i}(g)=\begin{array}{cc}\left\{ \begin{array}{ccl} \delta_{\bar{Q}}(l)\tau(l)f_i(h)v_i & if & g=lus(h)\; with\; l\in L_0(F),u\in U_0(F),h\in \CU \\ 0 & else & \\ \end{array}\right. \end{array}$$
These are elements in $\pi_{0}^{\infty}$, and we have
$$\CL_{\pi_0}(e_{f_1},e_{f_2})=\int_{\CU} \int_{\CU} f_1(h) \overline{f_2}(h') \omega(s(h))^{-1} \omega(s(h')) \CL_{\tau}(v_1,v_2) dh dh'.$$
Then we can easily choose $f_1,f_2$ such that $\CL_{\pi_0}(e_{f_1},e_{f_2})\neq 0$. This proves the other direction, and finishes the proof of the Proposition.
\end{proof}

Now let $\pi$ be a tempered representation of $G(F)$, then we can find a good parabolic subgroup $L_0U_0=\bar{Q}\subset \bar{P}(F)$ and a discrete series $\tau$ of $L_0$, such that $\pi=I_{\bar{Q}}^{G}(\tau)$ (note that we are in archimedean case, only $GL_1(F),GL_2(F)$ and $GL_1(D)$ have discrete series). Combining Proposition \ref{parabolic induction 2} and Proposition \ref{parabolic induction 3}, we have the following Proposition.

\begin{prop}\label{parabolic induction 4}
With the notation above, we have
$$\CL_{\pi}\neq 0 \iff \CL_{\tau}\neq 0.$$
\end{prop}

\subsection{Proof of Theorem \ref{thm 1}}
Let $\pi$ be a tempered representation of $G(F)$ with central character $\eta$, we already know $\CL_{\pi}\neq 0\Rightarrow m(\pi)\neq 0$. Let's prove the other direction. If $m(\pi)\neq 0$, let $0\neq l\in Hom_H(\pi^{\infty},\xi)$, we first prove
\begin{description}
\item[(1)] For all $e\in \pi^{\infty}$ and $f\in \CC(\zg,\eta^{-1})$, the integral
\begin{equation}\label{thm 1.3}
\int_{\zg}l(\pi(g)e)f(g)dg
\end{equation}
is absolutely convergent.
\end{description}

In fact, this is equivalent to the convergence of
$$\int_{H(F)\back G(F)}\mid l(\pi(x) e)\mid \int_{\zh} \mid f(hx)\mid dhdx.$$
By Proposition \ref{major 4}, for all $d>0$ and $x\in H(F)\back G(F)$, we have
\begin{equation}\label{thm 1.1}
\int_{\zh}\mid f(hx)\mid dh\ll \hc(x)\nor(x)^{-d}.
\end{equation}
On the other hand, by Lemma \ref{lemma 2}, there exist $d'>0$ such that for all $x\in H(F)\back G(F)$, we have
\begin{equation}\label{thm 1.2}
\mid l(\pi(x)e)\mid \ll \hc(x)\nor(x)^{d'}.
\end{equation}
Then (1) follows from \eqref{thm 1.1} and \eqref{thm 1.2}, together with Proposition \ref{major 4}.

Now we can compute \eqref{thm 1.3} in two different ways. First, since $\CC(\zg,\eta^{-1})=C_{c}^{\infty}(\zg,\eta^{-1})\ast \CC(\zg,\eta^{-1})$, we can write $f=\varphi\ast f'$ for $\varphi\in C_{c}^{\infty}(\zg,\eta^{-1})$ and $f'\in \CC(\zg,\eta^{-1})$. Then
\begin{eqnarray*}
&&\int_{\zg}l(\pi(g)e)f(g)dg \\
&=&\int_{\zg}\int_{\zg} l(\pi(g)e)\varphi(g') f'(g'^{-1} g) dg' dg\\
&=&\int_{\zg}\int_{\zg} l(\pi(g'g)e)\varphi(g') dg' f'(g)dg \\
&=&\int_{\zg} l(\pi(\varphi)\pi(g) e) f'(g) dg.
\end{eqnarray*}
Since the vector $l\circ \pi(\varphi)\in \pi^{-\infty}$ belongs to $\bar{\pi}^{\infty}$, by the definition of the action of $\CC(\zg,\eta^{-1})$ on $\pi^{\infty}$, we have
\begin{eqnarray*}
&&\int_{\zg} l(\pi(\varphi)\pi(g) e) f'(g) dg \\
&=&\int_{\zg} f'(g)(\pi(g)e,l\cdot \pi(\varphi)) dg\\
&=&(\pi(f')e,l\cdot \pi(\varphi))=l(\pi(\varphi)\pi(f')e)=l(\pi(f)e).
\end{eqnarray*}
This tells us
\begin{equation}\label{thm 1.4}
\int_{\zg}l(\pi(g)e)f(g)dg=l(\pi(f)e).
\end{equation}

On the other hand,
$$\int_{\zg}l(\pi(g)e)f(g)dg=\int_{H(F)\back G(F)}l(\pi(x)e) \int_{\zh} f(hx)\xi(h) \omega(h) dhdx.$$
By Lemma \ref{lemma 1}(4), if the map $\pi\in \CX_{temp}(G)\rightarrow \pi(f)$ is compactly supported, we have
\begin{eqnarray}\label{thm 1.5}
&&\int_{\zg}l(\pi(g)e)f(g)dg \\
&=&\int_{H(F)\back G(F)} l(\pi(x)e) \int_{Temp(G,\eta)} \CL_{\pi}(\pi(f)\pi(x^{-1})) \mu(\pi) d\pi dx. \nonumber
\end{eqnarray}
For $T\in C_{c}^{\infty}(Temp(G,\eta))$, by applying \eqref{thm 1.4} and \eqref{thm 1.5} to the function $f=f_T$, we have
\begin{equation}\label{thm 1.6}
l(T_{\pi}e)=\int_{H(F)\back G(F)} l(\pi(x)e)\int_{Temp(G,\eta)} \CL_{\pi}(T_{\pi} \pi(x^{-1})) \mu(\pi) d\pi dx
\end{equation}
for all $e\in \pi^{\infty}$. Now assume that $\pi=I_{Q}^{G}(\sigma)$ for some good parabolic subgroup $Q=LU_Q$ of $G$, and $\sigma\in \Pi_2(L)$. Let
$$\CO=\{ Ind_{Q}^{G}(\sigma_{\lambda})\mid \lambda\in i\Fa_{L,0}^{\ast}\}\subset Temp(G,\eta)$$
be the connected component containing $\pi$. Choose $e_0\in \pi^{\infty}$ such that $l(e_0)\neq 0$, let $T_0\in End(\pi)^{\infty}$ with $T_0(e_0)=e_0$. We can easily find an element $T^0\in C_{c}^{\infty}(Temp(G,\eta))$ such that
$$T^{0}_{\pi}=T_0,\; Supp(T^0)\subset \CO.$$
By applying \eqref{thm 1.6} to $e=e_0,T=T^0$, we know there exists $\lambda \in i\Fa_{L,0}^{\ast}$ such that $\CL_{\pi_{\lambda}}\neq 0$ where $\pi_{\lambda}=Ind_{Q}^{G}(\sigma_{\lambda})$. By Proposition \ref{parabolic induction} and Proposition \ref{parabolic induction 4}, this implies $\CL_{\sigma_{\lambda}} \neq 0$. We need a Lemma:

\begin{lem}\label{reduce model}
For all $\lambda\in i\Fa_{L,0}^{\ast}$, we have
$$\CL_{\sigma}\neq 0\iff\CL_{\sigma_{\lambda}} \neq 0.$$
\end{lem}

\begin{proof}
In Appendix B of \cite{Wan15}, we divide the reduce models $(L,H_{\bar{Q}})$ into two types. Type I contain all reduced models in the $\GL_3(D)$ case together with its analogy in the $\GL_6(F)$ case. To be specific, in the $\GL_6(F)$ case, it contain parabolic subgroups of type $(6),\;(4,2)$ and $(2,2,2)$. In the $\GL_3(D)$ case, it contain all parabolic subgroups (i.e. type $(3),\; (2,1)$ and $(1,1,1)$). Type II contain all the rest models in the $\GL_6(F)$ case.

If we are in type I case, there are three models: the Ginzburg-Rallis model which correspondent to the parabolic subgroups of type $(6)$ (resp. type $(3)$) in the $\GL_6(F)$ case (resp. $\GL_3(D)$ case); the "middle" model which correspondent to the parabolic subgroups of type $(4,2)$ (resp. type $(2,1)$) in the $\GL_6(F)$ case (resp. $\GL_3(D)$ case) and the trilinear $\GL_2$ model which correspondent to the parabolic subgroups of type $(2,2,2)$ (resp. type $(1,1,1)$) in the $\GL_6(F)$ case (resp. $\GL_3(D)$ case). For those models, it is easy to show (just by the definition) that the nonvanishing property of $\CL_{\sigma}$ is invariant under the unramified twist.

For type II models, it is not clear from the definition that the unramified twist will preserve the nonvanishing property. We will need a much stronger argument, we claim that for all type II models, $\CL_{\sigma}$ is always nonzero. This will definitely implies our Lemma. Hence in order to prove the Lemma, we only need the following Theorem whose proof will be given in Section \ref{reduce model section}.
\begin{thm}\label{thm 2}
If $(L,H_{\bar{Q}})$ is of type II, then
$\CL_{\sigma}\neq 0$
for all $\sigma \in \Pi^2(L)$.
\end{thm}
\end{proof}

\begin{rmk}\label{thm 2.1}
By the same argument as in this section, we can have a similar formula as \eqref{thm 1.6} for $\CL_{\sigma}$. Then since $\sigma$ is a discrete series, the connected component containing it does not contains other element (i.e. $\CO=\{\sigma\}$). Then by applying the same argument as in the Ginzburg-Rallis model case, we know that $m(\sigma)\neq 0\Rightarrow \CL_{\sigma}\neq 0$ (\textbf{The upshot is that since $\sigma$ is a discrete series, we don't need to worry about the unramified twist issue}). Here $m(\sigma)$ is the dimension of the Hom space $Hom_{H_{\bar{Q}}}(\sigma,\omega\otimes \xi|_{H_{\bar{Q}}})$. Therefore in oder to prove Theorem \ref{thm 2}, we only need to show that for all type II models, the multiplicity $m(\sigma)$ is always nonzero.
\end{rmk}

Now by applying Lemma \ref{reduce model}, we know $\CL_{\sigma}\neq 0$. Applying Proposition \ref{parabolic induction} and Proposition \ref{parabolic induction 4} again, we have $\CL_{\pi}\neq 0$. This proves the other direction, and finishes the proof of Theorem \ref{thm 1}.

\subsection{Some consequences of Theorem \ref{thm 1}}
\textbf{If $F=\BR$}, let $\pi$ be a tempered representation of $G(F)$ with central character $\eta$, since we are in archimedean case, there exists a tempered representation $\pi_0$ of $G_0$ such that $\pi=I_{\bar{P}}^{G}(\pi_0)$. We have the following result.
\begin{cor}\label{parabolic induction 6}
$m(\pi)=m(\pi_0).$
\end{cor}

\begin{proof}
Similar to Theorem \ref{thm 1}, we have
$m(\pi_0)\neq 0\iff \CL_{\pi_0}\neq 0.$
Then applying Proposition \ref{parabolic induction 2}, we have
$$m(\pi)\neq 0\iff \CL_{\pi}\neq 0\iff \CL_{\pi_0}\neq 0\iff m(\pi_0)\neq 0.$$
By strongly multiplicity one Theorem, $m(\pi)$ and $m(\pi_0)$ are either 1 or 0. Then the above equivalence just tells us
$m(\pi)=m(\pi_0).$
\end{proof}

\textbf{If $F$ is p-adic,} let $\pi$ be a tempered representation of $\GL_6(F)$ with central character $\eta$. We can find a good parabolic subgroup $\bar{Q}=LU_Q$ and a discrete series $\sigma$ of $L(F)$ such that $\pi=I_{\bar{Q}}^{G}(\sigma)$. By the construction of the local Jacquet-Langlands correspondence, we know that $\pi^D\neq 0$ iff $\bar{Q}$ is of type I. In fact, the local Jacquet-Langlands correspondence established in \cite{DKV84} gives a bijection between the discrete series series. Then the map can be extend naively to all the tempere representations via the parabolic induction since all tempered representations are full induced from some discrete series (note that we are in $\GL$ case). Therefore, in order to make $\pi^D\neq 0$, the Levi subgroup $L$ should have an analogy in $\GL_3(D)$, this is equivalent to say that $\bar{Q}$ is of type I.
\begin{cor}\label{main case 1}
If $\bar{Q}$ is of type II, Theorem \ref{main} holds.
\end{cor}

\begin{proof}
By the discussion above, we know $\pi^D=0$, so we only need to show that $m(\pi)=1$. By strong multiplicity one theorem, we only need to show that $m(\pi)\neq 0$. By Theorem \ref{thm 2}, we know $\CL_{\sigma}\neq 0$. Together with Proposition \ref{parabolic induction}, we have $\CL_{\pi}\neq 0$. By Theorem \ref{thm 1}, this implies $m(\pi)\neq 0$ and this proves the Corollary.
\end{proof}

Now let $\pi$ be a tempered representation of $G(F)$ with central character $\eta$ (note that $G(F)$ can be both $\GL_6(F)$ and $\GL_3(D)$), we can find a good parabolic subgroup $\bar{Q}=LU_Q$ and a discrete series $\sigma$ of $L(F)$ such that $\pi=I_{\bar{Q}}^{G}(\sigma)$. \textbf{We assume that $\bar{Q}$ is of type I.}

\begin{cor}\label{parabolic induction 5}
\begin{enumerate}
\item $m(\pi)=m(\sigma)$.
\item Let $\CK\subset Temp(G,\eta)$ be a compact subset, then there exists an element $T\in \CC(Temp(G,\eta))$ such that
$\CL_{\pi}(T_{\pi})=m(\pi)$
for all $\pi\in \CK$.
\end{enumerate}
\end{cor}

\begin{proof}
(1) follows from the same proof as Corollary \ref{parabolic induction 6}. For (2), it is enough to show that for all $\pi' \in Temp(G,\eta)$, there exist $T\in \CC(Temp(G,\eta))$ such that
$\CL_{\pi}(T_{\pi})=m(\pi)$
for all $\pi$ in some neighborhood of $\pi'$ in $Temp(G,\eta)$. Since $m(\sigma)$ is invariant under the unramified twist for type I models, combining with part (1) and Corollary \ref{main case 1}, we know that the map $\pi\rightarrow m(\pi)$ is locally constant (In fact, we even know that the map is constant on each connected components of $Temp(G,\eta)$). If $m(\pi')=0$, just take $T=0$, there is nothing to prove.

If $m(\pi')\neq 0$, then we know $m(\pi)=1$ for all $\pi$ in the connected component containing $\pi'$. By Theorem \ref{thm 1}, we can find $T'\in End(\pi')^{\infty}$ such that $\CL_{\pi'}(T')\neq 0$. Then let $T^0\in \CC(Temp(G,\eta))$ be an element with $T_{\pi'}^{0}=T'$. By Lemma \ref{lemma 1}(1), the function $\pi\rightarrow \CL_{\pi}(T^{0}_{\pi})$ is a smooth function. The value at $\pi'$ is just $\CL_{\pi'}(T')\neq 0$. Then we can definitely find a smooth and compactly supported function $\varphi$ on $Temp(G,\eta)$ such that
$\varphi(\pi)\CL_{\pi}(T^{0}_{\pi})=1$
for all $\pi$ lies inside a small neighborhood of $\pi'$. Then we just need to take $T=\varphi T^0$ and this proves the Corollary.
\end{proof}

\section{The archimedean case}\label{the archimedean case}
In this section, we prove our main theorems (Theorem \ref{main} and Theorem \ref{main 1}) for the case when $F=\BR$. The main ingredient of the proof is Corollary \ref{parabolic induction 6} which allows us to reduce the problem to the trilinear $GL_2$ model case. Then by applying the result of Prasad (\cite{P90}), we can prove the two main theorems.

\subsection{The trilinear $GL_2$ model}
In this subsection, we will recall Prasad's result on the trilinear $GL_2$ model. Let $G_0=GL_2(F)\times GL_2(F)\times GL_2(F)$, $H_0=GL_2(F)$ diagonally embed into $G_0$. For a given irreducible representation $\pi_0=\pi_1\otimes \pi_2\otimes \pi_3$ of $G_0$, assume that the central character of $\pi_0$ equals to $\chi^2$ on $Z_{H_0}(F)$ for some unitary character $\chi$ of $F^{\times}$. $\chi$ will induce an one-dimensional representation $\omega_0$ of $H_0$. Let
\begin{equation}
m(\pi_0)=dim(Hom_{H_0(F)} (\pi_0,\omega_0)).
\end{equation}
Similarly, we have the quaternion algebra version: let $G_{0,D}=GL_1(D)\times GL_1(D)\times GL_1(D)$ and $H_{0,D}=GL_1(D)$, we can still define the multiplicity $m(\pi_{0,D})$. The following theorem has been proved by Prasad in his thesis \cite{P90} under the assumption that at least one $\pi_i$ is discrete series (i=1,2,3), and by Loke in \cite{L01} for the case when $\pi_0$ is a principal series.

\begin{thm}\label{trilinear GL2 arch}
With the notation above, if $\pi_0$ is an irreducible generic representation of $G_0$, let $\pi_{0,D}$ be the Jacquet-Langlands correspondence of $\pi_0$ to $G_{0,D}$ if it exists; otherwise let $\pi_{0,D}=0$. Then we have
\begin{enumerate}
\item $m(\pi_0)+m(\pi_{0,D})=1$.
\item If the central character of $\pi$ is trivial on $Z_{H_0}(F)$, then
$$m(\pi_0)=1 \iff \epsilon(1/2,\pi_0)=1$$
and
$$m(\pi_{0,D})=1 \iff \epsilon(1/2,\pi_0)=-1.$$
\end{enumerate}
\end{thm}

\begin{rmk}
Both Prasad's result and Loke's result are based on the assumption that the product of the central characters of $\pi_i$ ($i=1,2,3$) is trivial. In our case, we assume that the product of the central characters is $\chi^2$. But we can always reduce ourself to their cases by replacing $\pi_1$ with $\pi_1 \otimes (\chi^{-1}\circ \det)$. Note that twist by characters will not affect the multiplicity. On the other hand, for the epsilon factor part, we do need the assumption that the product of the central character is trivial. Otherwise the Langlands parameter of $\pi_0$ will no longer be selfdual, hence the value of the epsilon factor at $1/2$ may not be $\pm 1$.
\end{rmk}

\subsection{Proof of Theorem \ref{main} and Theorem \ref{main 1}}
Let $\pi$ be an irreducible tempered representation of $GL_6(F)$, with $F=\BR$. There exists a tempered representation $\pi_0$ of $G_0$ such that $\pi=Ind_{\bar{P}}^{G}(\pi_0)$. Let $\pi_D$ be the Jacquet-Langlands correspondence of $\pi$ to $GL_3(D)$. Similarly we can find a tempered representation $\pi_{0,D}$ of $G_{0,D}$ such that $\pi=Ind_{\bar{P}_D}^{G_D}(\pi_{0,D})$. It is easy to see that $\pi_{0,D}$ is the Jacquet-Langlands correspondence of $\pi_0$ to $G_{0,D}$. Note that $\pi_D$ and $\pi_{0,D}$ may be zero. In fact, they are nonzero iff $\pi_{0,D}$ is a discrete series. By Corollary \ref{parabolic induction 6}, $m(\pi)=m(\pi_0)$ and $m(\pi_D)=m(\pi_{0,D})$. Then by applying Theorem \ref{trilinear GL2 arch}, we have
$m(\pi)+m(\pi_D)=m(\pi_0)+m(\pi_{0,D})=1.$
This proves Theorem \ref{main}.

For Theorem \ref{main 1}, by Theorem \ref{trilinear GL2 arch}, it is enough to show that
$$\epsilon(1/2, \pi,\wedge)=\epsilon(1/2,\pi_0).$$
For $i=1,2,3$, let $\phi_i$ be the Langlands parameter of $\pi_i$. Then the Langlands parameter of $\pi$ is $\phi_{\pi_0}=\phi_1\oplus \phi_2\oplus \phi_3$, hence
\begin{eqnarray*}
&&\wedge^3(\phi_{\pi_0})=\wedge^3(\phi_1\oplus \phi_2\oplus \phi_3)\\
&=&(\phi_1\otimes \phi_2\otimes \phi_3)\oplus (\det(\phi_2) \otimes \phi_1) \oplus (\det(\phi_3) \otimes \phi_1)\\
&&\oplus (\det(\phi_1) \otimes \phi_2)\oplus (\det(\phi_3) \otimes \phi_2)\oplus (\det(\phi_1) \otimes \phi_3)\oplus (\det(\phi_2) \otimes \phi_3).
\end{eqnarray*}
By our assumption on the central character, we have $\det(\phi_{\pi_0})=\det(\phi_1)\otimes \det(\phi_2)\otimes \det(\phi_3)=1$. Therefore $(\det(\phi_2)\otimes \phi_1)^{\vee}=\det(\phi_1)^{-1} \otimes \det(\phi_2)^{-1} \otimes \phi_1=\det(\phi_3)\otimes \phi_1$, this implies
$$\epsilon(1/2,\det(\phi_2)\otimes \phi_1)\epsilon(1/2,\det(\phi_3) \otimes \phi_1)=\det(\phi_1)\otimes \det(\phi_2)^2(-1)=\det(\phi_1)(-1).$$
Similarly, we have
\begin{eqnarray*}
\epsilon(1/2,\det(\phi_1)\otimes \phi_2)\epsilon(1/2,\det(\phi_3) \otimes \phi_2)&=&\det(\phi_1)^2\otimes \det(\phi_2)(-1)=\det(\phi_2)(-1),\\
\epsilon(1/2,\det(\phi_1)\otimes \phi_3)\epsilon(1/2,\det(\phi_2) \otimes \phi_3)&=&\det(\phi_1)^2\otimes \det(\phi_3)(-1)=\det(\phi_3)(-1).
\end{eqnarray*}
Combining the three equations above, we have
\begin{eqnarray*}
\epsilon(1/2, \pi,\wedge^3)&=&\det(\phi_1)\otimes \det(\phi_2)\otimes \det(\phi_3)(-1) \epsilon(1/2,\phi_1\otimes \phi_2\otimes \phi_3)\\
&=&\epsilon(1/2,\phi_1\otimes \phi_2\otimes \phi_3)=\epsilon(1/2,\pi_0).
\end{eqnarray*}
This proves Theorem \ref{main 1}.

Now the only thing left is to prove Theorem \ref{thm 2} for the archimedean case. By Remark \ref{thm 2.1}, we only need to prove the multiplicity is nonzero. Since we are in the archimedean case, there are only three type II models: Type $(2,2,1,1)$, $(2,1,1,1,1)$ and $(1,1,1,1,1,1)$. Type $(1,1,1,1,1,1)$ case is trivial since both $L$ and $H_{\bar{Q}}$ is abelian in this case. For Type $(2,1,1,1,1)$, by canceling the $\GL_1$ part (which is abelian), we are considering the model $(\GL_2(F),T)$ where $T=\{\begin{pmatrix} a&0\\0&b \end{pmatrix}| a,b\in F^{\times} \}$ is the maximal torus. This is the Bessel model for $(\GL_2,\GL_1)$ and we know the multiplicity is always nonzero by the archimedean Rankin-Selberg theory of Jacquet and Shalika (\cite{JS90}).

For Type $(2,2,1,1)$, by canceling the $\GL_1$ part, we are considering the following model:
for $M=\GL_2(F)\times \GL_2(F)$, and
$$M_0=\{m(a,b)=\begin{pmatrix} a&0\\c&b \end{pmatrix}\times \begin{pmatrix} a&0\\c&b \end{pmatrix}|a,b\in F^{\times},\;c\in F\}.$$
The character on $M_0$ is given by $\omega(m(a,b))=\chi(ab)$. Let $B$ be the lower Borel subgroup of $\GL_2(F)$, it is isomorphic to $M_0$, hence we can also view $\omega$ as a character on $B$. Let $\pi_3=I_{B}^{G}(\omega)$, it is a principal series of $\GL_2(F)$. For any irreducible tempered representation $\pi_1\otimes \pi_2$ of $M$, by Frobenius reciprocity, we have
$$Hom_{M_0}(\pi_1\otimes \pi_2,\omega)=Hom_{\GL_2(F)}(\pi_1\otimes \pi_2,\pi_3).$$
Here $\GL_2(F)$ maps diagonally into $M$. Therefore the Hom space is isomorphic to the Hom space of the trilinear $\GL_2$ model for the representation $\pi_0=\pi_1\otimes \pi_2\otimes \pi_3$. Since $\pi_3$ is a principal series, $\pi_{0,D}=0$. By Theorem \ref{trilinear GL2 arch}, $m(\pi_0)=1\neq 0$, hence the Hom space is nonzero and this proves Theorem \ref{thm 2}. \textbf{Now the proof of our main theorems (Theorem \ref{main} and Theorem \ref{main 1}) is complete for the archimedean case.}

\section{The trace formula}\label{Spectral and Geometric Expansion}
\textbf{For the rest of this paper, we assume that $F$ is p-adic}. In this section, we will prove our trace formula. We will define the distribution on Section 7.1. In Section 7.2, we write down the geometric expansion which has already been proved in \cite{Wan15}. In Section 7.3, we prove the spectral expansion.

\subsection{The distribution}
Let $\eta=\chi^2$ be two unitary characters of $F^{\times}$, for $f\in \CC(\zg,\eta^{-1})$, define the function $I(f,\cdot)$ on $H(F)\back G(F)$ to be
$$I(f,x)=\int_{\zh} f(x^{-1}hx)\xi(h)\omega(h) dh.$$
By Lemma \ref{major 2}(2), the above integral is absolutely convergent. The following Proposition together with Proposition \ref{major 4}(3) tell us that the integral
$$I(f):=\int_{H(F)\back G(F)} I(f,x)dx$$
is also absolutely convergent for all $f\in \CC_{scusp}(\zg,\eta^{-1})$, and this defines a continuous linear form
$$\CC_{scusp}(\zg,\eta^{-1})\rightarrow \BC:\; f\rightarrow I(f).$$

\begin{prop}\label{major 6}
\begin{enumerate}
\item There exist $d>0$ and a continuous semi-norm $\nu$ on $\CC(\zg,\eta^{-1})$ such that
$$ |I(f,x)| \leq \nu(f)\hc(x)^2 \nor(x)^{d}$$
for all $f\in \CC(\zg,\eta^{-1})$ and $x\in H(F)\back G(F)$.
\item For all $d>0$, there exists a continuous semi-norm $\nu_d$ on $\CC(\zg,\eta^{-1})$ such that
$$ |I(f,x)| \leq \nu_d(f) \hc(x)^2 \nor(x)^{-d}$$
for all $f\in \CC_{scusp}(\zg,\eta^{-1})$ and $x\in H(F)\back G(F)$.
\end{enumerate}
\end{prop}

\begin{proof}
The proof goes exactly the same as Proposition 7.1.1 of \cite{B15}. In the loc. cit., the author is dealing with the Gan-Gross-Prosad period, but the proof of that Proposition worked for general case except the following five results which are specified to the GGP model: Lemma 6.5.1, Lemma 6.6.1, Proposition 6.4.1, Proposition 6.7.1 and Proposition 6.8.1 in the loc. cit. But we already proved the above five results for Ginzburg-Rallis model in Section 4 for the Ginzburg-Rallis model, see Lemma \ref{norm}, Proposition \ref{spherical}, Proposition \ref{cartan 1}, Lemma \ref{major 2}, Lemma \ref{major 3}, Proposition \ref{major 4} and Lemma \ref{major 5}. Therefore the argument in the loc. cit. can be applied to our model smoothly. We refer the readers to my thesis \cite{Wan17} for details of the proof.
\end{proof}

\subsection{The geometric expansion}
Let's recall the results in the previous paper \cite{Wan15} for the geometric expansion of $I(f)$. In Section 3.4, for any strongly cuspidal $f$ belonging to the space $C_{c}^{\infty}(Z_G(F)\backslash G(F),\eta^{-1})$, we have defined a quasi-character $\theta_f$ on $G(F)$. As in Section 7.1 of loc. cit., we can define the geometric side of the trace formula via $\theta_f$.

Let $\CT$ be a subset of subtorus of $H_0$ defined as follows:
\begin{itemize}
\item If $H_0=\GL_2(F)$, then $\CT$ contain the trivial torus $\{1\}$ and the non-split torus $T_v$ for $v\in F^{\times}/(F^{\times})^2, v\neq 1$ where $T_v=\{\begin{pmatrix} a & bv \\ b & a \end{pmatrix} \in H_0(F) \mid a,b\in F,\;(a,b)\neq (0,0)\}$.
\item If $H_0=D$, then $\CT$ contain the subtorus $T_v$ for $v\in F^{\times}/(F^{\times})^2$ with $v\neq 1$, where $T_v\subset D$ is isomorphic to the quadratic extension $F(\sqrt{v})$ of $F$.
\end{itemize}
Let $\theta$ be a quasi-character on $G(F)$ with central character $\eta^{-1}$, and $T\in \CT$. If $T=\{1\}$, we are in the split case. In this case, we have a unique regular nilpotent orbit $\CO_{reg}$ in $\Fg(F)$ and take $c_\theta(1)=c_{\theta,\CO_{reg}}(1)$. If $T=T_v$ for some $v\in F^{\times}/(F^{\times})^2$
with $v\neq 1$, we take $t\in T_v$ to be a regular element. It is easy to see that in both cases $G_t(F)$ is $F$-isomorphic to $\GL_3(F_v)$ where $F_v=F(\sqrt{v})$ is the quadratic extension of $F$. Let $\CO_v$ be the unique regular nilpotent orbit in $\Fg \Fl_3(F_v)$, and take $c_\theta(t)=c_{\theta,\CO_v}(t)$. For the function $f$ above, let $c_f(t)=c_{\theta_f}(t)$ for all $t\in T_{reg}(F),\; T\in \CT$. Also we define a function $\Delta$ on $H_{ss}(F)$ to be
$$\Delta(x)=\mid \det((1-ad(x)^{-1})_{\mid U(F)/U_x(F)}) \mid_F.$$

\begin{defn}
The geometric side of the trace formula is defined by
\begin{equation}\label{geometric}
I_{geom}(f)=\sum_{T\in \CT}|W(H_0,T)|^{-1} \nu(T) \int_{Z_G(F)\backslash T(F)} c_f(t) D^H(t) \Delta(t) \chi(\det(t)) dt.
\end{equation}
The integral above is absolutely convergent by Proposition 5.2 of \cite{Wan15}.
\end{defn}

The following theorem is proved in \cite{Wan15}, it gives the geometric expansion of our trace formula.
\begin{thm}\label{geometric expansion}
For every function $f\in C_{c}^{\infty} (Z_G(F)\backslash G(F),\eta^{-1})$ that is strongly cuspidal, we have
$I(f)=I_{geom}(f).$
\end{thm}

\begin{proof}
This is just Theorem 5.4 of \cite{Wan15}.
\end{proof}

\subsection{The spectral expansion}
\begin{thm}\label{spectral expansion}
For all $f\in \CC_{scusp}(\zg,\eta^{-1})$, define
$$I_{spec}(f)=\int_{Temp(G,\eta)} \theta_f(\pi)m(\bar{\pi})d\pi,$$
with $\theta_f(\pi)$ defined in \eqref{theta_f}. Then we have $I(f)=I_{spec}(f)$.
\end{thm}

The purpose of the rest of this section is to prove this Theorem. We follow the method developed by Beuzart-Plessis in \cite{B15} for the GGP case. We fix $f\in \CC_{scusp}(\zg,\eta^{-1})$, for all $f'\in \CC(\zg,\eta)$, define
\begin{eqnarray*}
K_{f,f'}^{A}(g_1,g_2)&=& \int_{\zg} f(g_{1}^{-1} gg_2)f'(g)dg, \;g_1,g_2\in G(F),\\
K_{f,f'}^{1}(g,x)&=&\int_{\zh}K_{f,f'}^{A}(g,hx) \xi(h) \omega(\det(h)) dh, \; g,x\in G(F),\\
K_{f,f'}^{2}(x,y)&=&\int_{\zh}K_{f,f'}^{1}(h^{-1}x,y) \xi(h) \omega(\det(h)) dh, \; x,y\in G(F),\\
J_{aux}(f,f')&=&\int_{H(F)\back G(F)} K_{f,f'}^{2}(x,x)dx.
\end{eqnarray*}

\begin{prop}\label{spectral 1}
\begin{enumerate}
\item The integral defining $K_{f,f'}^{A}(g_1,g_2)$ is absolutely convergent. For all $g_1\in G(F)$, the map
$$g_2\in G(F)\rightarrow K_{f,f'}^{A}(g_1,g_2)$$
belongs to $\CC(G(F),\eta^{-1})$. For all $d>0$, there exists $d'>0$ such that for all continuous semi-norm $\nu$ on $\CC_{d'}^{w}(G(F),\eta^{-1})$, there exists a continuous semi-norm $\mu$ on $\CC(G(F),\eta)$ such that
$$\nu(K_{f,f'}^{A}(g,\cdot)) \leq \mu(f') \Xi^G(g)\sigma_0(g)^{-d}$$
for all $f'\in \CC(G(F),\eta)$ and $g\in G(F)$.
\item The integral defining $K_{f,f'}^{1}(g,x)$ is absolutely convergent. For all $d>0$, there exists $d'>0$ and a continuous semi-norm $\nu_{d,d'}$ on
$\CC(G(F),\eta)$ such that
$$|K_{f,f'}^{1}(g,x)|\leq \nu_{d,d'}(f') \Xi^G(g)\sigma_0(g)^{-d}\Xi^{H\back G}(x)\sigma_{H\back G}(x)^{d'}$$
for all $f'\in \CC(G(F),\eta)$ and $g,x\in G(F)$.
\item The integral defining $K_{f,f'}^{2}(x,y)$ is absolutely convergent. We have
\begin{equation}\label{spectral 1.1}
K_{f,f'}^{2}(x,y)=\int_{Temp(G,\eta)} \CL_{\pi}(\pi(x)\pi(f)\pi(y^{-1}))\overline{\CL_{\pi}(\pi(\overline{f'}))} \mu(\pi)d\pi
\end{equation}
for all $f'\in \CC(G(F),\eta)$ and $x,y\in G(F)$. Since we are in p-adic case, the integrand on right hand side is compactly supported, therefore the integral is always absolutely convergent.
\item The integral defining $J_{aux}(f,f')$ is absolutely convergent. And for all $d>0$, there exists a continuous semi-norm $\nu_d$ on $\CC(G(F),\eta)$ such that
$|K_{f,f'}^{2}(x,x)|\leq \nu_d(f')\Xi^{H\back G}(x)^2 \sigma_{H\back G}(x)^{-d}$
for all $f'\in \CC(G(F),\eta)$ and $x\in H(F)\back G(F)$. Moreover, the linear map
\begin{equation}\label{spectral 1.2}
f'\in \CC(G(F),\eta)\rightarrow J_{aux}(f,f')
\end{equation}
is continuous.
\end{enumerate}
\end{prop}

\begin{proof}
(1) follows from Theorem \ref{local trace formula}(1). (2) follows from part (1) together with Lemma \ref{major 2}(2) and Lemma \ref{lemma 2}(2). For (3), the absolutely convergence follows from part (2) and Lemma \ref{major 2}(2). The equation \eqref{spectral 1.1} follows from Lemma \ref{lemma 1}(5).

For (4), by Lemma \ref{lemma 1}(1), the section
$$T(f'):\pi\in Temp(G,\eta)\mapsto \overline{\CL_{\pi}(\pi(\overline{f'}))} \pi(f)\in End(\pi)^{\infty}$$
is smooth. It is also compactly supported since we are in the p-adic case. Then by the matrical Paley-Wiener Theorem, there exists a unique element $\varphi_{f'}\in \CC(\zg,\eta^{-1})$ such that $\pi(\varphi_{f'})=\overline{\CL_{\pi}(\pi(\overline{f'}))} \pi(f)$ for all $\pi\in Temp(G,\eta)$. Since $f$ is strongly cuspidal, by Lemma 5.3.1(1) of \cite{B15}, $\varphi_{f'}$ is also strongly cuspidal. Then by \eqref{spectral 1.1}, we have
\begin{eqnarray*}
K_{f,f'}^{2}(x,x)&=&\int_{Temp(G,\eta)} \CL_{\pi}(\pi(x)\pi(f)\pi(x^{-1}))\overline{\CL_{\pi}(\pi(\overline{f'}))} \mu(\pi)d\pi \\
&=& \int_{Temp(G,\eta)} \CL_{\pi}(\pi(x)\pi(\varphi_{f'}) \pi(x^{-1})) \mu(\pi) d\pi \\
&=& \int_{\zh} \varphi_{f'}(x^{-1}hx) \xi(h) \omega(h) dh=I(\varphi_{f'},x).
\end{eqnarray*}
Here the third equation follows from Lemma \ref{lemma 1}(4). Then by Proposition \ref{major 6}, for all $d>0$, there exists a continuous semi-norm $\nu_d$ on $\CC(G(F),\eta)$ such that
$|K_{f,f'}^{2}(x,x)|\leq \nu_d(\varphi_{f'})\Xi^{H\back G}(x)^2 \sigma_{H\back G}(x)^{-d}$
for all $f'\in \CC(G(F),\eta)$ and $x\in H(F)\back G(F)$. Combining with Proposition \ref{major 4}(4), we know the integral defining $J_{aux}(f,f')$ is absolutely convergent. Finally, in order to prove the rest part of (4), it is enough to show the map
$\CC(G(F),\eta)\rightarrow \CC(G(F),\eta^{-1}):f'\mapsto \varphi_{f'}$
is continuous. By the matrical Paley-Wiener Theorem, it is enough to show the map
$$f'\in \CC(G(F),\eta) \mapsto (\pi\in Temp(G,\eta)\rightarrow \pi(\varphi_{f'})=\overline{\CL_{\pi}(\pi(\overline{f'}))} \pi(f))\in \CC(Temp(G,\eta))$$
is continuous. This just follows from Lemma \ref{lemma 1}(1). This finishes the proof of the proposition.
\end{proof}

\begin{prop}
For all $f'\in \CC(G(F),\eta)$, we have
$$J_{aux}(f,f')=\int_{Temp(G,\chi)} \theta_f(\pi) \overline{\CL_{\pi}(\pi(\overline{f'}))} d\pi.$$
\end{prop}

\begin{proof}
The idea of proof comes from \cite{B15}. Let $a:\BG_m(F)\rightarrow Z_{G_0}(F)$ be a homomorphism defined by $a(t)=diag(t,t,1,1,t^{-1},t^{-1})$ in the split case, and $a(t)=diag(t,1,t^{-1})$ in the non-split case, then we know $\lambda(a(t)ha(t)^{-1})=t\lambda(h)$ for all $h\in H(F),t\in \BG_m(F)$. Fix $f'\in \CC(G(F),\eta)$, since we are in p-adic case, we can find an open compact neighborhood $K_a$ of $1$ in $F^{\times}$ such that $Ad_a(t)f'=f'$ for all $t\in K_a$. Let $\varphi\in C_{c}^{\infty}(F^{\times})$ be the characteristic function on $K_a$ divide by the measure of $K_a$, then we have $f'=Ad_a(\varphi)(f')$ and $J_{aux}(f,f')=\int_{F^{\times}} \varphi(t) J_{aux}(f,Ad_a(t)f') dt$. By the definition of $J_{aux}$, we have
$$J_{aux}(f,f')=\int_{F^{\times}}\int_{H(F)\back G(F)} \varphi(t) K_{f,Ad_a(t)f'}^{2}(x,x)dx dt.$$
By part (4) of previous proposition, the double integral is absolutely convergent. Then by changing variable $x\mapsto a(t)^{-1}x$ and switching the two integrals (note that the Jacobian of the map $h\in H(F)\mapsto a(t)ha(t)^{-1}\in H(F)$ is equal to $\delta_P(a(t))$), we have
\begin{equation}\label{8.2}
J_{aux}(f,f')=\int_{H(F)\back G(F)}\int_{F^{\times}} \varphi(t) \delta_P(a(t))^{-1} K_{f,Ad_a(t)f'}^{2}(a(t)x,a(t)x) dt dx.
\end{equation}
By the definition of $K_{f,f'}^{2}$, the inner integral is equal to
$$\int_{F^{\times}} \varphi(t) \delta_P(a(t))^{-1} \int_{\zh} K_{f,Ad_a(t)f'}^{1}(ha(t)x,a(t)x)\xi(h)^{-1} \omega(\det(h))^{-1} dh dt.$$
By part (2) of previous proposition, the double integral is still absolutely convergent. By changing variable $h\rightarrow a(t)^{-1}ha(t)$ and switching the two integrals, we have
\begin{eqnarray}\label{8.3}
&&\;\;\;\;\;\;\;\;\;\;\;\;\;\;\int_{F^{\times}} \varphi(t) \delta_P(a(t))^{-1} K_{f,Ad_a(t)f'}^{2}(a(t)x,a(t)x) dt\\
&=&\int_{\zh}\int_{F^{\times}} \varphi(t) K_{f,Ad_a(t)f'}^{1}(a(t)hx,a(t)x)\psi(-t\lambda(h)) \omega(\det(h))^{-1} dtdh \nonumber\\
&=&\int_{\zh}\int_{F^{\times}} \varphi(t) K_{f,R_a(t)f'}^{1}(hx,a(t)x)\psi(-t\lambda(h)) \omega(\det(h))^{-1} dtdh. \nonumber
\end{eqnarray}
Here $R_a(t)$ stands for the right translation by $a(t)$. By the definition of $K_{f,f'}^{1}$, the inner integral above is equal to
$$\int_{F^{\times}} \varphi(t) \int_{\zh} K_{f,R_a(t)f'}^{A}(hx,h'a(t)x)\xi(h') \omega(\det(h')) dh' \psi(-t\lambda(h)) \omega(\det(h))^{-1}dt.$$
By changing variable $h'\rightarrow a(t)^{-1}h'a(t)h^{-1}$, this equals
\begin{eqnarray*}
&&\int_{F^{\times}} \varphi(t)\int_{\zh} K_{f,R_a(t)f'}^{A}(hx,a(t)h'hx) \delta_P(a(t)) \psi(t\lambda(h')) \omega(\det(h')) dh' dt\\
&=&\int_{F^{\times}} \varphi(t)\int_{\zh} K_{f,f'}^{A}(hx,h'hx) \delta_P(a(t)) \psi(t\lambda(h')) \omega(\det(h')) dh' dt.
\end{eqnarray*}
By part (1) of previous proposition, the integral above is absolutely convergent. By switching two integrals, we have
\begin{eqnarray}\label{8.4}
&&\;\;\;\;\;\;\;\;\;\;\;\int_{F^{\times}} \varphi(t) K_{f,R_a(t)f'}^{1}(hx,a(t)x)\psi(-t\lambda(h)) \omega(\det(h))^{-1} dt\\
&=&\int_{\zh} \int_{F^{\times}} K_{f,f'}^{A}(hx,h'hx) \varphi(t)  \delta_P(a(t)) \psi(t\lambda(h')) \omega(\det(h')) dh' dt. \nonumber
\end{eqnarray}
We know $dt=|t|^{-1}d_a t$ where $d_at$ is an additive Haar measure on $F$. Let $\varphi'(t)=\varphi(t)\delta_P(a(t))|t|^{-1}$ and $\hat{\varphi}'(x)=\int_F \varphi'(t)\psi(tx)dt$ for $x\in F$. Combining \eqref{8.2}, \eqref{8.3} and \eqref{8.4}, we have
\begin{equation}\label{8.5}
J_{aux}(f,f')=\int_{H(F)\back G(F)} \int_{\zh} \int_{\zh} K_{f,f'}^{A}(hx,h'hx) \hat{\varphi}'(\lambda(h')) \omega(\det(h')) dh' dh dx.
\end{equation}

For $N,M>0$, let $\alpha_N:H(F)\back G(F)\rightarrow \{0,1\}$ (resp. $\beta_M:Z_G(F)\back G(F)\rightarrow \{0,1\}$) be the characteristic function of the set $\{x\in H(F)\back G(F)| \sigma_{H\back G}(x)\leq N \}$ (resp. $\{g\in Z_G(F)\back G(F)| \sigma_0(g)\leq M\}$). For $N\geq 1$ and $C>0$, define
\begin{eqnarray*}
J_{aux,N}(f,f')&=&\int_{H(F)\back G(F)} \alpha_N(x) \int_{\zh} \int_{\zh}\\
&& K_{f,f'}^{A}(hx,h'hx) \hat{\varphi}'(\lambda(h')) \omega(\det(h')) dh' dh dx, \\
J_{aux,N,C}(f,f')&=&\int_{H(F)\back G(F)} \alpha_N(x) \int_{\zh} \int_{\zh}\\
&& \beta_{C\log(N)}(h') K_{f,f'}^{A}(hx,h'hx) \hat{\varphi}'(\lambda(h')) \omega(\det(h')) dh' dh dx.
\end{eqnarray*}
By equation \eqref{8.5}, we have
\begin{equation}\label{8.6}
J_{aux}(f,f')=\lim_{N\rightarrow \infty} J_{aux,N}(f').
\end{equation}
We need to prove
\begin{itemize}
\item[(1)] The triple integrals defining $J_{aux,N}(f,f')$ and $J_{aux,N,C}(f,f')$ are absolutely convergent. Moreover, there exists $C>0$ such that
$$|J_{aux,N}(f,f')-J_{aux,N,C}(f,f')|\ll N^{-1}$$
for all $N\geq 1$.
\end{itemize}

In fact, since $\hat{\varphi}'$ is compactly supported on $F$, we have $|\hat{\varphi}'(\lambda)|\ll (1+|\lambda|)^{-1}$ for all $\lambda\in F$. Combining with Theorem \ref{local trace formula}, we know there exists $d>0$ such that
\begin{eqnarray*}
|J_{aux,N}(f,f')|&\ll&\int_{H(F)\back G(F)} \alpha_N(x) \int_{\zh}  \int_{\zh}\\
&& \Xi^G(hx)\Xi^G(h'hx)\sigma_0(hx)^d \sigma_0(h'hx)^d (1+|\lambda(h')|)^{-1} dh' dh dx, \\
|J_{aux,N,C}(f,f')|&\ll&\int_{H(F)\back G(F)} \alpha_N(x) \int_{\zh} \int_{\zh}\\
&& \beta_{C\log(N)}(h') \Xi^G(hx)\Xi^G(h'hx)\sigma_0(hx)^d \sigma_0(h'hx)^d (1+|\lambda(h')|)^{-1} dh' dh dx,\\
\end{eqnarray*}
and
$$|J_{aux,N}(f,f')-J_{aux,N,C}(f,f')|\ll \int_{H(F)\back G(F)} \alpha_N(x) \int_{\zh} $$
$$\times \int_{\zh} 1_{\sigma_0\geq C\log(N)}(h') \Xi^G(hx)\Xi^G(h'hx)\sigma_0(hx)^d \sigma_0(h'hx)^d (1+|\lambda(h')|)^{-1} dh' dh dx.$$
for all $N\geq 1$ and $C\geq 1$. Applying $(7)$ of Proposition \ref{major 4} to the case $c=1$, we know there exists $d'>0$ such that the first two integrals above are essentially bounded by
$$\int_{H(F)\back G(F)} \alpha_N(x)\Xi^{H\back G}(x)^2 \sigma_{H\back G}(x)^{d'} dx.$$
This is absolutely convergent since the integrand is compactly supported. Then applying $(7)$ of Proposition \ref{major 4} again, we know the third integral is essentially bounded by
$$e^{-\epsilon C\log(N)}\int_{H(F)\back G(F)} \alpha_N(x)\Xi^{H\back G}(x)^2 \sigma_{H\back G}(x)^{d'} dx, \; N\geq 1,C>0.$$
for some $\epsilon,d'>0$. By (4) of Proposition \ref{major 4}, there exists $d''>0$ such that the last integral is essentially bounded by $N^{d''}$ for all $N\geq 1$. Then once we choose $C$ larger than $(d''+1)/\epsilon$, we have the estimation in (1). This proves (1).

Now we fix some $C>0$ satisfies (1), then we have
\begin{equation}\label{8.7}
J_{aux}(f,f')=\lim_{N\rightarrow \infty} J_{aux,N,C}(f,f').
\end{equation}
Since the integral defining $J_{aux,N,C}$ is absolutely convergent, we can combine the first two parts and then switch two integrals, we have
\begin{eqnarray}\label{8.12}
J_{aux,N,C}(f,f')&=&\int_{\zg} \alpha_N(g) \int_{\zh} \nonumber \\
&&K_{f,f'}^{A}(g,h'g)\beta_{C\log(N)}(h')\hat{\varphi}'(\lambda(h'))\omega(\det(h')) dh' dg \\
&=&\int_{\zh} \beta_{C\log(N)}(h)\hat{\varphi}'(\lambda(h))\omega(\det(h))\nonumber \\
&&\times \int_{\zg} \alpha_N(g) K_{f,f'}^{A}(g,hg) dg dh. \nonumber
\end{eqnarray}
We are going to prove that for all $N\geq 1$, we have
\begin{equation}\label{8.8}
|J_{aux,N,C}(f,f')-J_{aux,C}(f,f')|\ll N^{-1}
\end{equation}
where
$$J_{aux,C}(f,f')=\int_{\zh}\beta_{C\log(N)}(h)\hat{\varphi}'(\lambda(h))\omega(\det(h)) \int_{\zg}  K_{f,f'}^{A}(g,hg) dg dh.$$
In fact, since $f$ is strongly cuspidal, by Theorem \ref{local trace formula}(3), there exists $c_1>0$ such that for all $d>0$, there exists $d'>0$ such that
$$|K_{f,f'}^{A}(g,hg)|\ll \Xi^G(g)^2 \sigma_0(g)^{-d} e^{c_1\sigma_0(h)}\sigma_0(h)^{d'}$$
for all $g\in G(F),h\in H(F)$. Fix such $c_1>0$, and choose $d_0>0$ so that the function $g\rightarrow \Xi^G(g)^2\sigma_0(g)^{-d_0}$ is integrable on $G(F)/Z_G(F)$. Then for all $d>d_0$, there exists $d'>0$ such that the left hand side of \eqref{8.8} is essentially bounded by
$$N^{c_1C-d+d_0}\log(N)^{d'}\int_{\zh} \beta_{C\log(N)}(h) dh$$
for all $N\geq 1$. It is easy to see that the integral above is essentially bounded by $N^{c_2}$ for some $c_2>0$. Therefore once we choose $d>c_1C+d_0+c_2+1$, we have the estimation in \eqref{8.8}. This proves \eqref{8.8}. Therefore we have
\begin{eqnarray}\label{8.9}
J_{aux}(f,f')&=&\lim_{N\rightarrow \infty} \int_{\zh}\beta_{C\log(N)}(h)\hat{\varphi}'(\lambda(h))\omega(\det(h)) \nonumber \\
&&\times \int_{\zg}  K_{f,f'}^{A}(g,hg) dg dh.
\end{eqnarray}
Since $f$ is strongly cuspidal, by Theorem \ref{local trace formula}(4), we have
\begin{equation}\label{8.9}
\int_{\zg} K_{f,f'}^{A}(g,hg)dg=\int_{Temp(G,\eta)} \theta_f(\pi) \theta_{\bar{\pi}}(R(h^{-1})f') d\pi.
\end{equation}
Since $\pi$ is tempered, $|\theta_{\bar{\pi}}(R(h^{-1})f')|\ll \Xi^G(h)$ for all $h\in H(F)$. Combining with the fact that $\theta_f(\pi)$ is smooth and compactly supported on $Temp(G,\eta)$, we have
$$\int_{Temp(G,\eta)} |\theta_f(\pi) \theta_{\bar{\pi}}(R(h^{-1})f')| d\pi\ll \Xi^G(h).$$
Combining with Lemma \ref{major 2}, we know the integral
$$\int_{\zh} \hat{\varphi}'(\lambda(h))\omega(\det(h)) \int_{Temp(G,\eta)} \theta_f(\pi) \theta_{\bar{\pi}}(R(h^{-1})f') d\pi dh$$
is absolutely convergent. Combining with \eqref{8.8} and \eqref{8.9}, the integral above is equal to $J_{aux}(f,f')$. Switching the two integrals and applying Lemma \ref{linear form}, we have
\begin{eqnarray*}
J_{aux}(f,f')&=& \int_{Temp(G,\eta)} \theta_f(\pi)\overline{\CL_{\pi}(\pi(\overline{Ad_a(\varphi)f'}))} d\pi\\
&=&\int_{Temp(G,\eta)} \theta_f(\pi)\overline{\CL_{\pi}(\pi(\overline{f'}))} d\pi.
\end{eqnarray*}
This finishes the proof of the Proposition.
\end{proof}

Now we are ready to prove Theorem \ref{spectral expansion}. Recall that $I(f)=\int_{H(F)\back G(F)} I(f,x)dx$ where $I(f,x)=\int_{\zh} f(x^{-1}hx)\xi(h) \omega(\det(h))dh$. Applying Lemma \ref{lemma 1}, we have
\begin{equation}\label{8.10}
I(f,x)=\int_{Temp(G,\eta)} \CL_{\pi}(\pi(x)\pi(f)\pi(x)^{-1})\mu(\pi) d\pi.
\end{equation}
By Corollary \ref{parabolic induction 5}, there exists a function $f'\in \CC(G(F),\eta)$ such that
$$\CL_{\pi}(\pi(\overline{f'}))=m(\pi)$$
for all $\pi\in Temp(G,\eta)$ with $\pi(f)\neq 0$. Applying Theorem \ref{thm 1} and Corollary \ref{parabolic induction 5}, for all $\pi\in Temp(G,\eta)$, $\CL_{\pi}\neq 0$ iff $m(\pi)=1$. Then \eqref{8.10} becomes
$$I(f,x)=\int_{Temp(G,\eta)} \CL_{\pi}(\pi(x)\pi(f)\pi(x)^{-1})\overline{\CL_{\pi}(\pi(\overline{f'}))} \mu(\pi) d\pi.$$
Combining with Proposition \ref{spectral 1}(3), we have $I(f,x)=K_{f,f'}^{2}(x,x)$. Therefore $I(f)=J_{aux}(f,f')$. By the previous Proposition, together with the fact that $\overline{\CL_{\pi}(\pi(\overline{f'}))}=m(\pi)=m(\bar{\pi})$, we have
$$I(f)=J_{aux}(f,f')=\int_{Temp(G,\eta)} \theta_f(\pi)m(\bar{\pi}) d\pi=I_{spec}(f).$$
This finishes the proof of Theorem \ref{spectral expansion}.

\section{Proof of Theorem \ref{main} and Theorem \ref{main 1}}\label{Proof of Theorem}
In this section, we prove our main theorems (Theorem \ref{main} and Theorem \ref{main 1}) for the p-adic case. In Section 8.1, we introduce the multiplicity formulas for both the Ginzburg-Rallis model and all other reduced models. We show that Theorem \ref{main} follows from the multiplicity formula. Also we prove that the multiplicity formula is compatible with the parabolic induction. In Section 8.2, by applying the trace formula in Section 7, we prove the multiplicity formula. This proves Theorem \ref{main}. In Section 8.3, we prove Theorem \ref{main 1}.

\subsection{The multiplicity formula}
Let $\pi$ be an irreducible tempered representation of $G(F)$ with central character $\eta=\chi^2$. Similar to Section 7.2, we define the geometric multiplicity to be
$$m_{geom}(\pi)=\sum_{T\in \CT}|W(H_0,T)|^{-1} \nu(T) \int_{Z_G(F)\backslash T(F)} c_{\pi}(t) D^H(t) \Delta(t) \chi(\det(t))^{-1} dt.$$
Here $c_{\pi}(t)=c_{\theta_{\pi}}(t)$ is the germ associated to the distribution character $\theta_{\pi}$. The multiplicity formula is just
\begin{equation}\label{9.1}
m(\pi)=m_{geom}(\pi).
\end{equation}

Let $\pi=I_{\bar{Q}}^{G}(\tau)$ for some good parabolic subgroup $\bar{Q}=LU_Q$ and some discrete series $\tau$ of $L(F)$. We also need the geometric multiplicity for the reduced model $(L,H_{\bar{Q}})$ where $H_{\bar{Q}}=H\cap \bar{Q}$. If $\bar{Q}$ is of Type II, we are in the $\GL_6(F)$ case, define
$$m_{geom}(\tau)=c_{\theta_{tau},\CO_{reg}}(1).$$
By the work of Rodier \cite{Rod81}, we know $m_{geom}(\tau)$ is always equal to 1 in this case. If $\bar{Q}$ is of Type I, the definition of $m_{geom}(\tau)$ is very similar as $m_{geom}(\pi)$: it is still the summation of the integrals of the germs on $Z_H(F)\back T(F)$ for $T\in \CT$. For details, see Appendix B. The following lemma tells us the relation between $m_{geom}(\pi)$ and $m_{geom}(\tau)$.

\begin{lem}\label{multiplicity formula}
With the notation above, we have
$$m_{geom}(\pi)=m_{geom}(\tau).$$
\end{lem}

\begin{proof}
This is a direct consequence of Lemma \ref{quasi character parabolic}(2). In fact, if $\bar{Q}$ is of Type I, by applying the lemma, we know the germs associated to $\pi$ and $\tau$ are the same:
$c_{\pi}(t)=c_{\tau}(t)$, $\forall t\in T_{reg}(F),T\in \CT.$
Hence $m_{geom}(\pi)=m_{geom}(\tau)$. If $\bar{Q}$ is of type II, by applying the lemma, we know the germ $c_{\pi}(t)$ is zero for all $t\in T_{reg}(F),T\in \CT$ with $t\neq 1$. Therefore we have
$m_{geom}(\pi)=c_{\pi}(1)=1=c_{\tau}(1)=m_{geom}(\tau).$
This proves the lemma.
\end{proof}

\begin{prop}
Theorem \ref{main} follows from the multiplicity formula \eqref{9.1}.
\end{prop}

\begin{proof}
The proof already appears on the previous paper \cite{Wan15}, we include it here for completion. Let $G=GL_6(F)$ and $G_D=GL_3(D)$. Similarly we have $H_0, H_{0,D}, U$ and $U_D$. Let $\pi,\pi_D,\chi,\omega,\omega_D, \xi$ and $\xi_D$ be the same as Conjecture \ref{jiang}. We assume that $\pi$ is tempered. By \eqref{9.1}, we know
\begin{eqnarray*}
m(\pi)&=&c_{\theta_{\pi},\CO_{reg}}(1)+\Sigma_{v\in F^{\times}/(F^{\times})^2, v\neq 1} \mid W(H,T_v)\mid^{-1} \nu(T_v)\\
&&\times \int_{Z_H\backslash T_v(F)} \omega^{-1}(t) c_{\pi}(t) D^H(t) \Delta(t)dt
\end{eqnarray*}
and
\begin{eqnarray*}
m(\pi_D)&=&\Sigma_{v\in F^{\times}/(F^{\times})^2, v\neq 1} \mid W(H_D,T_v)\mid^{-1} \nu(T_v)\\
&&\times \int_{Z_{H_D}\backslash T_v(F)} \omega_{D}^{-1}(t') c_{\pi_D}(t') D^{H_D}(t') \Delta_D(t')dt.
\end{eqnarray*}
Here we use $t$ to denote elements in $\GL_6(F)$ and $t'$ to denote elements in $\GL_3(D)$. We can match $t$ and $t'$ via the characteristic polynomial: we write $t \leftrightarrow t'$ if they have the same characteristic polynomial. Since $\pi$ is tempered, it is generic. So by \cite{Rod81}, we know $c_{\theta_{\pi},\CO_{reg}}(1)=1$. Also for $v\in F^{\times}/(F^{\times})^2, v\neq 1$, we have
$$\mid W(H_{0,D},T_v)\mid=\mid W(H_0,T_v)\mid, Z_{H_0}=Z_{H_{0,D}}.$$
So in order to prove Theorem \ref{main}, we only need to show for any $v\in F^{\times}/(F^{\times})^2, v\neq 1$, the sum of
$$\int_{Z_H(F)\backslash T_v(F)} \omega^{-1}(t) c_{\pi}(t) D^H(t) \Delta(t) d_c t$$
and
$$\int_{Z_H(F)\backslash T_v(F)} \omega_{D}^{-1}(t')c_{\pi_D}(t') D^{H_D}(t') \Delta_D(t') d_c t'$$
Because for $t,t'\in T_v(F)$ regular with $t\leftrightarrow t'$, we have
$$D^H(t)=D^{H_D}(t), \Delta(t)=\Delta_D(t'), \omega(t)=\omega_D(t').$$
Therefore it is enough to show that for any $v\in F^{\times}/(F^{\times})^2, v\neq 1$, and for any $t,t'\in T_v(F)$ regular with $t\leftrightarrow t'$, we have
\begin{equation}\label{9.5}
c_{\pi}(t)+c_{\pi_D}(t')=0.
\end{equation}

By Section 13.6 of \cite{W10} or Proposition 4.5.1 of \cite{B15}, we know
$$c_{\pi}(t)=D^G(t)^{-1/2}|W(G_t, T_{qs,t}|^{-1} \lim_{x\in T_{qs,t}(F)\rightarrow t} D^G(x)^{1/2} \theta_{\pi}(x)$$
and
$$c_{\pi_D}(t')=D^{G_D}(t)^{-1/2}|W((G_D)_{t'}, T_{qs,t'}|^{-1} \lim_{x'\in T_{qs,t'}(F)\rightarrow t} D^{G_D}(x')^{1/2} \theta_{\pi_D}(x')$$
where $T_{qs,t}$ (resp. $T_{qs,t'}$) is a maximal torus contained in the Borel subgroup $B_t$ (resp. $B_{t'}$) of $G_t$ (resp. $(G_D)_{t'}$). Note that if $t,t'\in T_v$ is regular, both $G_t$ and $(G_D)_{t'}$ are isomorphic to $GL_3(F_v)$ which is quasi-split over $F$. We are able to choose the Borel subgroup $B_t$ (resp. $B_{t'}$). In particular,
$|W(G_t, T_{qs,t})|^{-1}=|W((G_D)_t, T_{qs,t})|^{-1}.$
Also for those matched $t\leftrightarrow t'$, we have
$D^G(t)=D^{G_D}(t).$
And for $x\in T_{qs,t}(F)$ (resp. $x'\in T_{qs,t'}(F)$) sufficiently close to $t$ (resp. $t'$) with $x\leftrightarrow x'$, they are also regular and we have
$D^G(x)=D^{G_D}(x').$
Therefore in order to prove \eqref{9.5}, it is enough to show that for any regular $x\in G(F)$ and $x'\in G_D(F)$ with $x\leftrightarrow x'$, we have
\begin{equation}
\theta_{\pi}(x)+\theta_{\pi_D}(x')=0.
\end{equation}
This just follows from the relations of the distribution characters under the Jacquet-Langlands correspondence, see \cite{DKV84}. This proves Theorem \ref{main}
\end{proof}

By the proposition above, we only need to prove \eqref{9.1}. We will only prove it for the $\GL_6$ case since the quaternion case follows from the same argument.

\subsection{The proof of \eqref{9.1}}
In this section, the group $G$ will be $\GL_6(F)$. Recall that $\pi$ is of the form $\pi=I_{\bar{Q}}^{G}(\tau)$ for some good parabolic subgroup $\bar{Q}=LU_Q$ and some discrete series $\tau$ of $L(F)$.

\begin{lem}\label{9.7}
If $\bar{Q}$ is of type II, the multiplicity formula \eqref{9.1} holds.
\end{lem}

\begin{proof}
In Lemma \ref{multiplicity formula}, we prove that $m_{geom}(\pi)=1$. By Corollary \ref{main case 1}, we also know $m(\pi)=1$. Hence the multiplicity formula \eqref{9.1} holds.
\end{proof}

If $\bar{Q}$ is of type I, there are only three possibilities: type $(6)$, type $(4,2)$ and type $(2,2,2)$. Since both $m(\pi)$ and $m_{geom}(\pi)$ are compatible with the parabolic induction, we only need to prove $m(\tau)=m_{geom}(\tau)$. \textbf{We first assume that \eqref{9.1} holds for the $(2,2,2)$ and $(4,2)$ case, we are going to prove \eqref{9.1} for the $(6)$ case, i.e. the case when $\pi$ is a discrete series.}

\begin{prop}\label{(6)}
Under the hypothesis above, if $\pi$ is a discrete series, the multiplicity formula \eqref{9.1} holds.
\end{prop}

\begin{proof}
By Theorem \ref{geometric expansion} and Theorem \ref{spectral expansion}, for every strongly cuspidal function $f\in C_{c}^{\infty} (Z_G(F)\backslash G(F),\eta)$, we have
\begin{equation}\label{9.2}
I_{geom}(f)=I(f)=\int_{Temp(G,\eta^{-1})} \theta_f(\Pi)m(\bar{\Pi})d\Pi.
\end{equation}
By Proposition \ref{expansion of character}, we have
\begin{equation}\label{9.3}
I_{geom}(f)=\int_{Temp(G,\eta^{-1})} \theta_f(\Pi)m_{geom}(\bar{\Pi}) d\Pi.
\end{equation}

Let $\Pi^2(G,\eta^{-1})\subset Temp(G,\eta^{-1})$ be the subset consisting of discrete series, and let $Temp'(G,\eta^{-1})=Temp(G,\eta^{-1})-\Pi^2(G, \eta^{-1})$. Then by \eqref{9.2} and \eqref{9.3}, we have
\begin{eqnarray}\label{9.6}
&&\int_{Temp'(G,\eta^{-1})} \theta_f(\Pi)m(\bar{\Pi})d\Pi+\int_{\Pi^2(G,\eta^{-1})} \theta_f(\Pi)m(\bar{\Pi})d\Pi \\
&=&\int_{Temp'(G,\eta^{-1})} \theta_f(\Pi)m_{geom}(\bar{\Pi}) d\Pi+\int_{\Pi^2(G,\eta^{-1})} \theta_f(\Pi)m_{geom}(\bar{\Pi}) d\Pi. \nonumber
\end{eqnarray}
For $\pi'\in Temp'(G,\eta^{-1})$, by our assumption and Lemma \ref{9.7}, we have $m(\pi')=m_{geom}(\pi')$. Therefore \eqref{9.6} becomes
\begin{equation}\label{9.8}
\int_{\Pi^2(G,\eta^{-1})} \theta_f(\Pi)m(\bar{\Pi})d\Pi=\int_{\Pi^2(G,\eta^{-1})} \theta_f(\Pi)m_{geom}(\bar{\Pi}) d\Pi.
\end{equation}

Now take $f\in C_{c}^{\infty}(\zg,\eta)$ be the pseudo coefficient of $\bar{\pi}$, this means that $trace(\bar{\pi}(f))=1$ and $trace(\sigma(f)) = 0$ for all $\sigma \in Temp(G, \eta^{-1})$ with $\sigma\neq \bar{\pi}$. The existence of such f was proved in Lemma \ref{pseudo coefficient}, the lemma also shows that $f$ is strongly cuspidal. For such $f$ and for any $\Pi\in \Pi^2(G,\eta)$, we have $\theta_f(\Pi)=trace(\Pi(f))$. Hence it is nonzero if and only if $\Pi=\bar{\pi}$. Therefore \eqref{9.8} becomes
$$\theta_f(\bar{\pi})m(\pi)=\theta_f(\bar{\pi})m_{geom}(\pi).$$
Hence $m_{geom}(\pi)=m(\pi)$, and this proves the Proposition.
\end{proof}

By the proposition above, we only need to prove \eqref{9.1} for the $(2,2,2)$ and $(4,2)$ cases. For the $(4,2)$ case, we only need to prove \begin{equation}\label{9.9}
m(\tau)=m_{geom}(\tau)
\end{equation}
where $\tau$ is a discrete series of $\GL_4(F)\times \GL_2(F)$. By the same argument as in the previous proposition, we are reducing to prove \eqref{9.9} for all tempered representations $\tau$ which are not discrete series. Then by applying Lemma \ref{9.7}, we are reducing to the $(2,2,2)$ case.

For the $(2,2,2)$ case, we need to prove
\begin{equation}\label{9.10}
m(\tau)=m_{geom}(\tau)
\end{equation}
where $\tau$ is a discrete series of $\GL_2(F)\times \GL_2(F)\times \GL_2(F)$. By Lemma \ref{9.7}, we know \eqref{9.10} holds for all tempered representations $\tau$ which are not discrete series (note that there is no Type I reduced model for the trilinear $\GL_2$ model expect itself). Then \eqref{9.10} follows from the same argument as in the previous proposition. This finishes the proof of the multiplicity formula \eqref{9.1}, and hence the proof of Theorem \ref{main}.

\subsection{The proof of Theorem \ref{main 1}}
Let $\pi$ be an irreducible tempered representation of $\GL_6(F)$ with trivial central character. Let $\pi=I_{\bar{Q}}^{G}(\tau)$ for some good parabolic subgroup $\bar{Q}=LU_Q$ and some discrete series $\tau$ of $L(F)$. By our assumptions in Theorem \ref{main 1}, $\bar{Q}$ cannot be of type $(6)$ or type $(4,2)$. Then there are two possibilities: $\bar{Q}$ is of type $(2,2,2)$ or $\bar{Q}$ is of Type II.

If $\bar{Q}$ is of type $(2,2,2)$. By a similar argument as in Section 6.2, we know
$\epsilon(1/2,\pi,\wedge^3)=\epsilon(1/2,\tau).$
Combining with Prasad's results for the trilinear $\GL_2$ model (\cite{P90}) and the fact that $m(\pi)=m(\tau)$, we prove Theorem \ref{main 1}.

If $\bar{Q}$ is of type II, by Corollary \ref{main case 1}, $m(\pi)=1$. Hence it is enough to prove the following proposition.

\begin{prop}
If $\bar{Q}$ is of type II, we have
$\epsilon(1/2,\pi,\wedge^3)=1.$
\end{prop}

\begin{proof}
Since $\bar{Q}$ is of type II, it is contained in some type II maximal parabolic subgroups. There are only two type II maximal parabolic subgroups: type $(5,1)$ and type $(3,3)$.

If $\bar{Q}$ is contained in the parabolic subgroup $Q_{5,1}$ of type $(5,1)$, then there exists a tempered representation $\sigma=\sigma_1\times \sigma_2$ of $\GL_5(F)\times \GL_1(F)$ such that $\pi=I_{Q_{5,1}}^{G}(\sigma)$. Let $\phi_i$ be the Langlands parameter of $\sigma_i$ for $i=1,2$. Then $\phi=\phi_1\oplus \phi_2$ is the Langlands parameter for $\pi$. Then we have
\begin{eqnarray*}
\wedge^3(\phi)&=&\wedge^3(\phi_1\oplus \phi_2)=\wedge^3(\phi_1)\oplus (\wedge^2(\phi_1)\otimes \phi_2).
\end{eqnarray*}
Since the central character of $\pi$ is trivial, $\det(\phi)=\det(\phi_1)\otimes \det(\phi_2)=1$. Therefore $(\wedge^3(\phi_1))^{\vee}=\wedge^2(\phi_1)\otimes \det(\phi_1)^{-1}=\wedge^2(\phi_1)\otimes \det(\phi_2)=\wedge^2(\phi_1)\otimes \phi_2$, this implies
$$\epsilon(1/2,\pi,\wedge^3)=\det(\wedge^3(\phi_1))(-1)=(\det(\phi_1))^6(-1)=1.$$

If $\bar{Q}$ is contained in the parabolic subgroup $Q_{3,3}$ of type $(3,3)$, then there exists a tempered representation $\sigma=\sigma_1\times \sigma_2$ of $\GL_3(F)\times \GL_3(F)$ such that $\pi=I_{Q_{3,3}}^{G}(\sigma)$. Let $\phi_i$ be the Langlands parameter of $\sigma_i$ for $i=1,2$. Then $\phi=\phi_1\oplus \phi_2$ is the Langlands parameter for $\pi$. Then we have
\begin{eqnarray*}
\wedge^3(\phi)&=&\wedge^3(\phi_1\oplus \phi_2)\\
&=& (\wedge^2(\phi_1)\otimes \phi_2)\oplus(\phi_1\otimes \wedge^2(\phi_2)) \oplus \det(\phi_1)\oplus \det(\phi_2).
\end{eqnarray*}
Since the central character of $\pi$ is trivial, $\det(\phi)=\det(\phi_1)\otimes \det(\phi_2)=1$. Therefore $(\wedge^2(\phi_1)\otimes \phi_2)^{\vee}=(\phi_1\otimes \det(\phi_1)^{-1})\otimes (\wedge^2(\phi_2)\otimes \det(\phi_2)^{-1})=\phi_1\otimes \wedge^2(\phi_2)$ and $(\det(\phi_1))^{\vee}=\det(\phi_2)$, this implies
\begin{eqnarray*}
\epsilon(1/2,\pi,\wedge^3)&=&\det(\wedge^2(\phi_1)\otimes \phi_2)(-1)\times \det(\phi_1)(-1) \\
&=& \det(\wedge^2(\phi_1))^3(-1)\times \det(\phi_2)^3(-1)\times \det(\phi_1)(-1)\\
&=&(\det(\phi_1)^2(-1))^3\times (\det(\phi_2)(-1))^3\times \det(\phi_1)(-1)=1.
\end{eqnarray*}
This finishes the proof of the proposition and hence the proof of Theorem \ref{main 1}.
\end{proof}

Now the proof of Theorem \ref{main} and Theorem \ref{main 1} is complete. The only thing left is to prove Theorem \ref{thm 2} for the p-adic case. This will be done in the next section.

\section{The proof of Theorem \ref{thm 2}}\label{reduce model section}
In this section, we will prove Theorem \ref{thm 2} for the p-adic case, the archimedean case has already been proved in Section 6. By remark \ref{thm 2.1}, we only need to prove the multiplicity is nonzero for all type II reduced models $(L,H_{\bar{Q}})$. Our method is similar to the Ginzburg-Rallis model case we considered in previous sections. In other word, we need a local relative trace formula for such models, this gives us a multiplicity formula for such models in terms of the germs associated to the distribution characters of the representations. \textbf{The key feather for type II models is that all semisimple elements in $H_{\bar{Q}}$ is split. As a result, in the geometric side of our trace formula, we only have the germ at the identity element. Hence the multiplicity formula is always 1 by the work of Rodier \cite{Rod81}.}

Let $(L,H_{\bar{Q}})$ be a type II reduced model. For simplicity, we use $\omega\otimes \xi$ to denote the character $\omega\otimes \xi|_{H_{\bar{Q}}}$ on $H_{\bar{Q}}$. Given an irreducible tempered representation $\tau$ of $L(F)$ whose central character $\alpha$ equals to $\eta$ on $Z_G(F)$, we want to show that
\begin{equation}\label{q.1}
m(\tau):=dim(Hom_{H_{\bar{Q}}(F)}(\tau,\omega\otimes \xi))\neq 0.
\end{equation}

As in the Ginzburg-Rallis model case, for $f\in C_{c}^{\infty}(Z_L(F)\back L(F),\alpha^{-1})$ strongly cuspidal, we define our distribution $I_Q(f)$ to be
$$I_Q(f)=\int_{H_{\bar{Q}}(F)\back L(F)}I_Q(f,x) dx$$
where
$$I_Q(f,x)=\int_{Z_{H_{\bar{Q}}}(F)\back H_{\bar{Q}}(F)} f(x^{-1}hx)\xi(h)\omega(h)dh,\; x\in H_{\bar{Q}}(F)\back L(F).$$
By a similar argument as in Proposition \ref{major 6}, we know the integral defining $I_Q(f)$ is absolutely convergent. The next proposition gives the trace formula for such model.

\begin{prop}\label{q.2}
For all $f\in C_{c}^{\infty}(Z_L(F)\back L(F),\alpha^{-1})$ strongly cuspidal, we have
$$I_{Q,geom}(f)=I_Q(f)=I_{Q,spec}(f)$$
where
$I_{Q,geom}(f)=c_{\theta_f,\CO_{reg}}(1)$, and
$$I_{Q,spec}(f)=\int_{Temp(L,\alpha)} \theta_f(\tau)m(\bar{\tau}) d\tau.$$
\end{prop}

\begin{proof}
The geometric expansion $I_Q(f)=I_{Q,geom}(f)$ is proved in Appendix B of the previous paper \cite{Wan15}. The proof for the spectral expansion is similar to the Ginzburg-Rallis case we proved in Theorem \ref{spectral expansion}, we will skip it here.
\end{proof}

By applying the proposition above, together with the same arguments as in Section 8.2, we have a multiplicity formula for type II models:
$$m(\tau)=m_{geom}(\tau):=c_{\theta_{\tau},\CO_{reg}}(1)$$
for all tempered representations $\tau$ of $L(F)$. Then by the work of Rodier, we have $m(\tau)=1\neq 0$. This finishes the proof of Theorem \ref{thm 2}.

\appendix
\section{The Cartan Decomposition}
\subsection{The problem}
In this Appendix, we are going to prove the Cartan decomposition for the trilinear $\GL_2$ model (as in Proposition \ref{cartan 1}). Let $F$ be a p-adic field, $\CO_F$ be the ring of integers, $\varpi_F$ be the uniformizer, $|\;|=|\;|_F$, and $\BF_q$ be the residue field with $q=p^n$. Let $G=\GL_2(F)\times \GL_2(F)\times \GL_2(F)$, $H=\GL_2(F)$ diagonally embedded into $G$, $K'=\GL_2(\CO_F) \cup \GL_2(\CO_F) \begin{pmatrix} 1&0\\0&\varpi_F\end{pmatrix}$, $K_0=\GL_2(\CO_F)\times \GL_2(\CO_F)\times \GL_2(\CO_F)$ be the maximal compact subgroup of $G$, $K=K_0(K'\times K'\times K')K_0$ be a compact subset of $G$ with $K=K_0KK_0$, and
$$A^{+}=\{(\begin{pmatrix} 1&-1\\ 0&1\end{pmatrix} a_1 \begin{pmatrix} 1&1 \\ 0&1\end{pmatrix},a_2,a_3)|a_1,a_2\in A_{0}^{-},\;a_3\in A_{0}^{+}\}$$
where $A_{0}^{+}=\{\begin{pmatrix} a&0\\0&b\end{pmatrix} |a,b\in F^{\times},\; |a|\geq |b|  \}$ and $A_{0}^{-}=\{\begin{pmatrix} a&0\\0&b\end{pmatrix} |a,b\in F^{\times},\; |a|\leq |b|  \}$. Our goal is to show that
\begin{equation}\label{A.1}
G=HA^{+}K.
\end{equation}

We first do some reductions. For $(g_1,g_2,g_3)\in G$, by timing some elements on $K^{-1}$ on the right and by timing some elements in the center (which is contained in $A^{+}$), may assume that $\det(g_1)=\det(g_2)=\det(g_3)=1$. Then by timing $(g_{1}^{-1},g_{1}^{-1},g_{1}^{-1})\in H$ on the left, we only need to consider elements of the form $(1,g,g')$. Applying the Cartan decomposition $\GL_2(F)=\GL_2(\CO)A_{0}^{+}\GL_2(\CO)$ to $g$ and $g'$, then absorb the right $\GL_2(\CO_F)$ part by elements in $K_0$, we only need to consider elements of the form $(1,ka,k'a')$ with $k,k'\in \GL_2(\CO_F)$ and $a,a'\in A_{0}^{+}$. Then by timing $(a^{-1}k^{-1},a^{-1}k^{-1},a^{-1}k^{-1})\in H$ on the left, and absorbing $k^{-1}$ by elements in $K_0$, we only need to consider elements of the form $(a,1,g)$ with $a\in A_{0}^{-}$ and $g\in \GL_2(F)$. Applying the Iwasawa decomposition to $g$, we may assume that $g$ is upper triangular. Therefore we only need to consider elements of the form
$$(a,a',b)$$
where $a\in A_{0}^{-}$ with $\det(a)=1$, $a'=I_2$, and $b$ is upper triangular with $\det(b)=1$. Then by timing $(u,u,u)\in H$ on the left with $u=\begin{pmatrix} 1&-1\\0&1\end{pmatrix}$, and absorbing the $u$ in the second coordinate by elements in $K_0$, we only need to consider elements of the form
\begin{equation}\label{A.2}
(ua,a',b)
\end{equation}
where $a\in A_{0}^{-}$ with $\det(a)=1$, $a'=I_2$, and $b$ is upper triangular with $\det(b)=1$.
Now \eqref{A.1} has been reduced to the following proposition.

\begin{prop}\label{A.3}
For all elements $g=(ua,a',b)$ of the form \eqref{A.2}, there exists $h\in H(F),\; t\in A^{+}$ and $k\in K_0$ such that
$$g=htk.$$
\end{prop}

\subsection{The case when $b$ is diagonal}
In this section, we prove Proposition \ref{A.3} for the case when $b$ is a diagonal matrix. We let $a=\begin{pmatrix} x^{-1}&0\\0&x\end{pmatrix}$ with $|x|\geq 1$. By our assumption, $b=\begin{pmatrix} y&0\\0&y^{-1}\end{pmatrix}$ or $\begin{pmatrix} y^{-1}&0\\0&y\end{pmatrix}$ with $|y|\geq 1$.

\textbf{Case 1:} If $b=\begin{pmatrix} y&0\\0&y^{-1}\end{pmatrix}$, let
$$h=(I_2,I_2,I_2),\;t=(uau^{-1},I_2,b)\in A^{+},\;k=(u,I_2,I_2). $$
Then we have
$$g=htk.$$

\textbf{Case 2:} If $b=\begin{pmatrix} y^{-1}&0\\0&y\end{pmatrix}$, let
$$h=(\begin{pmatrix} 0&-1\\1&2\end{pmatrix},\begin{pmatrix} 0&-1\\1&2\end{pmatrix},\begin{pmatrix} 0&-1\\1&2\end{pmatrix}),\; t=(uau^{-1},I_2,b^{-1})$$
and
$$k=(u\begin{pmatrix} 1&0\\-\frac{1}{x^2}&1\end{pmatrix},\begin{pmatrix} 2&1\\-1&0\end{pmatrix},\begin{pmatrix} \frac{2}{y^2}&1\\-1&0\end{pmatrix}).$$
Then we have
$$g=htk.$$

\subsection{The general situation}
In this section, we prove Proposition \ref{A.3} for the general case (i.e. $b$ is a upper triangular matrix). We still let $a=\begin{pmatrix} x^{-1}&0\\0&x\end{pmatrix}$ with $|x|\geq 1$. The proof breaks into four cases.

\textbf{Case 1:} If $b=\begin{pmatrix} a&b\\0&c\end{pmatrix}$ with $|a|\geq |b|$, then $b=\begin{pmatrix} a&0\\0&c\end{pmatrix}\begin{pmatrix} 1&\frac{b}{a}\\0&1\end{pmatrix}$ with $\begin{pmatrix} 1&\frac{b}{a}\\0&1\end{pmatrix}\in \GL_2(\CO_F)$. Then by timing some elements in $K_0$, we reduce to the case when $b$ is a diagonal matrix, which has been considered in previous section.

\textbf{Case 2:} If $b=\begin{pmatrix} 1&t\\0&1\end{pmatrix}\begin{pmatrix} y&0\\0&y^{-1}\end{pmatrix}$ with $|y|\geq 1$. If $|t|\leq |y|^2$, we are back to case 1. So assume that $|t|>|y|^2\geq 1$. Let
$$h=(\begin{pmatrix} 1-t^{-1}&0\\t^{-1}&1\end{pmatrix},\begin{pmatrix} 1-t^{-1}&0\\t^{-1}&1\end{pmatrix},\begin{pmatrix} 1-t^{-1}&0\\t^{-1}&1\end{pmatrix}),\; t=(uau^{-1},I_2, \begin{pmatrix} \frac{t}{y}&0\\0&\frac{y}{t}\end{pmatrix})$$
and
$$k=(u\begin{pmatrix} 1&0\\\frac{1}{x^2t}&1-t^{-1}\end{pmatrix}^{-1},\begin{pmatrix} 1-t^{-1}&0\\t^{-1}&1\end{pmatrix}^{-1},\begin{pmatrix} -\frac{1}{y^2}&-1\\1&\frac{y^2}{t}\end{pmatrix}^{-1}).$$
Then we have
$$g=htk.$$

\textbf{Case 3:} If $b=\begin{pmatrix} 1&t\\0&1\end{pmatrix}\begin{pmatrix} y^{-1}&0\\0&y\end{pmatrix}$ with $|y|\geq 1$ and $|t|>1$. Let
$$h=(\begin{pmatrix} \frac{t}{t+1}&0\\ \frac{1}{t+1}&1\end{pmatrix},\begin{pmatrix} \frac{t}{t+1}&0\\ \frac{1}{t+1}&1\end{pmatrix},\begin{pmatrix} \frac{t}{t+1}&0\\ \frac{1}{t+1}&1\end{pmatrix}),\; t=(uau^{-1},I_2, \begin{pmatrix} yt&0\\0&\frac{1}{yt}\end{pmatrix})$$
and
$$k=(u\begin{pmatrix} 1&0\\\frac{1}{x^2(t+1)}&\frac{t}{t+1}\end{pmatrix}^{-1},\begin{pmatrix} \frac{t}{t+1}&0\\\frac{1}{t+1}&1\end{pmatrix}^{-1},\begin{pmatrix} 0&-1\\ \frac{t}{t+1}&\frac{1}{ty^2}\end{pmatrix}^{-1}).$$
Then we have
$$g=htk.$$

\textbf{Case 4:} If $b=\begin{pmatrix} 1&t^{-1}\\0&1\end{pmatrix}\begin{pmatrix} y^{-1}&0\\0&y\end{pmatrix}$ with $|y|,|t|\geq 1$. If $|t|\geq |y|^2$, we are back to case 1. So assume that $1\leq |t|<|y|^2$. There are two subcases.

\textbf{Case 4(a):} If $|t|\geq |x|^2$. Times g by $(\begin{pmatrix} 1&-t^{-1}\\0&1\end{pmatrix},\begin{pmatrix} 1&-t^{-1}\\0&1\end{pmatrix},\begin{pmatrix} 1&-t^{-1}\\0&1\end{pmatrix})$ on the left, note that $a^{-1}u^{-1}\begin{pmatrix} 1&-t^{-1}\\0&1\end{pmatrix}ua=\begin{pmatrix} 1&x^2t^{-1}\\0&1\end{pmatrix} \in \GL_2(\CO_F)$, hence by modulo an element in $K_0$, we may assume that $b=\begin{pmatrix} y^{-1}&0\\0&y\end{pmatrix}$ is a diagonal matrix, which has been considered in previous section.

\textbf{Case 4(b):} If $1\leq |t|< |x|^2$. We have three subcases.

\textbf{Case 4(b)(i):} If $|t+1|\geq 1$. Let
$$h=(\begin{pmatrix} \frac{1}{t}&1+\frac{1}{t(t+1)}\\1& \frac{1}{t+1}\end{pmatrix},\begin{pmatrix}\frac{1}{t}&1+\frac{1}{t(t+1)}\\1& \frac{1}{t+1} \end{pmatrix},\begin{pmatrix} \frac{1}{t}&1+\frac{1}{t(t+1)}\\1& \frac{1}{t+1} \end{pmatrix}),\; t=(uau^{-1},I_2, \begin{pmatrix} y&0\\0&\frac{1}{y}\end{pmatrix})$$
and
$$k=(u\begin{pmatrix} \frac{t+1}{t}&0\\\frac{1}{x^2}&-\frac{t}{t+1}\end{pmatrix}^{-1},\begin{pmatrix} \frac{1}{t}&1+\frac{1}{t(t+1)}\\1& \frac{1}{t+1}\end{pmatrix}^{-1},\begin{pmatrix} 0&1\\1&\frac{1}{(t+1)y^2}\end{pmatrix}^{-1}).$$
Then we have
$$g=htk.$$

\textbf{Case 4(b)(ii):} If $t=-1$. Times $g$ by $(u^{-1},u^{-1},u^{-1})$ on the left and absorb the second $u^{-1}$ by some elements in $K_0$, we may assume that $g=(\begin{pmatrix} x^{-1}&0\\0&x\end{pmatrix},I_2,\begin{pmatrix} y^{-1}&0\\0&y\end{pmatrix})$ with $|x|,|y|\geq 1$. If $|y|\geq |x|$, let
$$h=(\begin{pmatrix} 0&x^{-1}\\x&0\end{pmatrix},\begin{pmatrix} 0&x^{-1}\\x&0\end{pmatrix},\begin{pmatrix} 0&x^{-1}\\x&0\end{pmatrix}),\; t=(I_2,\begin{pmatrix} x^{-1}&0\\0&x\end{pmatrix},\begin{pmatrix} x^{-1}y&0\\0&xy^{-1}\end{pmatrix})$$
and
$$k=(\begin{pmatrix} 0&1\\1&0\end{pmatrix},\begin{pmatrix} 0&1\\1&0\end{pmatrix},\begin{pmatrix} 0&1\\1&0\end{pmatrix}).$$
If $|y|<|x|$, let
$$h=(\begin{pmatrix} y^{-1}&y^{-1}\\-y&0\end{pmatrix},\begin{pmatrix} y^{-1}&y^{-1}\\-y&0\end{pmatrix},\begin{pmatrix} y^{-1}&y^{-1}\\-y&0\end{pmatrix}),\; t=(u\begin{pmatrix} x^{-1}y&0\\0&xy^{-1}\end{pmatrix}u^{-1},\begin{pmatrix} y^{-1}&0\\0&y\end{pmatrix},I_2)$$
and
$$k=(u\begin{pmatrix} 1&0\\x^{-2}y^2&1\end{pmatrix},\begin{pmatrix} y^{-2}&1\\-1&0\end{pmatrix}^{-1},\begin{pmatrix} 1&1\\-1&0\end{pmatrix}^{-1}).$$
Then in both cases, we have
$$g=htk.$$

\textbf{Case 4(b)(iii):} If $|t+1|<1$ with $t\neq -1$, then $|t|=1$. If $|(t+1)y^2|\leq 1$, we have
\begin{eqnarray*}
b&=&\begin{pmatrix} 1&t^{-1}\\0&1\end{pmatrix}\begin{pmatrix} y^{-1}&0\\0&y\end{pmatrix} \\
&=& \begin{pmatrix} 1&-1\\0&1\end{pmatrix} \begin{pmatrix} 1&t^{-1}+1\\0&1\end{pmatrix} \begin{pmatrix} y^{-1}&0\\0&y\end{pmatrix}\\
&=& \begin{pmatrix} 1&-1\\0&1\end{pmatrix}\begin{pmatrix} y^{-1}&0\\0&y \end{pmatrix} \begin{pmatrix} 1&(t^{-1}+1)y^2\\0&1\end{pmatrix}
\end{eqnarray*}
with $\begin{pmatrix} 1&(t^{-1}+1)y^2\\0&1\end{pmatrix}\in \GL_2(\CO_F)$. Then up to modulo an element in $K_0$, we can eliminate $\begin{pmatrix} 1&(t^{-1}+1)y^2\\0&1\end{pmatrix}$, and we have reduced to Case 4(b)(ii).

If $|(t+1)x^2|\leq 1$, times $g$ by $(\begin{pmatrix} 1&-t^{-1}+1\\0&1\end{pmatrix},\begin{pmatrix} 1&-t^{-1}+1\\0&1\end{pmatrix},\begin{pmatrix} 1&-t^{-1}+1\\0&1\end{pmatrix})\in H$ on the left, then modulo an element in $K_0$ to eliminate $a^{-1}u^{-1}\begin{pmatrix} 1&t^{-1}+1\\0&1\end{pmatrix}ua=\begin{pmatrix} 1&(t^{-1}+1)x^2\\0&1\end{pmatrix}\in \GL_2(\CO_F)$ in the first coordinate and $\begin{pmatrix} 1&t^{-1}+1\\0&1\end{pmatrix}\in \GL_2(\CO_F)$ in the second coordinate, we still reduce to Case 4(b)(ii).

Now the only case left is $|(t+1)y^2|,|(t+1)x^2|>1$. Let
$$h=(\begin{pmatrix} \frac{1}{t+1}&0\\ \frac{t}{t+1}&1 \end{pmatrix},\begin{pmatrix}\frac{1}{t+1}&0\\ \frac{t}{t+1}&1 \end{pmatrix},\begin{pmatrix} \frac{1}{t+1}&0\\ \frac{t}{t+1}&1 \end{pmatrix}),$$
$$t=(u \begin{pmatrix} x^{-1}&0\\0&x(t+1) \end{pmatrix}u^{-1},\begin{pmatrix} t+1&0\\0&1 \end{pmatrix}, \begin{pmatrix} \frac{y(t+1)}{t}&0\\0&-\frac{t}{y}\end{pmatrix})$$
and
$$k=(u\begin{pmatrix} 1&0\\-\frac{t}{x^2(t+1)}&1\end{pmatrix},\begin{pmatrix} 1&0\\-t& 1\end{pmatrix},\begin{pmatrix} 0&1\\1&-y^{-2}t \end{pmatrix}^{-1}).$$
Then we have
$$g=htk.$$

The proof of Proposition \ref{A.3} is finally complete.

\section{The Reduced Model}
In this section, we will state some similar results for the reduced models of the Ginzburg-Rallis model, which appear naturally under the parabolic induction. To be specific, we can have analogy results of Theorem \ref{main} and Theorem \ref{main 1} for those models. Since most of proof goes exactly the same as the Ginzburg-Rallis model case which we discussed in previous Sections, we will skip it here. We refer the readers to my thesis \cite{Wan17} for details of the proof. We still use $(L,H_{\bar{Q}})$ to denote the reduced model, and use $m(\tau)$ to denote the multiplicity.

\subsection{Type II models}
If $(L,H_{\bar{Q}})$ is of type II, the following result has already been proved in Section 9.
\begin{thm}\label{type II}
With the assumption above, we have
$$m(\tau)=c_{\theta_{\tau},\CO_{reg}}(1)=1$$
for all tempered representation $\tau$ of $L(F)$.
\end{thm}

\subsection{The trilinear model}
Choose $Q$ be the parabolic subgroup of $\GL_6(F)$ (resp. $\GL_3(D)$) of $(2,2,2)$ type (resp. $(1,1,1)$ type) which contains the lower minimal parabolic subgroup. Then it is easy to see the model $(L,H_Q)$ is just the trilinear $\GL_2$ model. To make our notation simple, in this section we will temporarily let $G=GL_2(F)\times GL_2(F)\times GL_2(F)$ and $H=GL_2(F)$ diagonally embeded into $G$. For a given irreducible representation $\pi$ of $G$, assume the central character $\omega_{\pi}$ equals to $\chi^2$ on $Z_H(F)$ for some character $\chi$ of $F^{\times}$. $\chi$ will induce a one-dimensional representation $\omega$ of $H$. Let
\begin{equation}
m(\pi)=\dim \Hom_{H(F)} (\pi,\omega).
\end{equation}

Similarly, we have the quaternion algebra version with the pair $G_D=GL_1(D)\times GL_1(D)\times GL_1(D)$ and $H_D=GL_1(D)$. We can still define the multiplicity $m(\pi_D)$. The following theorem has been proved by Prasad in his thesis \cite{P90} for general generic representation using different method. This can also be deduced essentially from Waldspurger's result on the model $SO(4)\times SO(3)/SO(3)$. By using our method in this paper, we can prove the tempered case.

\begin{thm}
If $\pi$ is a tempered representation of $G$, let $\pi_D$ be the Jacquet-Langlands correspondence of $\pi$ to $G_D$. Then
$$m(\pi)+m(\pi_D)=1.$$
\end{thm}

We can also prove a multiplicity formula for this model in the p-adic case. Let $(G,H)$ be either $(G,H)$ or $(G_D,H_D)$ defined as above, and $F$ be a p-adic field. Let $\CT$ be the subset of subtorus $T$ in $H_0$ defined in Section 7.2. Let $\theta$ be a quasi-character on $Z_G(F)\backslash G(F)$ with central character $\eta=\chi^2$, and $T\in \CT$. If $T=\{1\}$, then we are in the split case. Since there is a unique regular nilpotent orbit in $\Fg(F)$, we let $c_\theta(t)=c_{\theta,\CO_{reg}}(t)$. If $T=T_v$ for some $v\in F^{\times}/(F^{\times})^2, v\neq 1$, and $t\in T_v$ is a regular element, then $G_t$ is abelian. Since in this case the germ of quasi-character is just itself, we define $c_\theta(t)=\theta(t)$. Then for all irreducible tempered representation $\pi$ of $G(F)$ with central character $\eta$, we have
$$m(\pi)=m_{geom}(\pi):=\sum_{T\in \CT} \mid W(H,T)\mid^{-1} \nu(T) \int_{Z_G(F)\backslash T(F)} c_{\pi}(t) D^H(t) \chi(\det(t))^{-1} dt$$
where $c_{\pi}(t)=c_{\theta_{\pi}}(t)$.

\subsection{The middle model}
Choose $Q$ be the parabolic subgroup of $\GL_6(F)$ (resp. $\GL_3(D)$) of $(4,2)$ type (resp. $(2,1)$ type) which contains the lower minimal parabolic subgroup. Then the reduced model $(L,H_Q)$ we get is the following (Once again to make our notation simple, we will use $(G,\;H,\;U)$ instead of $(L,H_Q)$):  Let $G=GL_4(F)\times GL_2(F)$ and $P=MU$ be the parabolic subgroup of $G(F)$ with the Levi part $M$ isomorphic to $GL_2(F)\times GL_2(F)\times GL_2(F)$ (i.e. $P$ is the product of the second $GL_2(F)$ and the parabolic subgroup $P_{2,2}$ of the first $GL_4(F)$). The unipotent radical $U$ consists of elements of the form
\begin{equation}
u=u(X):=\begin{pmatrix} 1 & X & 0 \\ 0 & 1 & 0 \\ 0 & 0 & 1 \end{pmatrix},\; X\in M_2(F).
\end{equation}
The character $\xi$ on $U$ is defined to be $\xi(u(X))=\psi(\tr(X))$. Let $H_0=\GL_2(F)$ diagonally embeded into $M$. For a given irreducible representation $\pi$ of $G$, assume the central character $\omega_{\pi}$ equals to $\chi^2$ on $Z_H(F)$ for some character $\chi$ of $F^{\times}$. $\chi$ will induce a one-dimensional representation $\omega$ of $H_0$. Combining $\xi$ and $\omega$, we have a one-dimensional representation $\omega\otimes \xi$ of $H:=H_0\ltimes U$. Let
\begin{equation}
m(\pi)=\dim\Hom_{H(F)} (\pi,\omega\otimes \xi).
\end{equation}
This model can be thought as the "middle one" between the Ginzburg-Rallis model and the trilinear model of $\GL_2$.

Similarly, for the quaternion algebra case, we can define the multiplicity $m(\pi_D)$. The following theorem is an analogy of Theorem \ref{main} and Theorem \ref{main 1} for this model.

\begin{thm}
\begin{enumerate}
For any tempered representation $\pi$ of $G$, the following hold.
\item Let $\pi_D$ be the Jacquet-Langlands correspondence of $\pi$ to $G_D$. Then
$$m(\pi)+m(\pi_D)=1.$$
\item If $F=\BR$ and the central character of $\pi$ is trivial on $Z_H(F)$, we have
\begin{eqnarray*}
m(\pi)=1 &\iff& \epsilon(1/2,\pi,\wedge^3)=1,\\
m(\pi)=0 &\iff& \epsilon(1/2,\pi,\wedge^3)=-1.
\end{eqnarray*}
\item If $F$ is p-adic and the central character of $\pi$ is trivial on $Z_H(F)$, we have
\begin{eqnarray*}
m(\pi)=1 &\iff& \epsilon(1/2,\pi,\wedge^3)=1,\\
m(\pi)=0 &\iff& \epsilon(1/2,\pi,\wedge^3)=-1.
\end{eqnarray*}
if $\pi$ is not a discrete series.
\end{enumerate}
\end{thm}

We can also prove a multiplicity formula for this model in the p-adic case. Let $(G,H)$ be either $(G,H)$ or $(G_D,H_D)$ defined as above, and $F$ be a p-adic field. Let $\CT$ be the subset of subtorus $T$ in $H_0$ defined in Section 7.2. Let $\theta$ be a quasi-character on $Z_G(F)\backslash G(F)$ with central character $\eta=\chi^2$, and $T\in \CT$. If $T=\{1\}$, then we are in the split case. Since there is a unique regular nilpotent orbit in $\Fg(F)$, we let $c_\theta(t)=c_{\theta,\CO_{reg}}(t)$. If $T=T_v$ for some $v\in F^{\times}/(F^{\times})^2, v\neq 1$, and $t\in T_v$ is a regular element, $G_t=GL_2(F_v)\times GL_1(F_v)$. Let $\CO=\CO_1\times \CO_2$ where $\CO_1$ is the unique regular nilpotent orbit in $\Fg \Fl_2(F_v)$ and $\CO_2=\{0\}$ is the unique nilpotent orbit in $\Fg \Fl_1(F_v)$, define $c_\theta(t)=c_{\theta,\CO}(t)$. Note that $\Fg \Fl_1(F_v)$ is abelian, the only nilpotent element is zero, the germ expansion is just evaluation. Then for all irreducible tempered representation $\pi$ of $G(F)$ with central character $\eta$, we have
$$m(\pi)=m_{geom}(\pi):=\sum_{T\in \CT} \mid W(H,T)\mid^{-1} \nu(T) \int_{Z_G(F)\backslash T(F)} c_{\pi}(t) D^H(t) \chi(\det(t))^{-1} dt$$
where $c_{\pi}(t)=c_{\theta_{\pi}}(t)$.



\begin{thebibliography}{dihuajiang}
\bibitem[Ar75]{Ar75}
J. Arthur,
{\it A theorem on the Schwartz space of a reductive Lie group.} Proc. Nat. Acad. Sci. U.S.A. 72 (1975), no. 12, 4718-4719

\bibitem[Ar81]{Ar81}
J. Arthur,
{\it The trace formula in invariant form.} Ann. of Math. (2) 114 (1981), 1-74


\bibitem[Ar89]{Ar89}
J. Arthur,
{\it Intertwining operators and residues I. Weighted characters,} J. Funct. Analysis 84 (1989), 19-84.

\bibitem[Ar91]{Ar91}
J. Arthur,
{\it A local trace forumula.} Publ. Math. Inst. Hautes \'Etudes Sci. 73(1991), 5-96

\bibitem[B12]{B12}
Beuzart-Plessis, R.,
{\it La conjecture locale de Gross-Prasad pour les repr¨¦sentations temp\'er\'ees des groupes unitaires.} Prepublication, 2012.

\bibitem[B15]{B15}
Beuzart-Plessis, R.,
{\it A local trace formula for the Gan-Gross-Prasad conjecture for unitary groups: the archimedean case.} Preprint, 2015.

\bibitem[Ber88]{Ber88}
J. Bernstein,
{\it On the support of Plancherel measure.} Jour. of Geom. and Physics 5, No.4 1988.

\bibitem[BDK]{BDK}
J. Bernstein, P. Deligne, D. Kazhdan,
{\it Trace Paley-Wiener theorem for reductive p-adic groups.} J. d'Analyse Math. 47 (1986), 441-472.

\bibitem[BK14]{BK14}
J. Bernstein, B. Krotz,
{\it Smooth Fr\'echet globalizations of Harish-Chandra modules.} Israel J. Math. 199 (2014), no. 1, 45-111.


\bibitem[DKV84]{DKV84}
P. Deligne, D. Kazhdan, M.-F. Vign\'eras,
{\it Repr\'esentations des alg\`ebres centrales simples p-adiques.}  Repr\'esentations des groupes r\'eductifs sur un corps local, Travaux en Cours, Paris(1984): Hermann, pp. 33¨C117

\bibitem[GGP12]{GGP12}
Gan, Wee Teck, Gross, Benedict H., Prasad, Dipendra,
{\it Symplectic local root numbers, central critical $L$-values, and restriction problems in the representation theory of classical groups}.
Sur les conjectures de Gross et Prasad. I. Ast\'erisque No. 346 (2012), 1--109.


\bibitem[GR00]{GR00}
D. Ginzburg, S. Rallis,
{\it The exterior cube L-function for GL(6).} Compositio Math. 123(2000), no. 3, 243-272

\bibitem[J08]{J08}
Dihua Jiang,
{\it Residues of Eisenstein series and related problems.} Eisenstein Series and Applications, Progress in Math, 258, 187¨C204, 2008

\bibitem[JS90]{JS90}
H. Jacquet, J. Shalika,
{\it Rakin-Selberg convolutions: Archimedean theory.} In Festschrift in honor of I.I. Piatetski-Shapiro on the occasion of his sixtieth birthday, Part I, Israel Math. Conf. Proc., vol. 2, Weizmann, 1990, 125-207.

\bibitem[JSZ11]{JSZ11}
Dihua Jiang, Bingyong Sun, Chengbo Zhu,
{\it Uniqueness of the Ginzburg-Rallis models: the archimedean case.} Trans. Amer. Math. Soc. 363 (2011), 2763--2802.

\bibitem[K05]{K05}
R. E. Kottwitz,
{\it Harmonic analysis on reductive p-adic groups and Lie algebras.} in "Harmonic analysis, the trace formula, and Shimura varieties", 393522, Clay Math. Proc. 4, Amer. Math. Soc., Providence, RI, 2005

\bibitem[KKSS]{KKSS}
F. Knop, B. Krotz, E. Sayag, H. Schlichtkrull,
{\it Simple compactifications and polar decomposition of homogeneous real spherical spaces.} prepublication 2014


\bibitem[Knop95]{Knop95}
Friedrich Knop,
{\it On the set of orbits for a Borel subgroup.} Comment. Math. Helv., 70(2):285-309, 1995



\bibitem[L01]{L01}
Loke, Hung Yean
{\it Trilinear forms of $\Fg \Fl_2$.} Pacific J. Math. 197 (2001), no. 1, 119¨C144.

\bibitem[N06]{N06}
C. Nien,
{\it Models of representations of general linear groups over p-adic fields.} PHD Thesis, University of Minnesota, 2006.

\bibitem[P90]{P90}
D. Prosad,
{\it Trilinear forms for representations of GL(2) and local (epsilon)-factors.} Compositio Math. 75 (1990), 1-46

\bibitem[Rod81]{Rod81}
F.Rodier,
{\it Mod\`ele de Whittaker et caract\`eres de repr\'esentations.} Noncommutative harmonic analysis, Lecture Notes in Mathematics, vol.466, eds J. Carmona, J. Dixmier and M. Vergne(Springer, Berlin, 1981), 151-171.



\bibitem[SV12]{SV12}
Y. Sakellaridis, A. Venkatesh,
{\it Periods and harmonic analysis on spherical varieties.} Prepublication 2012

\bibitem[W03]{W03}
J.-L. Waldspurger,
{\it La formule de Plancherel pour les groupes p-adiques(d'apr\`es Harish-Chandra)}, J. Inst. Math. Jussieu 2 (2003), no.2, 235-333


\bibitem[W10]{W10}
J.-L. Waldspurger,
{\it Une formule int\'egrale reli\'ee \`a la conjecture locale de Gross-Prasad.} Compos. Math. 146(2010), no.5, 1180-1290.

\bibitem[W12]{W12}
J.-L. Waldspurger,
{\it Une formule int\'egrale reli\'ee \`a la conjecture locale de Gross-Prasad, 2e partie : extension aux repr\'esentations temp\'er\'ees.}  in "Sur les conjectures de Gross et Prasad. I" Ast\'erisque No. 346 (2012), 171-312

\bibitem[Wan15]{Wan15}
Chen Wan,
{\it A local relative trace formula for the Ginzburg-Rallis model: the geometric side.} Submited, 2015.


\bibitem[Wan17]{Wan17}
Chen Wan,
{\it The Ginzburg-Rallis model.} PhD Thesis, 2017.
\end{thebibliography}
\end{document}